\documentclass[english,a4paper,10pt,reqno,oneside]{amsart}
\usepackage{amsmath}
\usepackage{amsfonts}
\usepackage{accents}
\usepackage{mathtools}
\usepackage{mathrsfs}
\usepackage{url}
\usepackage{geometry}
\usepackage{amsthm}
\usepackage{microtype}
\usepackage{amssymb}
\usepackage[english]{babel}
\usepackage[utf8]{inputenc}
\usepackage[T1]{fontenc}
\usepackage{lmodern}
\usepackage{csquotes}
\usepackage{enumitem}
\usepackage{esint}
\usepackage{bbm}
\usepackage[hidelinks]{hyperref}
\usepackage{verbatim}
\usepackage[backend=biber,style=numeric,sortcites=true,isbn=false,url=false,doi=false,maxnames=5]{biblatex}
\usepackage{subcaption}
\usepackage{caption}
\usepackage{float}

\numberwithin{equation}{section}

\AtEveryBibitem{
    \clearfield{note}
    \clearfield{issn}
    \clearfield{language}
    \clearlist{language}
}

\theoremstyle{plain} 
\newtheorem{thm}{Theorem}[section]

\theoremstyle{definition} 
\newtheorem{defi}[thm]{Definition} 

\theoremstyle{plain} 
\newtheorem{lemma}[thm]{Lemma} 

\theoremstyle{plain}

\theoremstyle{remark} 
\newtheorem{rem}[thm]{Remark}

\theoremstyle{plain}
\newtheorem{cor}[thm]{Corollary} 

\theoremstyle{definition} 
\newtheorem{ex}[thm]{Example}

\theoremstyle{remark}

\newcommand{\R}{\mathbb{R}}
\newcommand{\N}{\mathbb{N}}

\newcommand{\Tail}[3]{\Tailoperator(#1;#2;#3)}
\newcommand{\tail}[3]{\tailoperator(#1;#2;#3)}
\renewcommand{\L}{\mathcal{L}}

\newcommand{\loc}{\textup{loc}}

\newcommand{\indicator}{\mathbbm{1}}

\renewcommand{\d}{\mathrm{d}}
\newcommand{\dx}{\d x}
\newcommand{\dy}{\d y}
\newcommand{\dz}{\d z}

\newcommand{\F}{\mathcal{F}}
\newcommand{\B}{\mathcal{B}}
\newcommand{\M}{\mathcal{M}}
\newcommand{\C}{\mathbf{C}}

\newcommand{\bigset}[2]{\left\{#1\:\middle\vert\: #2\right\}}
\newcommand{\Rd}{\R^d}
\newcommand{\intd}{\int_{\Rd}}
\newcommand{\pv}{\text{p.v.}}

\renewcommand{\and}{\quad\text{and}\quad}
\renewcommand{\div}{\operatorname{div}}

\DeclareMathOperator{\dist}{dist}
\DeclareMathOperator{\Tailoperator}{Tail}
\DeclareMathOperator{\tailoperator}{tail}

\DeclareMathOperator*{\osc}{osc}

\begin{document}

\raggedbottom

\title[Hölder regularity for nonlocal equations]{Hölder regularity for nonlocal equations governed by measures and non-standard growth}

\author{Luke Schleef}
\address{Fakultät für Mathematik\\Universität Bielefeld\\Universitätsstraße 25\\33615 Bielefeld\\Germany}
\email{luke.schleef@uni-bielefeld.de}

\begin{abstract}
    On doubling metric measure spaces, we study nonlocal operators with non-standard $(p,q)$-Orlicz growth ($1<p\leq q<\infty$), where the interaction kernel is given by a general, non-translation-invariant measure. Under natural assumptions on the interaction measure --- namely symmetry, a suitable nonlocal Poincar\'{e} inequality, and a tail bound --- we prove that every weak solution to the corresponding homogeneous nonlocal equation admits a locally Hölder continuous representative. These regularity results are new even in the Euclidean setting.
\end{abstract}

\thanks{I would like to thank my advisor Moritz Kassmann for suggesting this problem and for interesting discussions.}
\subjclass[2020]{45K05, 35B45, 35D30, 46E36}
\keywords{Nonlocal operator, non-standard growth, Hölder regularity, doubling metric measure space}

\maketitle

\section{Introduction}

In this article, we study nonlocal equations of the form $\L u=0$, where the operator $\L$ is formally given by
\begin{align}\label{eqIntro:Operator}
    \L u(x)=\pv\int_X h\left(x,y,\frac{u(x)-u(y)}{d(x,y)^s}\right)\frac{\nu(x,\dy)}{d(x,y)^s}.
\end{align}
Here, $(X,d,\mu)$ is a doubling metric measure space, $s\in(0,1)$ is the differential order of the operator, and $h(x,y,t)\frac{t}{|t|}=-h(y,x,-t)\frac{t}{|t|}\asymp f(|t|)$ is comparable to the derivative $f$ of a Young function with $(p,q)$-growth. In particular, for any $t\geq 0$, it holds that
\begin{align*}
    1<p\leq\frac{f(t)t}{F(t)}\leq q<\infty, \quad\text{where}\quad F(t)=\int_0^t f(\tau)\d\tau.
\end{align*}
Typical examples include
\begin{align*}
    h(x,y,t)=a(x,y)|t|^{p-2}t+b(x,y)|t|^{q-2}t \and h(x,y,t)=a(x,y)|t|^{p-2}t\log(1+|t|)
\end{align*}
for some symmetric, uniformly bounded, measurable functions $a,b\asymp 1$ and some $1<p<q<\infty$. The interaction kernel of the operator $\L$ is given by a family of measures $\{\nu(x,\cdot)\}_{x\in X}$ that may be singular with respect to $\mu$ and shall satisfy the following natural assumptions. First, $\nu$ should be symmetric in the sense that, whenever $A,B\subseteq X$ are two separated measurable sets, the identity
\begin{align*}
    \int_A\int_B\nu(x,\dy)\mu(\dx)=\int_B\int_A\nu(x,\dy)\mu(\dx)
\end{align*}
holds true. Second, we assume that $\nu$ satisfies the nonlocal Poincar\'{e} inequality
\begin{align}\label{eqIntro:PoincareAssumption}
    \int_B|u-(u)_B|\d\mu\lesssim\int_B\int_B|u(x)-u(y)|\nu(x,\dy)\mu(\dx)
\end{align}
for any ball $B$ in $X$ and any measurable function $u\colon B\to\R$. The last assumption is an upper bound on the tail of the measure $\nu(x,\cdot)$, i.e.
\begin{align}\label{eqIntro:TailAssumption}
    \int_{X\setminus B_R(x)}\frac{\nu(x,\dy)}{d(x,y)^s}\lesssim R^{-s}
\end{align}
for any $R>0$ and $x\in X$. Specifically, let us emphasize that our assumptions on $\nu$ are completely independent of the growth function and its indices $p$ and $q$. In fact, we will give more general assumptions in Section \ref{sec:Assumptions} that do involve the lower index $p$, but not the exact structure of the growth function. For the sake of accessible presentation, we stick with the above assumptions during the introduction. The measure $\nu(x,\dy)=|x-y|^{-d}\dy$ on $\Rd$ is the kernel of the fractional $p$-Laplace operator $(-\Delta_p)^s$ and satisfies the above assumptions. However, our results apply to very general nonlocal operators.

\subsection*{Main Result}
Our main result is about a priori local Hölder continuity of weak solutions to the equation $\L u=0$ with $\L$ as in \eqref{eqIntro:Operator}.

\begin{thm}\label{thmIntro:Hoelder}
    Let $u\in L_\loc^\infty(\Omega)$ be a locally bounded local weak solution to the equation $\L u=0$ in some open set $\Omega$ in $X$. Then, under the above assumptions, $u$ is locally $\alpha$-Hölder continuous in $\Omega$ for some $\alpha\in(0,1)$ and moreover, for any ball $B_R(x_0)\Subset\Omega$, the Hölder estimate
    \begin{align}\label{eqIntro:HoelderEstimate}
        [u]_{C^{0,\alpha}(B_{R/4}(x_0))}\leq\frac{C}{R^\alpha}\left(\|u\|_{L^\infty(B_R(x_0))}+\Tail{u}{x_0}{R}\right)
    \end{align}
    holds true for some $C>0$. Here, the constants $\alpha$ and $C$ only depend on $s$, on the doubling constant of $\mu$, on the indices $p$ and $q$, on the constants in the comparability assumption $h(x,y,t)\frac{t}{|t|}\asymp f(|t|)$, and on the constants in the assumptions \eqref{eqIntro:PoincareAssumption} and \eqref{eqIntro:TailAssumption}.
\end{thm}

While special cases of Theorem \ref{thmIntro:Hoelder} have been studied previously, our result unifies and substantially extends several separate branches of regularity theory. Existing works establish interior Hölder regularity for non-standard growth equations with standard fractional kernels, linear equations driven by general measures --- both mainly within the Euclidean setting --- and linear operators on metric measure spaces, see the discussion about related results below. Theorem \ref{thmIntro:Hoelder} provides the first result to simultaneously accommodate non-standard Orlicz growth, general non-translation-invariant interaction measures, and abstract doubling metric measure spaces.

This generality introduces two major technical hurdles. A primary obstacle is to formulate structural conditions on the interaction measure and the growth function that are decoupled, yet sufficiently compatible to recover a Poincar\'{e} inequality of type \eqref{eqIntro:PoincareAssumption} that incorporates the growth function $F$. Furthermore, the absence of standard Euclidean tools complicates the higher integrability estimates required for the De Giorgi iteration. We overcome this by establishing a Sobolev--Poincar\'{e} inequality relying solely on the doubling property of the underlying metric measure space --- a core technical contribution of this paper. In addition, we improve the known energy estimates for weak solutions as a by-product, which allow us to assume a very general nonlocal Poincar\'{e} inequality.

\medskip

Let us remark that the local boundedness assumption in Theorem \ref{thmIntro:Hoelder} cannot be removed. An a priori local boundedness result as well as Harnack's inequality fail in general for singular kernels, see \cite{BS05}.

As one expects for nonlocal operators, the Hölder estimate \eqref{eqIntro:HoelderEstimate} involves a tail term that describes the behavior of the solution outside of the domain. However, in our setting it is not clear how to define this term appropriately. Recall that for the fractional $p$-Laplacian $(-\Delta_p)^s$ on $\Rd$, the tail term is given by
\begin{align}\label{eqIntro:TailFractionalPLaplace}
    \Tail{u}{x_0}{R}=\left(R^{sp}\int_{\Rd\setminus B_R(x_0)}\frac{|u(y)|^{p-1}}{|y-x_0|^{d+sp}}\dy\right)^\frac{1}{p-1}
\end{align}
as introduced in \cite{DKP16}. The adaptation to metric measure spaces is straightforward. The challenging parts are to generalize the tail to Orlicz growth and to an abstract measure $\nu(x,\dy)$. Both parts have been done independently in \cite{CKW22} and \cite{Che25}, respectively. Note that when considering a measure $\nu(x,\dy)$, which may highly oscillate with respect to $x$, it is not sufficient to just take into account the center $x_0$ of the ball as in \eqref{eqIntro:TailFractionalPLaplace}. Hence, for $u\colon X\to\R$, we define its tail in $B_R(x_0)$ as
\begin{align}\label{eqIntro:Tail}
    \Tail{u}{x_0}{R}=R^sf^{-1}\left(R^s\sup_{x\in B_{R/2}(x_0)}\int_{X\setminus B_R(x_0)}f\left(\frac{|u(y)|}{d(x,y)^s}\right)\frac{\nu(x,\dy)}{d(x,y)^s}\right).
\end{align}
The fact that the supremum in \eqref{eqIntro:Tail} is taken over $x\in B_{R/2}(x_0)$ is a choice. One could take any portion of $B_R(x_0)$ as long as $x$ and $y$ stay separated by a uniform portion of $R$. In fact, we will see that all of these choices lead to the same class of functions having finite tails. The precise definitions of the corresponding function spaces and of a weak solution will be given below.

\subsection*{Strategy of the proof}
The proof of Hölder regularity follows the De Giorgi--Nash--Moser theory. First, we will use specific test functions to prove energy estimates for weak supersolutions that will be used in the iterations. Afterwards, we will perform a De Giorgi iteration to derive a critical mass lemma. In order for the iteration to converge, we need an improvement of integrability of solutions, which is usually given by Sobolev inequalities. In contrast to the Euclidean case, Sobolev inequalities themselves are not available in general, but they come hand in hand with Poincar\'{e} inequalities. We show that the doubling property of the metric measure space is sufficient to self-improve the assumption \eqref{eqIntro:PoincareAssumption} on the interaction measure in two ways. First, \eqref{eqIntro:PoincareAssumption} implies a Poincar\'{e} inequality that takes into account the growth function $F$, and second, this Poincar\'{e} inequality self-improves to a Sobolev--Poincar\'{e} inequality. The resulting integrability gain supplies the exponent needed in the De Giorgi iteration. The critical mass lemma itself will lead to an expansion of positivity result by invoking a logarithmic estimate and \eqref{eqIntro:PoincareAssumption} once more. Finally, we show that we are able to control the appearing tails in a suitable way so that the expansion of positivity result yields Hölder regularity in the usual way.

\subsection*{Organization of the paper}
Our work is organized as follows. First, we introduce notation and collect some preliminary results about metric measure spaces and growth functions in Section \ref{sec:Pre}, where we also define the notion of weak solutions to the equation $\L u=0$, the notion of tail terms and agree on the precise assumptions for our setting. In Section \ref{sec:SobPoi}, we prove the necessary Sobolev--Poincar\'{e} inequality. Section \ref{sec:Energy} is devoted to the usual energy estimates for weak supersolutions, that are, the Caccioppoli inequality and the logarithmic estimate. The main part of this article is Section \ref{sec:EoP}, where we prove an expansion of positivity result for weak supersolutions, which is the key ingredient to prove a decay of oscillation and then Hölder regularity, which will be done in Section \ref{sec:Hoelder}.

\subsection*{Background and related results}
Let us first comment on known results in the Euclidean case $X=\Rd$. The literature about local equations of non-standard growth of the form
\begin{align*}
    \div(g(x,u,Du))=0
\end{align*}
is very rich. Especially, $C_\loc^{0,\alpha}$- and $C_\loc^{1,\alpha}$-regularity of weak solutions has been studied under various assumptions on $g$ by several authors, including \cite{Gio57,LU68,GG82,Lie91,MN91,Mar89,AM01,CM15,Ok20}.

In recent years, the interest in the study of nonlocal equations with non-standard growth has grown. These equations are described by the operator
\begin{align}\label{eqIntro:NonlocalNonStandardGrowth}
    \pv\intd h\left(x,y,\frac{u(x)-u(y)}{|x-y|^s}\right)\frac{\dy}{|x-y|^{d+s}}.
\end{align}
If $h(x,y,t)=|t|^{p-2}t$, then \eqref{eqIntro:NonlocalNonStandardGrowth} reduces to the fractional $p$-Laplace operator $(-\Delta_p)^s$, which has been extensively studied, see e.g.\ \cite{Kas09,DKP16,DKP14,Coz17,Now21,BLS18,GL24,BT25,GJS25} for some results on interior regularity. More generally, when $h$ has the form $h(x,y,t)=f(|t|)\frac{t}{|t|}a(x,y)$ for some symmetric, measurable kernel 
\begin{align*}
    0<\lambda\leq a(x,y)\leq\Lambda<\infty,    
\end{align*}
then \eqref{eqIntro:NonlocalNonStandardGrowth} boils down to general growth equations 
\begin{align*}
    \pv\intd f\left(\frac{|u(x)-u(y)|}{|x-y|^s}\right)\frac{u(x)-u(y)}{|u(x)-u(y)|}\frac{a(x,y)}{|x-y|^{d+s}}\dy=0.
\end{align*}
The regularity properties of those operators have been investigated in \cite{FSV22,BKO23,CKW22}. For general non-standard growth equations, where $h$ satisfies a comparability assumption as in this article, Hölder estimates and Harnack inequalities for non-standard growth equations were proven in \cite{CKW22,CKW23}. In addition, Hölder regularity for double phase ($h(x,y,t)=|t|^{p-2}t+a(x,y)|t|^{q-2}t$) and variable exponent equations ($h(x,y,t)=|t|^{p(x,y)-2}t$) is established in \cite{BOS22,DP19} and \cite{CK23,Ok23}, respectively. We emphasize that these two models are not covered in this article.

On the other hand, there has been considerable interest in investigating the preservation of Hölder regularity if the interaction kernel is not a density, but a possibly singular measure. In the work of \cite{DK20}, the authors considered linear equations of the form
\begin{align*}
    \pv\intd (u(x)-u(y))\mu(x,\dy)=0
\end{align*}
where $\mu(x,\cdot)$ is a Borel measure, the canonical example is $\mu(x,\dy)=|x-y|^{-d-2s}\dy$. They proved interior Hölder regularity for weak solutions when $\mu$ is uniformly comparable to a non-degenerate $2s$-stable measure. Afterwards, \cite{Che25} generalized the result of \cite{DK20} by providing general conditions on the measure $\mu$ that also lead to interior Hölder regularity. We adapted the idea of \cite{Che25} and formulated similar conditions on the interaction measure, which also allow for the investigation of nonlinear operators with non-standard growth. We refer to \cite{RS16,BKS19,CS20,DK20,IS20,Hep26} for further discussions about possible conditions on the interaction measure.

Finally, let us comment on existing results about regularity theory on metric measure spaces. If $(X,d,\mu)$ is a so-called Ahlfors-$Q$-regular space, that is, $\mu(B_r(x))\asymp r^Q$ for some dimension $Q>0$, then the fractional $p$-Laplacian on $X$ formally reads as
\begin{align*}
    (-\Delta_p)^su(x)=\pv\int_X\frac{|u(x)-u(y)|^{p-2}(u(x)-u(y))}{d(x,y)^{Q+sp}}\mu(\dy),
\end{align*}
similarly to the Euclidean one. More generally, one can consider nonlocal operators on abstract doubling metric measure spaces. The canonical equivalent of the fractional $p$-Laplace operator is given by
\begin{align*}
    \pv\int_X\frac{|u(x)-u(y)|^{p-2}(u(x)-u(y))}{d(x,y)^{sp}\left(\mu(B_{d(x,y)}(x))+\mu(B_{d(x,y)}(y))\right)}\mu(\dy).
\end{align*}
For $p=2$, instead of Hölder regularity one often aims for heat kernel estimates via a probabilistic approach. The authors of \cite{BL02} studied Markov chains on the lattice $\mathbb{Z}^d$ that are comparable to a symmetric $2s$-stable process. They obtained two-sided heat kernel estimates for those Markov chains. This was generalized in \cite{CK03} to so-called $d$-sets, a special case of Ahlfors-$Q$-regular metric measure spaces. Further discussions about heat kernel estimates for Dirichlet forms on metric measure spaces can be found in \cite{BGK09,GH08,GHL14,CKKW18} including probabilistic and analytic methods. Robust versions of Hölder regularity and Harnack's inequality were established in \cite{Cho24}. For nonlinear equations, \cite{KS01} adapted the De Giorgi--Nash--Moser theory to metric measure spaces and established Hölder regularity for local $p$-Laplace equations on metric measure spaces satisfying a Poincar\'{e} inequality. In the nonlocal case, the recent papers \cite{CKKSS25,CGKS26} established regularity for the fractional $p$-Laplacian on metric measure spaces using a generalization of the extension method \cite{CS07}. To the best of our knowledge, there is no result on a priori Hölder regularity for fractional $p$-Laplace equations on arbitrary doubling metric measure spaces for $p\neq 2$ using the De Giorgi--Nash--Moser theory, so this work also contributes to the regularity theory on metric measure spaces, even for the fractional $p$-Laplacian.

\section{Preliminaries}\label{sec:Pre}

First, we agree on some notation and define important objects in the theory of metric measure spaces and Young functions. Let us stress that in the proofs we use constants $c$ and $C$, which may change from line to line, but only depend on the data as stated in the corresponding statement. More precisely, we use parentheses to express dependencies, i.e.\ $C=C(\lambda)$ is a constant that only depends on the parameter $\lambda$ and nothing else. Specific constants will be denoted $\overline{c},C_1,C_2$, for instance.

For a real number $a\in\R$, we define the positive part of $a$ by $a_+=\max\{a,0\}$ and its negative part by $a_-=\max\{-a,0\}$.

\subsection{Metric measure spaces}

Given a metric space $(X,d)$, we equip it with its Borel $\sigma$-algebra $\B(X)$, generated by all open sets in $X$ with respect to the metric $d$. For any $x\in X$ and $r>0$, we denote the open ball
\begin{align*}
    B_r(x)=\{y\in X\mid d(x,y)<r\}.
\end{align*}
With a slight abuse of notation, we may write $\lambda B$ for a scaled ball, i.e.\ if $B=B_r(x)$, then $\lambda B=B_{\lambda r}(x)$.

If not stated differently any $L^p$-space in this work, $p\in[1,\infty]$, is taken with respect to the measure $\mu$. Moreover, for any Borel measure $\mu$ on $(X,d)$, we define the mean integral and the mean value as follows. Given a Borel set $A\subseteq X$ with $\mu(A)\in(0,\infty)$ and a function $u\in L^1(A)$, we define
\begin{align*}
    (u)_A\coloneqq\fint_A u\,\d\mu\coloneqq\frac{1}{\mu(A)}\int_A u\,\d\mu.
\end{align*}

\begin{defi}
    \begin{enumerate}[label=(\roman*)]
        \item A triple $(X,d,\mu)$ is called a metric measure space if $(X,d)$ is a metric space and $\mu$ is a Borel measure on $(X,d)$.
        \item We say that $\mu$ has the doubling property (or $\mu$ is a doubling measure) if there exists a constant $\C>1$ such that, for any $x\in X$ and $R>0$,
        \begin{align*}
            \mu(B_{2R}(x))\leq\C\mu(B_R(x)).    
        \end{align*}
        \item Let $(X,d,\mu)$ be a complete metric measure space. If $\mu$ has the doubling property and satisfies
        \begin{align}\label{eqPre:BallsFinitePositiveMeasure}
            0<\mu(B_r(x))<\infty \enspace\text{for any}\  r>0,\ x\in X,
        \end{align}
        then we say that $(X,d,\mu)$ is a doubling metric measure space.
    \end{enumerate}
\end{defi}

Let us collect some classical auxiliary results about doubling metric measure spaces.

\begin{lemma}[{\cite[Lemma 8.1.13]{HKST15}}] \label{lemPre:LowerMassBound}
    Let $(X,d,\mu)$ be a doubling metric measure space. Then there exist some $Q\geq 1$ and $c>0$, depending only on $\C$ such that, for any $x,y\in X$ and $0<r\leq R$ with $y\in B_R(x)$, it holds
    \begin{align*}
        \frac{\mu(B_r(y))}{\mu(B_R(x))}\geq c\left(\frac{r}{R}\right)^Q.
    \end{align*}
\end{lemma}

The optimal \emph{doubling dimension} $Q$ might be smaller than $1$. However, if that is the case, Lemma \ref{lemPre:LowerMassBound} remains true for $Q=1$. We require $Q\geq 1$ in order to prove a Sobolev--Poincar\'{e} inequality later on. The optimal value of $Q$ is not important.

\begin{lemma}[{\cite[Lemma 4.1.14]{HKST15}}] \label{lemPre:Properness}
    Let $(X,d,\mu)$ be a doubling metric measure space. Then $X$ is proper, i.e.\ any closed ball $\overline{B_R(x)}$ is compact.
\end{lemma}

\begin{lemma}[Lebesgue's differentiation theorem, {\cite[Chapter 3.4]{HKST15}}] \label{lemPre:LebesgueDifferentiation}
    Let $(X,d,\mu)$ be a doubling metric measure space. Then, for any $u\in L_\loc^1(X)$, it holds that
    \begin{align*}
        u(x)=\lim_{r\to 0}\fint_{B_r(x)}u(y)\mu(\dy)\enspace\text{for $\mu$-almost every}\ x\in X.
    \end{align*}
\end{lemma}

The next lemma is often called $5B$-covering lemma.

\begin{lemma}\label{lemPre:5BCovering}
    Let $(X,d,\mu)$ be a doubling metric measure space. Let $\F$ be a family of balls with uniformly bounded radii. Then there exists a countable, pairwise disjoint subfamily $\mathcal{G}\subseteq\F$ such that
    \begin{align}\label{eqPre:5BCovering}
        \bigcup_{B\in\F}B\subseteq\bigcup_{B\in\mathcal{G}}5B.
    \end{align}
\end{lemma}

The existence of a pairwise disjoint subfamily $\mathcal{G}$ with \eqref{eqPre:5BCovering} follows from Zorn's lemma and is a general result on any metric space $(X,d)$, see \cite[Theorem 1.2]{Hei01}. The facts that $\mu$ is $\sigma$-finite and has full support, which follow from \eqref{eqPre:BallsFinitePositiveMeasure}, ensure that $X$ is separable. Thus, there cannot exist an uncountable pairwise disjoint family of balls in $X$. In fact, the doubling property of $\mu$ is not needed for this lemma.

\medskip

Finally, we recall that the space of Hölder continuous functions is denoted
\begin{align*}
    C^{0,\alpha}(A)=\bigset{u\colon A\to\R}{\|u\|_{C^{0,\alpha}(A)}\coloneqq\sup_A|u|+[u]_{C^{0,\alpha}(A)}<\infty}
\end{align*}
for $\alpha\in(0,1)$ and any subset $A\subseteq X$, where
\begin{align*}
    [u]_{C^{0,\alpha}(A)}\coloneqq\sup_{\substack{x,y\in A\\x\neq y}}\frac{|u(x)-u(y)|}{d(x,y)^\alpha}.
\end{align*}
If $\alpha=1$, then $C^{0,1}(A)$ is the space of Lipschitz functions, defined in the same way. Moreover, for an open set $\Omega\subseteq X$, the space $C_\loc^{0,\alpha}(\Omega)$ is defined as the set of all functions $u\colon\Omega\to\R$ such that $u\vert_K\in C^{0,\alpha}(K)$ for any compact set $K\subseteq\Omega$.

\subsection{Young functions}

\begin{defi}
    Let $1\leq p\leq q\leq\infty$. A function $F\colon[0,\infty)\to[0,\infty)$ is called a \emph{Young function} with $(p,q)$-growth if it is differentiable, increasing and convex, and satisfies $F(0)=0$ and the growth condition
    \begin{align}\label{eqPre:GrowthAssumption}
        p\leq\frac{f(t)t}{F(t)}\leq q
    \end{align}
    for any $t>0$, where $f=F^\prime$ denotes the derivative of $F$.
\end{defi}

We collect a few properties of Young functions in the following lemma. In this article, we will use those properties frequently without referring to the corresponding lemma.

\begin{lemma}
    Let $F$ be a Young function with $(p,q)$-growth for some $1\leq p\leq q<\infty$. Then the following statements hold true.
    \begin{enumerate}[label=(\alph*)]
        \item The function $t\mapsto\frac{F(t)}{t^p}$ is non-decreasing on $(0,\infty)$.
        \item For any $a\in[0,1]$ and $b\geq 0$,
        \begin{align*}
            a^qF(b)\leq F(ab)\leq a^pF(b).    
        \end{align*}
        \item For any $a\in[1,\infty)$ and $b\geq 0$,
        \begin{align*}
            a^pF(b)\leq F(ab)\leq a^qF(b).    
        \end{align*}
        \item For any $a,b\geq 0$,
        \begin{align*}
            \frac{1}{2}(F(a)+F(b))\leq F(a+b)\leq 2^{q-1}(F(a)+F(b)).
        \end{align*}
    \end{enumerate}
\end{lemma}

Similar properties also hold true for the derivative of any Young function.

\begin{lemma}
    Let $F$ be a Young function with $(p,q)$-growth for some $1\leq p\leq q<\infty$. Denote its derivative by $f=F^\prime$. Then the following statements hold true.
    \begin{enumerate}[label=(\alph*)]
        \item The function $t\mapsto\frac{f(t)}{t^{p-1}}$ is almost non-decreasing on $(0,\infty)$, i.e.\ it is true for any $0<t\leq\tau$ that
        \begin{align*}
            \frac{f(t)}{t^{p-1}}\leq\frac{q}{p}\frac{f(\tau)}{\tau^{p-1}}.
        \end{align*}
        \item For any $a\in[0,1]$ and $b\geq 0$,
        \begin{align*}
            \frac{p}{q}a^{q-1}f(b)\leq f(ab)\leq \frac{q}{p}a^{p-1}f(b).    
        \end{align*}
        \item For any $a\in[1,\infty)$ and $b\geq 0$,
        \begin{align*}
            \frac{p}{q}a^{p-1}f(b)\leq f(ab)\leq \frac{q}{p}a^{q-1}f(b).    
        \end{align*}
        \item For any $a,b\geq 0$,
        \begin{align*}
            \frac{1}{2}(f(a)+f(b))\leq f(a+b)\leq 2^{q-1}\frac{q}{p}(f(a)+f(b)).
        \end{align*}
    \end{enumerate}
\end{lemma}

All of the above properties are well known. We refer to \cite{HH19,CKW22} for various proofs.

\begin{defi}
    A function $F$ is called an \emph{$\mathcal{N}$-function} with $(p,q)$-growth if $F$ is a Young function with $(p,q)$-growth with $p>1$ and $q<\infty$, and moreover,
    \begin{align*}
        \lim_{t\to 0}\frac{F(t)}{t}=0 \and \lim_{t\to\infty}\frac{F(t)}{t}=\infty.
    \end{align*}
\end{defi}

\begin{rem}
    Later on, we will often be confronted with the term $F^\prime(|t|)\frac{t}{|t|}$, which originally appeared as the derivative of the function $F(|t|)$ and is defined for $t\in\R\setminus\{0\}$. However, when $F$ is an $\mathcal{N}$-function, we have
    \begin{align*}
        \lim_{t\to 0^+}F^\prime(|t|)\frac{t}{|t|}=\lim_{t\to 0^+}F^\prime(t)\begin{cases}
            \leq q\lim_{t\to 0^+}\frac{F(t)}{t}=0, \\
            \geq p\lim_{t\to 0^+}\frac{F(t)}{t}=0
        \end{cases}
    \end{align*}
    and similarly,
    \begin{align*}
        \lim_{t\to 0^-}F^\prime(|t|)\frac{t}{|t|}=-\lim_{t\to 0^-}F^\prime(-t)=0.
    \end{align*}
    Therefore, the convention $F^\prime(|t|)\frac{t}{|t|}\vert_{t=0}=0$ is justified. In fact, we will simply use the convention that $\frac{t}{|t|}\vert_{t=0}=0$ throughout this work.
\end{rem}

We will need the following useful corollary that will allow for a Jensen-type inequality below.

\begin{cor}\label{corPre:ConvexityComposition}
    Let $F$ be an $\mathcal{N}$-function with $(p,q)$-growth. Then the function $t\mapsto F(t^{1/p})$ is equivalent to a convex function on $[0,\infty)$, that is, there exists $G\colon[0,\infty)\to[0,\infty)$ and a constant $C=C(p,q)>0$ such that, for any $t\geq 0$, it holds that
    \begin{align}\label{eqPre:ConvexityCompositionEquivalence}
        G(t)\leq F(t^{1/p})\leq CG(t).
    \end{align}
\end{cor}

\begin{proof}
    Define $\varphi(t)=F(t^{1/p})$. Recall that $t\mapsto F(t)/t^p$ is non-decreasing. Setting $\tau=t^p$, we see that
    \begin{align}\label{eqPre:ConvexityCompositionMonotonicity}
        \tau\mapsto\frac{\varphi(\tau)}{\tau}
    \end{align}
    is also non-decreasing. Moreover,
    \begin{align}\label{eqPre:ConvexityCompositionIndicesPhi}
        \frac{\varphi^\prime(\tau)\tau}{\varphi(\tau)}=\frac{\frac{1}pF^\prime(\tau^{1/p})\tau^{1/p}}{F(\tau^{1/p})}\geq 1
    \end{align}
    by \eqref{eqPre:GrowthAssumption}. Let us claim that the function
    \begin{align*}
        G(t)=\int_0^t\frac{\varphi(\tau)}{\tau}\d\tau
    \end{align*}
    is convex and equivalent to $\varphi$. The convexity follows from
    \begin{align*}
        G^{\prime\prime}(t)=\frac{\varphi^\prime(t)t-\varphi(t)}{t^2}\geq 0,
    \end{align*}
    where we applied \eqref{eqPre:ConvexityCompositionIndicesPhi}. Now, let us prove \eqref{eqPre:ConvexityCompositionEquivalence}. Using once more \eqref{eqPre:ConvexityCompositionIndicesPhi} and the fundamental theorem of calculus, we see that
    \begin{align*}
        G(t)=\int_0^t\frac{\varphi(\tau)}{\tau}\d\tau\leq\int_0^t\varphi^\prime(\tau)\d\tau=\varphi(t),
    \end{align*}
    that is, the lower bound in \eqref{eqPre:ConvexityCompositionEquivalence}. For the upper bound, we use the monotonicity of \eqref{eqPre:ConvexityCompositionMonotonicity} to write
    \begin{align*}
        G(t)\geq\int_{t/2}^t\frac{\varphi(\tau)}{\tau}\d\tau\geq\left(t-\frac{t}{2}\right)\frac{\varphi(t/2)}{t/2}=\varphi(t/2)=F\left(\frac{t^{1/p}}{2^{1/p}}\right)\geq 2^{-q/p}F(t^{1/p})\geq c\varphi(t).
    \end{align*}
\end{proof}

\begin{defi}
    Let $1<p\leq q<\infty$ and let $F$ be an $\mathcal{N}$-function with $(p,q)$-growth. We define the \emph{Legendre transform} $F^\ast$ of $F$ by
    \begin{align*}
        F^\ast(t)=\sup_{\tau\geq 0}(t\tau-F(\tau)).
    \end{align*}
\end{defi}

\begin{lemma}[{\cite[Proposition 2.4.9]{HH19}}]
    For any $\mathcal{N}$-function $F$ with $(p,q)$-growth, its Legendre transform $F^\ast$ is also an $\mathcal{N}$-function with $(\frac{q}{q-1},\frac{p}{p-1})$-growth.
\end{lemma}

Later, we will use the following version of Young's inequality.

\begin{lemma}\label{lemPre:Young}
    Let $F$ be an $\mathcal{N}$-function with $(p,q)$-growth and denote its derivative by $f$. Then, for any $t,\tau\geq 0$ and $\varepsilon,\delta\in(0,1]$, it holds that
    \begin{align*}
        f(t)\delta\tau\leq\varepsilon(q-1) F(t)\delta^\frac{q}{q-1}+\varepsilon^{-q}F(\tau).
    \end{align*}
\end{lemma}

\begin{proof}
    The so-called \emph{Fenchel's inequality} (\cite[Section 2.4]{HH19}) for $\mathcal{N}$-functions reads as
    \begin{align*}
        t\tau\leq F^\ast(t)+F(\tau).
    \end{align*}
    If we apply this inequality for $\varepsilon f(t)\delta$ and $\tau/\varepsilon$, then we arrive at
    \begin{align*}
        f(t)\delta\tau=\varepsilon f(t)\delta\frac{\tau}{\varepsilon}\leq F^\ast(\varepsilon f(t)\delta)+F\left(\frac{\tau}{\varepsilon}\right).
    \end{align*}
    Since $F^\ast$ is also a Young function with indices $q/(q-1)$ and $p/(p-1)$, and $\varepsilon,\delta\leq 1$, we have 
    \begin{align*}
        F^\ast(\varepsilon f(t)\delta)\leq\varepsilon F^\ast(f(t)\delta)\leq\varepsilon F^\ast(f(t))\delta^\frac{q}{q-1},
    \end{align*}
    whereas $F(\tau/\varepsilon)\leq\varepsilon^{-q}F(\tau)$. Hence, it is left to notice that it can be seen from the definition of $F^\ast$ that
    \begin{align*}
        F^\ast(f(t))=f(t)t-F(t)\leq (q-1)F(t)
    \end{align*}
    due to \eqref{eqPre:GrowthAssumption}.
\end{proof}

Next, we recall the definition of Orlicz spaces. Fix a metric measure space $(X,d,\mu)$, let $F$ be a Young function and $A\subseteq X$ be some Borel set. Then the Orlicz space $L^F(A)$ is defined as
\begin{align*}
    L^F(A)=\bigset{u\colon A\to\R\ \text{Borel measurable}}{\int_A F(|u|)\d\mu<\infty}.
\end{align*}
The norm in $L^F(A)$ is given by the Luxemburg norm
\begin{align*}
    \|u\|_{L^F(A)}=\inf\bigset{\lambda>0}{\int_A F\left(\frac{|u|}{\lambda}\right)\d\mu\leq 1}.
\end{align*}
For an open set $\Omega\subseteq X$, we also define the space $L_\loc^F(\Omega)$ as the set of all Borel measurable functions $u\colon\Omega\to\R$ such that, for any compact subset $K\subseteq\Omega$, $u\vert_K$ belongs to $L^F(K)$.

\subsection{Function spaces and general nonlocal operators}

In this and the next subsection, we fix the setting in which we will work for the rest of this article.

Let $1<p\leq q<\infty$ be two indices, let $F\colon [0,\infty)\to[0,\infty)$ be an $\mathcal{N}$-function with $(p,q)$-growth and denote its derivative by $f$. Fix a doubling metric measure space $(X,d,\mu)$, where we denote the doubling constant of $\mu$ by $\C$. Let $h\colon X\times X\times\R\to\R$ be a jointly Borel measurable function satisfying $h(x,y,t)=-h(y,x,-t)$ and moreover, we assume that there exists $\kappa\geq 1$ such that
\begin{align}\label{eqPre:Comparability}
    \kappa^{-1}f(|t|)\leq h(x,y,t)\frac{t}{|t|}\leq\kappa f(|t|)
\end{align}
for any $t\in\R$ and $x,y\in X$. Note that, in particular, $h(x,y,t)$ and $t$ always have the same sign, $h(x,y,0)=0$, and
\begin{align}\label{eqPre:UpperBoundNonStandardGrowthFunction}
    |h(x,y,t)|\leq\kappa f(|t|).    
\end{align}

Moreover, we fix a family $\{\nu(x,\cdot)\}_{x\in X}$ of kernels $\nu(x,\cdot)\colon\B(X)\to[0,\infty]$, i.e.\ $\nu(x,\cdot)$ is a Borel measure on $X$ for any $x\in X$ and $\nu(\cdot,A)$ is Borel measurable for any $A\in\B(X)$. We will always assume that $\nu(x,\{x\})=0$ for $\mu$-almost all $x\in X$.

Before stating the structural assumptions on $\nu$, we introduce the formal operator and the associated function spaces. Their well-posedness under those assumptions will be discussed below.

Define the nonlocal operator $\L$ formally by
\begin{align*}
    \L u(x)&=\pv\int_X h\left(x,y,\frac{|u(x)-u(y)|}{d(x,y)^s}\right)\frac{\nu(x,\dy)}{d(x,y)^s} \\
    &\coloneqq\lim_{\varepsilon\to 0^+}\int_{X\setminus B_\varepsilon(x)}h\left(x,y,\frac{|u(x)-u(y)|}{d(x,y)^s}\right)\frac{\nu(x,\dy)}{d(x,y)^s}.
\end{align*}
The notion $\pv$ refers to the \emph{Cauchy principal value} of the integral. However, we do not want to evaluate the operator pointwise, so the principal value is not important for us. Instead, we want to introduce the notion of weak solutions. For that, we extend the notion of fractional Orlicz--Sobolev spaces in the following way. Recall that the Orlicz space $L^F$ has already been defined above.
\begin{defi}
    Let $\Omega\subseteq X$ be an open set. 
    \begin{enumerate}[label=(\alph*)]
        \item The \emph{fractional Orlicz--Sobolev space} $W^{s,F,\nu}(\Omega)$ is defined as
        \begin{align*}
            W^{s,F,\nu}(\Omega)=\bigset{u\in L^F(\Omega)}{\int_\Omega\int_\Omega F\left(\frac{|u(x)-u(y)|}{d(x,y)^s}\right)\nu(x,\dy)\mu(\dx)<\infty}.
        \end{align*}
        \item The space $W_\loc^{s,F,\nu}(\Omega)$ is the set of all functions $u\in L_\loc^F(\Omega)$ such that, for any compact set $K\subseteq\Omega$, there exists an open set $U$, satisfying $K\subseteq U\Subset\Omega$, such that $u\vert_U\in W^{s,F,\nu}(U)$.
        \item The space $W_0^{s,F,\nu}(\Omega)$ is defined as
        \begin{align*}
            W_0^{s,F,\nu}(\Omega)=\{u\in W^{s,F,\nu}(X) \mid u=0\ \text{$\mu$-almost everywhere in}\ X\setminus\Omega\}.
        \end{align*}
    \end{enumerate}
\end{defi}

The notation $U\Subset\Omega$ means that $\overline{U}$ is compact and $\overline{U}\subseteq\Omega$. For more information about classical fractional Orlicz--Sobolev spaces, we refer to \cite{FS19,ACPS21} and references therein.

When working with weak solutions in $\Omega$, it is not sufficient to require $u\in W^{s,F,\nu}(\Omega)$. We also need an appropriate condition on the tail of $u$, which takes care of the behavior of $u$ in the exterior of $\Omega$. We define the tail as follows.
\begin{defi}
    \begin{enumerate}
        \item For any Borel measurable function $u\colon X\to\R$ and for $x_0\in X$ and $R>0$, we define
        \begin{align}\label{eqPre:TailDefinition}
            \tail{u}{x_0}{R}=\sup_{x\in B_{R/2}(x_0)}\int_{X\setminus B_R(x_0)}f\left(\frac{|u(y)|}{d(x,y)^s}\right)\frac{\nu(x,\dy)}{d(x,y)^s}
        \end{align}
        and its tail in $X\setminus B_R(x_0)$ by
        \begin{align*}
            \Tail{u}{x_0}{R}=R^sf^{-1}(R^s\tail{u}{x_0}{R}).
        \end{align*}
        \item We define the tail space $L_{s,\nu}^f(X)$ as
        \begin{align*}
            L_{s,\nu}^{f}(X)&=\{u\in L_\loc^1(X) \mid \tail{u}{x_0}{R}<\infty \ \text{for all}\ x_0\in X,\ R>0\}
        \end{align*}
    \end{enumerate}
\end{defi}

Notice that the inverse function $f^{-1}$ of $f$ is well-defined, since the lower growth index of $F$ is larger than $1$. The supremum in \eqref{eqPre:TailDefinition} has to be understood as an essential supremum with respect to the measure $\mu$. Moreover, it is important in \eqref{eqPre:TailDefinition} that $x$ and $y$ are separated by a multiple of $R$, otherwise the singularity will blow up the integral. The choice $1/2$ of this multiple is artificial and the numerical value of $\tail{u}{x_0}{R}$ does depend on this choice. However, the following lemma shows that the tail space $L_{s,\nu}^f(X)$ is independent of that choice.

\begin{lemma}\label{lemPre:TailFiniteness}
    Let $u\in L_{s,\nu}^f(X)$. Then, for any $0<r<R$ and any $x_0\in X$, it holds that
    \begin{align*}
        \sup_{x\in B_r(x_0)}\int_{X\setminus B_R(x_0)}f\left(\frac{|u(y)|}{d(x,y)^s}\right)\frac{\nu(x,\dy)}{d(x,y)^s}<\infty.
    \end{align*}
\end{lemma}

\begin{proof}
    Let $\rho=R-r$. By Lemma \ref{lemPre:Properness}, we can find finitely many points $x_1,\ldots,x_n\in B_r(x_0)$ such that
    \begin{align*}
        \bigcup_{i=1}^n B_{\rho/2}(x_i)\supseteq B_r(x_0).
    \end{align*}
    Since $u\in L_{s,\nu}^f(X)$, we know that
    \begin{align*}
        T_i\coloneqq\tail{u}{x_i}{\rho}=\sup_{x\in B_{\rho/2}(x_i)}\int_{X\setminus B_\rho(x_i)}f\left(\frac{|u(y)|}{d(x,y)^s}\right)\frac{\nu(x,\dy)}{d(x,y)^s}<\infty.
    \end{align*} 
    Moreover, note that if $y\in X\setminus B_R(x_0)$, then 
    \begin{align*}
        d(y,x_i)\geq d(y,x_0)-d(x_0,x_i)\geq R-r=\rho
    \end{align*}
    so that
    \begin{align*}
        \sup_{x\in B_{\rho/2}(x_i)}\int_{X\setminus B_R(x_0)}f\left(\frac{|u(y)|}{d(x,y)^s}\right)\frac{\nu(x,\dy)}{d(x,y)^s}\leq\sup_{x\in B_{\rho/2}(x_i)}\int_{X\setminus B_\rho(x_i)}f\left(\frac{|u(y)|}{d(x,y)^s}\right)\frac{\nu(x,\dy)}{d(x,y)^s}=T_i
    \end{align*}
    holds true for all $i$. Therefore, we can take the maximum with respect to $i$, which yields
    \begin{align*}
        \sup_{x\in B_r(x_0)}\int_{X\setminus B_R(x_0)}f\left(\frac{|u(y)|}{d(x,y)^s}\right)\frac{\nu(x,\dy)}{d(x,y)^s}\leq\max\{T_1,\ldots,T_n\}<\infty.
    \end{align*}
\end{proof}

\begin{defi}\label{defPre:WeakSolution}
    Let $\Omega\subseteq X$ be an open set. We say that $u\in W^{s,F,\nu}(\Omega)\cap L_{s,\nu}^f(X)$ is a weak supersolution to $\L u=0$ in $\Omega$ if, for any $\varphi\in W_0^{s,F,\nu}(\Omega)$ with $\varphi\geq 0$,
    \begin{align}\label{eqPre:Test}
        \int_X\int_X h\left(x,y,\frac{u(x)-u(y)}{d(x,y)^s}\right)\frac{\varphi(x)-\varphi(y)}{d(x,y)^s}\nu(x,\dy)\mu(\dx)\geq 0.
    \end{align}
    Similarly, we say that $u$ is a weak subsolution to $\L u=0$ in $\Omega$ if the opposite inequality holds true for any $\varphi$ as above. If $u$ is both a weak subsolution and a weak supersolution to $\L u=0$ in $\Omega$, then $u$ is called a weak solution to $\L u=0$ in $\Omega$. Finally, a function $u\in W_\loc^{s,F,\nu}(\Omega)\cap L_{s,\nu}^f(X)$ is called a local weak solution to $\L u=0$ in $\Omega$ if $u$ is a weak solution to $\L u=0$ in $U$ for any open set $U\Subset\Omega$.
\end{defi}

\subsection{Assumptions on the interaction measure}\label{sec:Assumptions}

Let us introduce the assumptions on $\nu$ that we will assume throughout this work. For the presented proofs, these assumptions are very general and natural in view of our proofs. We will comment on the assumptions afterwards and will also derive useful consequences. The assumptions are as follows:
\begin{itemize}
    \item \emph{(Symmetry)} For any Borel sets $A,B\subseteq X$ with $\dist(A,B)>0$,
    \begin{align}\label{eqPre:SymmetryAssumption}
        \int_A\int_B\nu(x,\dy)\mu(\dx)=\int_B\int_A\nu(x,\dy)\mu(\dx).\tag{A1}
    \end{align}
    \item \emph{(Poincar\'{e} inequality)} There exist some $C_1>0$, $\lambda\geq 1$, $R_0>0$ and $t\in[0,s)$ such that, for any $0<R\leq R_0$, $x_0\in X$ and every Borel measurable function $u\colon B_{\lambda R}(x_0)\to\R$,
    \begin{align}\label{eqPre:PoincareAssumption}
        \quad\quad\enspace\int_{B_R(x_0)}|u-(u)_{B_R(x_0)}|^p\d\mu\leq C_1R^{tp}\int_{B_{\lambda R}(x_0)}\int_{B_{\lambda R}(x_0)}\frac{|u(x)-u(y)|^p}{d(x,y)^{tp}}\nu(x,\dy)\mu(\dx).\tag{A2}
    \end{align}
    \item There exists some $C_2>0$ such that, for any $0<R\leq R_0$ and $x\in X$,
    \begin{align}\label{eqPre:TailIntegrabilityAssumption}
        \int_{X\setminus B_R(x)}\frac{\nu(x,\dy)}{d(x,y)^{sp}}\leq\frac{C_2}{R^{sp}}.\tag{A3}
    \end{align}
\end{itemize}

\begin{rem}
    If both sides of \eqref{eqPre:SymmetryAssumption} are infinite, the identity is still considered to be true. Similarly, if the right hand side of \eqref{eqPre:PoincareAssumption} is infinite, the inequality remains true.
\end{rem}

Let us first remark that under assumptions \eqref{eqPre:SymmetryAssumption} and \eqref{eqPre:TailIntegrabilityAssumption}, one can show that Definition \ref{defPre:WeakSolution} is well-defined.

It is easy to see that the assumption \eqref{eqPre:TailIntegrabilityAssumption} is weaker than the one given in \eqref{eqIntro:TailAssumption}. In view of Remark \ref{remSobPoi:Remark1Poincare} below and since we are allowed to choose $t=0$ in \eqref{eqPre:PoincareAssumption}, our set of assumptions is indeed weaker than the one presented in the introduction. However, the advantage of the assumptions in the introduction is that they are completely independent of the growth function and of its growth indices.

For convenience, we provide a list of important examples of kernels $\nu$. Apart from the second one, no example is covered in the existing works \cite{BKO23,CKW22} about Orlicz growth equations.

\begin{ex}
    \begin{enumerate}[label=(\alph*)]
        \item The measure 
        \begin{align*}
            \nu(x,\dy)=\frac{\mu(\dy)}{\mu(B_{d(x,y)}(x))+\mu(B_{d(x,y)}(y))}
        \end{align*}
        is the canonical kernel on doubling metric measure spaces.
        \item On $\Rd$, the canonical kernel boils down to
        \begin{align}\label{eqPre:ExampleMeasure}
            \nu(x,\dy)=\frac{\dy}{|x-y|^d},
        \end{align}
        which, among others, induces the fractional $p$-Laplace operator.
        \item On the integer lattice $\mathbb{Z}^d$ with the discrete Euclidean metric and the counting measure $\mu$, the corresponding fractional $p$-Laplace operator can be similarly defined with the kernel
        \begin{align*}
            \nu(x,\dy)=\indicator_{x\neq y}\frac{\mu(\dy)}{|x-y|^d}.
        \end{align*}
        \item Let $\Omega\subseteq\Rd$ be an open domain with Lipschitz boundary and equip its closure $\overline{\Omega}$ with the restriction of the Euclidean metric and the Lebesgue measure. Then $\nu$ as in \eqref{eqPre:ExampleMeasure} induces, for instance, the regional $p$-Laplace operator, given by
        \begin{align*}
            (-\Delta_p)_\Omega^su(x)=\pv\int_\Omega\frac{|u(x)-u(y)|^{p-2}(u(x)-u(y))}{|x-y|^{d+sp}}\dy.
        \end{align*}
        \item For $i\in\{1,\ldots,d\}$ and $x\in\Rd$ let $e_i$ denote the $i$-th unit vector and
        \begin{align*}
            A_i(x)=\{x+\lambda e_i\mid\lambda\in\R\}\subseteq\Rd
        \end{align*}
        denote the $i$-th coordinate axis shifted towards $x$. Let $\mathcal{H}^1$ be the $1$-dimensional Hausdorff measure. Then 
        \begin{align*}
            \nu(x,\dy)=\sum_{i=1}^d\frac{\left.\mathcal{H}^1(\dy)\right\vert_{A_i(x)}}{|x-y|}
        \end{align*}
        can create the sum of $1$-dimensional Laplace operators
        \begin{align}\label{eqPre:SumOfOneDimLaplacians}
            \L=(-\partial_{x_1x_1})^s+\ldots+(-\partial_{x_dx_d})^s.
        \end{align}
        \item A discrete version of \eqref{eqPre:SumOfOneDimLaplacians} can be constructed by
        \begin{align*}
            \nu(x,\dy)=\sum_{i=1}^d\sum_{j=1}^\infty\left(\delta_{x+2^{-j}e_i}(\dy)+\delta_{x-2^{-j}e_i}(\dy)\right),
        \end{align*}
        where $\delta_{x_0}$ is the Dirac measure at $x_0$.
        \item Given $v\in\mathbb{S}^{d-1}$ and $\theta\in[0,1)$, define a double-cone
        \begin{align*}
            C=\bigset{h\in\Rd}{\left|\frac{h}{|h|}\cdot v\right|\geq\theta}
        \end{align*}
        and the measures
        \begin{align*}
            \nu_1(x,\dy)=\indicator_C(x-y)\frac{\dy}{|x-y|^d} \and \nu_2(x,\dy)=\indicator_{C\cap B_1(0)}(x-y)\frac{\dy}{|x-y|^d}
        \end{align*}
        on $\Rd$. Then both $\nu_1$ and $\nu_2$ induce admissible operators for this article, i.e.\ any weak solution is locally Hölder continuous.
    \end{enumerate}
\end{ex}

Note that, in all of the given examples, the order of fractional differentiability $s$ is not encoded in $\nu$ itself. The only connection to $s$ is that $\nu$ and $s$ have to be compatible in the sense of assumptions \eqref{eqPre:PoincareAssumption} and \eqref{eqPre:TailIntegrabilityAssumption}, which exclude no $s\in(0,1)$ for the provided examples. Since all of the above examples satisfy even \eqref{eqIntro:PoincareAssumption} and \eqref{eqIntro:TailAssumption}, they can be attached with any admissible function $h$ of non-standard growth in our setting.

Now, we will formulate our main result Theorem \ref{thmIntro:Hoelder} in full generality.

\begin{thm}\label{thmPre:Hoelder}
    Let $\Omega\subseteq X$ be an open set and $u\in W_\loc^{s,F,\nu}(\Omega)\cap L_{s,\nu}^f(X)\cap L_\loc^\infty(\Omega)$ be a locally bounded local weak solution to the equation $\L u=0$ in $\Omega$. Then $u\in C_\loc^{0,\alpha}(\Omega)$ for some $\alpha=\alpha(\C,s,t,p,q,\kappa,\lambda,C_1,C_2)\in(0,1)$. Furthermore, for any $x_0\in X$ and $0<R\leq R_0$ with $B_R(x_0)\Subset\Omega$, it holds
    \begin{align*}
        [u]_{C^{0,\alpha}(B_{R/4}(x_0))}\leq\frac{C}{R^\alpha}\left(\|u\|_{L^\infty(B_R(x_0))}+\Tail{u}{x_0}{R}\right)
    \end{align*}
    for some $C=C(\C,s,t,p,q,\kappa,\lambda,C_1,C_2)>0$.
\end{thm}

If the solution is globally bounded, we get the following result.

\begin{cor}\label{corPre:Hoelder}
    Let $\Omega\subseteq X$ be an open set and $u\in W_\loc^{s,F,\nu}(\Omega)\cap L^\infty(X)$ be a bounded local weak solution to $\L u=0$ in $\Omega$. Then $u\in C_\loc^{0,\alpha}(\Omega)$ for some $\alpha=\alpha(\C,s,t,p,q,\kappa,\lambda,C_1,C_2)\in(0,1)$. Furthermore, for any compact subset $K\subseteq\Omega$, it is true that
    \begin{align*}
        \|u\|_{C^{0,\alpha}(K)}\leq C\|u\|_{L^\infty(X)}
    \end{align*}
    for some $C=C(\C,s,t,p,q,\kappa,\lambda,C_1,C_2,\dist(K,\partial\Omega),R_0)>0$.
\end{cor}

Let us deduce some useful consequences of the assumptions \eqref{eqPre:SymmetryAssumption} and \eqref{eqPre:TailIntegrabilityAssumption}.

\begin{lemma}\label{lemPre:SymmetryFunctions}
    Let $A,B\subseteq X$ be two Borel sets with $\dist(A,B)>0$. Let $g\colon X\times X\to\R$ be a $(\B(X)\otimes\B(X))$-measurable function with
    \begin{align}\label{eqPre:SymmetryFunctionsIntegrability}
        \int_A\int_B|g(x,y)|\nu(x,\dy)\mu(\dx)<\infty.
    \end{align}
    Then
    \begin{align}\label{eqPre:SymmetryFunctions}
        \int_A\int_B g(x,y)\nu(x,\dy)\mu(\dx)=\int_B\int_A g(y,x)\nu(x,\dy)\mu(\dx).
    \end{align}
\end{lemma}

\begin{proof}
    Set $\rho=\dist(A,B)>0$. We proceed by measure-theoretic induction. First, we consider the case where $g=\indicator_D$ is a characteristic function. Define the set
    \begin{align*}
        \mathcal{P}=\{D\in\B(X)\otimes\B(X)\mid \eqref{eqPre:SymmetryFunctions}\ \text{holds true for}\ g=\indicator_D\}
    \end{align*}
    and two Borel measures $\gamma_1,\gamma_2$ on $X\times X$ by
    \begin{align*}
        \gamma_1(D)=\int_A\int_B\indicator_D(x,y)\nu(x,\dy)\mu(\dx) \and \gamma_2(D)=\int_B\int_A\indicator_D(y,x)\nu(x,\dy)\mu(\dx).
    \end{align*}
    Since $\mathcal{P}$ is closed under finite intersections, it forms a \emph{$\pi$-system}. Note that the set of rectangles
    \begin{align*}
        \mathcal{R}=\{E\times F\mid E,F\subseteq\B(X)\}
    \end{align*}
    is contained in $\mathcal{P}$ because
    \begin{align*}
        \int_A\int_B \indicator_{E\times F}(x,y)\nu(x,\dy)\mu(\dx)&=\int_{A\cap E}\int_{B\cap F}\nu(x,\dy)\mu(\dx) \\
        &=\int_{B\cap F}\int_{A\cap E}\nu(x,\dy)\mu(\dx)=\int_B\int_A \indicator_{E\times F}(y,x)\nu(x,\dy)\mu(\dx)
    \end{align*}
    by \eqref{eqPre:SymmetryAssumption}. Moreover, $\mathcal{R}$ generates the product $\sigma$-algebra of $X\times X$ and thus, the uniqueness of measure theorem (\cite[Proposition 1.15]{ADM11}) implies $\gamma_1=\gamma_2$ on $\B(X)\times\B(X)$, provided $\gamma_1$ is $\sigma$-finite. In fact, it follows from \eqref{eqPre:TailIntegrabilityAssumption} that
    \begin{align*}
        \gamma_1(B_n(x_0)\times B_n(x_0))\leq (2n)^{sp}\int_{B_n(x_0)}\int_{X\setminus B_\rho(x)}\frac{\nu(x,\dy)}{d(x,y)^{sp}}\mu(\dx)\leq C_2(2n)^{sp}\frac{\mu(B_n(x_0))}{\rho^{sp}}<\infty
    \end{align*}
    for an arbitrary $x_0\in X$ and any integer $n\geq 1$, which proves the $\sigma$-finiteness. In other words, \eqref{eqPre:SymmetryFunctions} is proven for characteristic functions. By linearity of the integral, \eqref{eqPre:SymmetryFunctions} also follows for simple functions of the form
    \begin{align*}
        g=\sum_{i=1}^N a_i\indicator_{D_i}, \quad N\in\N,\ a_i\geq 0,\ D_i\in\B(X),
    \end{align*}
    and monotone convergence theorem gives \eqref{eqPre:SymmetryFunctions} for any nonnegative $(\B(X)\otimes\B(X))$-measurable function. For a general measurable function $g$, we make use of \eqref{eqPre:SymmetryFunctionsIntegrability}. First we can apply the previous observation to $|g|$ to obtain that
    \begin{align*}
        \int_B\int_A|g(y,x)|\nu(x,\dy)\mu(\dx)<\infty
    \end{align*}
    as well and by applying \eqref{eqPre:SymmetryFunctions} to $g_+$ and $g_-$ separately, the result follows by subtracting both equations.
\end{proof}

\begin{lemma}\label{lemPre:distance_integrability_nu}
    For any $\alpha>0$, $x_0\in X$, $0<R\leq R_0$, and any $x\in B_R(x_0)$, it holds
    \begin{align*}
        \int_{B_R(x_0)}d(x,y)^\alpha\nu(x,\dy)\leq CR^\alpha
    \end{align*}
    for some $C=C(s,p,C_2,\alpha)>0$.
\end{lemma}

\begin{proof}
    For any $x\in B_R(x_0)$, we have
    \begin{align*}
        &\int_{B_R(x_0)} d(x,y)^\alpha\nu(x,\dy)\leq\int_{B_{2R}(x)}d(x,y)^{\alpha+sp}\frac{\nu(x,\dy)}{d(x,y)^{sp}} \\
        &\quad =\sum_{i=0}^\infty\int_{B_{2^{1-i}R}(x)\setminus B_{2^{-i}R}(x)}d(x,y)^{\alpha+sp}\frac{\nu(x,\dy)}{d(x,y)^{sp}}\leq\sum_{i=0}^\infty\left(2^{1-i}R\right)^{\alpha+sp}\int_{X\setminus B_{2^{-i}R}(x)}\frac{\nu(x,\dy)}{d(x,y)^{sp}} \\
        &\quad\leq\sum_{i=0}^\infty 2^{(1-i)(\alpha+sp)}R^{\alpha+sp}\frac{C_2}{(2^{-i}R)^{sp}}=2^{\alpha+sp}C_2R^\alpha\sum_{i=0}^\infty 2^{-i\alpha}\leq CR^\alpha
    \end{align*}
    by \eqref{eqPre:TailIntegrabilityAssumption}.
\end{proof}

\begin{lemma}\label{lemPre:SupEstimate}
    Let $A,B\subseteq X$ be two Borel sets with $\dist(A,B)>0$ and let $u\colon X\to\R$ be Borel measurable. Then $|u(y)|\leq\|u\|_{L^\infty(B)}$ for $\nu(x,\cdot)$-almost all $y\in B$ for $\mu$-almost all $x\in A$. 
\end{lemma}

\begin{proof}
    Set $N=\{y\in B\mid |u(y)|>\|u\|_{L^\infty(B)}\}$. By definition, $\mu(N)=0$. Thus, by \eqref{eqPre:SymmetryAssumption}, we have
    \begin{align*}
        \int_A\nu(x,N)\mu(\dx)=\int_N\nu(x,A)\mu(\dx)=0,
    \end{align*}
    which proves that $\nu(x,N)=0$ for $\mu$-almost all $x\in A$.
\end{proof}

Moreover, we will need the following lemma about the relation between two nested tails.

\begin{lemma}\label{lemPre:TailMonotonicity}
    Let $x_0\in X$ and $0<R\leq 2R_0$. Assume that $u\in L^\infty(B_R(x_0))\cap L_{s,\nu}^f(X)$. Then, for any $z\in B_{R/4}(x_0)$ and $0<r\leq R/2$, it is true that
    \begin{align*}
        \Tail{u}{z}{r}\leq C\left(\|u\|_{L^\infty(B_R(x_0))}+\Tail{u}{x_0}{R}\right)
    \end{align*}
    for some $C=C(p,q,C_2)>0$.
\end{lemma}

\begin{proof}
    Fix $x\in B_{r/2}(z)\subseteq B_{R/2}(x_0)$. Then, by \eqref{eqPre:TailIntegrabilityAssumption} and Lemma \ref{lemPre:SupEstimate},
    \begin{align*}
        &\int_{X\setminus B_r(z)}f\left(\frac{|u(y)|}{d(x,y)^s}\right)\frac{\nu(x,\dy)}{d(x,y)^s} \\
        &\quad\leq\int_{B_R(x_0)\setminus B_r(z)}f\left(\frac{\|u\|_{L^\infty(B_R(x_0))}}{r^s}\left(\frac{r}{d(x,y)}\right)^s\right)\frac{\nu(x,\dy)}{d(x,y)^s}+\int_{X\setminus B_R(x_0)}f\left(\frac{|u(y)|}{d(x,y)^s}\right)\frac{\nu(x,\dy)}{d(x,y)^s} \\
        &\quad\leq Cf\left(\frac{\|u\|_{L^\infty(B_R(x_0))}}{r^s}\right)r^{s(p-1)}\int_{X\setminus B_{r/2}(x)}\frac{\nu(x,\dy)}{d(x,y)^{sp}}+\tail{u}{x_0}{R} \\
        &\quad\leq\frac{C}{r^s}f\left(\frac{\|u\|_{L^\infty(B_R(x_0))}}{r^s}\right)+\tail{u}{x_0}{R},
    \end{align*}
    which proves the result by evaluating $\Tail{u}{z}{r}$.
\end{proof}

\section{Sobolev--Poincar\'{e} inequality}\label{sec:SobPoi}

In this section, we want to use the assumption that $\nu$ supports a Poincar\'{e} inequality \eqref{eqPre:PoincareAssumption} to prove a Sobolev--Poincar\'{e} inequality that involves the growth function $F$. Note that the improvement of integrability in the Sobolev--Poincar\'{e} inequality is crucial to make the De Giorgi iteration work. As a first step, we prove a Jensen-type inequality.

\begin{lemma}\label{lemSobPoi:JensenNu}
    Let $x_0\in X$, $0<R\leq R_0$ and $u\in W^{s,F,\nu}(B_R(x_0))$. Then, for any $x\in B_R(x_0)$, it holds
    \begin{align*}
        F\left(R^{t-s}\left(\int_{B_R(x_0)}\frac{|u(x)-u(y)|^p}{d(x,y)^{tp}}\nu(x,\dy)\right)^\frac{1}{p}\right)\leq C\int_{B_R(x_0)}F\left(\frac{|u(x)-u(y)|}{d(x,y)^s}\right)\nu(x,\dy)
    \end{align*}
    for some $C=C(\C,s,t,p,q,C_2)>0$.
\end{lemma}

\begin{proof}
    Define
    \begin{align*}
        W(x,y)=\left(\frac{d(x,y)}{R}\right)^{(s-t)p} \and I(x)=\int_{B_R(x_0)}W(x,y)\nu(x,\dy).
    \end{align*}
    If $I(x)=0$, then the claim is immediate because this implies $\nu(x,B_R(x_0)\setminus\{x\})=0$. Let us assume $I(x)>0$. By Lemma \ref{lemPre:distance_integrability_nu}, we know that
    \begin{align}\label{eqSobPoi:JensenNuMassBound}
        I(x)=\int_{B_R(x_0)}\left(\frac{d(x,y)}{R}\right)^{(s-t)p}\nu(x,\dy)\leq C
    \end{align}
    for some $C=C(s,t,p,C_2)>0$. Moreover, we define weights
    \begin{align*}
        w(x,y)=\frac{W(x,y)}{I(x)}.
    \end{align*}
    It is clear from the definition of $I(x)$ that
    \begin{align*}
        \int_{B_R(x_0)}w(x,y)\nu(x,\dy)=1
    \end{align*}
    for any $x\in B_R(x_0)$. Let $G$ be the convex function from Corollary \ref{corPre:ConvexityComposition}. We apply Jensen's inequality to $G$ and the probability measure $w(x,y)\nu(x,\dy)$, i.e.
    \begin{align*}
        &F\left(R^{t-s}\left(\int_{B_R(x_0)}\frac{|u(x)-u(y)|^p}{d(x,y)^{tp}}\nu(x,\dy)\right)^\frac{1}{p}\right) \\
        &=F\left(\left(\int_{B_R(x_0)}\frac{|u(x)-u(y)|^p}{d(x,y)^{sp}}W(x,y)\nu(x,\dy)\right)^\frac{1}{p}\right) \\
        &\quad\leq CG\left(\int_{B_R(x_0)}\frac{|u(x)-u(y)|^p}{d(x,y)^{sp}}I(x)w(x,y)\nu(x,\dy)\right) \\
        &\quad\leq C\int_{B_R(x_0)}G\left(\frac{|u(x)-u(y)|^p}{d(x,y)^{sp}}I(x)\right)w(x,y)\nu(x,\dy) \\
        &\quad\leq C\int_{B_R(x_0)}F\left(\frac{|u(x)-u(y)|}{d(x,y)^s}I^{1/p}(x)\right)w(x,y)\nu(x,\dy).
    \end{align*}
    Note that we applied \eqref{eqPre:ConvexityCompositionEquivalence} twice in the above argument. Since $I(x)\leq C$ by \eqref{eqSobPoi:JensenNuMassBound}, we may write
    \begin{align*}
        &F\left(R^{t-s}\left(\int_{B_R(x_0)}\frac{|u(x)-u(y)|^p}{d(x,y)^{tp}}\nu(x,\dy)\right)^\frac{1}{p}\right) \\
        &\quad\leq C\int_{B_R(x_0)}F\left(\frac{|u(x)-u(y)|}{d(x,y)^s}\right)I(x)w(x,y)\nu(x,\dy) \\
        &\quad\leq C\int_{B_R(x_0)}F\left(\frac{|u(x)-u(y)|}{d(x,y)^s}\right)W(x,y)\nu(x,\dy)
    \end{align*}
    using the definition of $w$. Finally, recall that $W(x,y)\leq 2^{(s-t)p}$ for any $x,y\in B_R(x_0)$, which finishes the proof.
\end{proof}

At this point, we could use this lemma and the theory of maximal functions (in particular, \cite[Theorem 2.2]{Hei01}) to establish a fractional Poincar\'{e} inequality, which involves the Young function $F$. However, the proof will be similar to the one given below, where we use fractional maximal functions to derive a Sobolev--Poincar\'{e} inequality. We will need the following weak-type bound for fractional maximal functions. We could not find a source that provides exactly this statement without assuming an absolute lower mass bound on the metric measure space. For convenience, we will give the full proof below, which follows the ideas of \cite[Theorem 2.2]{Hei01} and \cite[Corollary 5.3]{HKKT15}.

\begin{lemma}\label{lemSobPoi:WeakTypeBoundFractionalMaximalFunction}
    Let $x_0\in X$ and $R>0$. For a nonnegative function $g\in L^1(B_{3R}(x_0),\mu)$, define the fractional maximal function $\M_{B_R(x_0)}^s g$ by
    \begin{align}\label{eqSobPoi:FractionalMaximalFunction}
        \M_{B_R(x_0)}^s g(x)=\sup_{0<r\leq 2R}\frac{r^s}{\mu(B_r(x))}\int_{B_r(x)\cap B_R(x_0)}g\,\d\mu.
    \end{align}
    Then $\M_{B_R(x_0)}^s g$ satisfies the weak-type bound
    \begin{align*}
        \left(\frac{\mu(\{x\in B_R(x_0)\mid \M_{B_R(x_0)}^s g(x)>\tau\})}{\mu(B_R(x_0))}\right)^\frac{Q-s}{Q}\leq C\frac{R^s}{\tau}\fint_{B_{3R}(x_0)}g\,\d\mu
    \end{align*}
    for any $\tau>0$ and some $C=C(\C,s)>0$.

\end{lemma}

\begin{proof}
    Fix $\tau>0$ and define the superlevel set $E_\tau=\{x\in B_R(x_0)\mid\M_{B_R(x_0)}^s g(x)>\tau\}$. By definition of the fractional maximal function \eqref{eqSobPoi:FractionalMaximalFunction}, for any $x\in E_\tau$, we can find a radius $r_x\in(0,2R]$ such that
    \begin{align}\label{eqSobPoi:WeakTypeBoundFractionalMaximalFunctionSupremumImplication}
        \frac{r_x^s}{\mu(B_x)}\int_{B_x\cap B_R(x_0)}g\,\d\mu>\tau,
    \end{align}
    where we denote $B_x=B_{r_x}(x)$. Since $r_x\leq 2R$, we have $B_x\subseteq B_{3R}(x_0)$ such that, by Lemma \ref{lemPre:LowerMassBound},
    \begin{align*}
        \frac{\mu(B_x)}{\mu(B_{3R}(x_0))}\geq c\left(\frac{r_x}{R}\right)^Q
    \end{align*}
    or, in other words,
    \begin{align*}
        r_x^s=\left(r_x^Q\right)^\frac{s}{Q}\leq C\left(R^Q\frac{\mu(B_x)}{\mu(B_{3R}(x_0))}\right)^\frac{s}{Q}=CR^s\left(\frac{\mu(B_x)}{\mu(B_{3R}(x_0))}\right)^\frac{s}{Q}.
    \end{align*}
    Hence, \eqref{eqSobPoi:WeakTypeBoundFractionalMaximalFunctionSupremumImplication} yields
    \begin{align*}
        \tau&<CR^s\left(\frac{\mu(B_x)}{\mu(B_{3R}(x_0))}\right)^\frac{s}{Q}\frac{1}{\mu(B_x)}\int_{B_x\cap B_R(x_0)}g\,\d\mu \\
        &\leq CR^s\left(\frac{\mu(B_{3R}(x_0))}{\mu(B_x)}\right)^\frac{Q-s}{Q}\frac{1}{\mu(B_{3R}(x_0))}\int_{B_x}g\,\d\mu.
    \end{align*}
    We can rearrange this inequality to get
    \begin{align}\label{eqSobPoi:WeakTypeBoundMeasureEstimate}
        \left(\frac{\mu(B_x)}{\mu(B_R(x_0))}\right)^\frac{Q-s}{Q}\leq C\left(\frac{\mu(B_x)}{\mu(B_{3R}(x_0))}\right)^\frac{Q-s}{Q}\leq C\frac{R^s}{\tau}\frac{1}{\mu(B_{3R}(x_0))}\int_{B_x}g\,\d\mu
    \end{align}
    by the doubling property. Let us apply Lemma \ref{lemPre:5BCovering} to the family $\{B_x\}_{x\in E_\tau}$, i.e.\ we can extract a countable, pairwise disjoint subfamily $\{B^i\}_{i=1}^\infty$ satisfying
    \begin{align*}
        E_\tau\subseteq\bigcup_{x\in E_\tau}B_x\subseteq\bigcup_{i=1}^\infty 5B^i.
    \end{align*}
    Using the doubling property three times, we arrive at
    \begin{align*}
        \mu(E_\tau)\leq\sum_{i=1}^\infty\mu(5B^i)\leq\C^3\sum_{i=1}^\infty\mu(B^i).
    \end{align*}
    Recall that each $B^i$ is one of the $B_x$. Since $0<s<1\leq Q$, the subadditivity $(a+b)^\beta\leq a^\beta+b^\beta$ for $a,b\geq 0$ and $\beta\in(0,1)$ as well as \eqref{eqSobPoi:WeakTypeBoundMeasureEstimate} imply that
    \begin{align*}
        \left(\frac{\mu(E_\tau)}{\mu(B_R(x_0))}\right)^\frac{Q-s}{Q}&\leq C\left(\sum_{i=1}^\infty\frac{\mu(B^i)}{\mu(B_{3R}(x_0))}\right)^\frac{Q-s}{Q}\leq C\sum_{i=1}^\infty\left(\frac{\mu(B^i)}{\mu(B_{3R}(x_0))}\right)^\frac{Q-s}{Q} \\
        &\leq C\sum_{i=1}^\infty\frac{R^s}{\tau}\frac{1}{\mu(B_{3R}(x_0))}\int_{B^i}g\,\d\mu=C\frac{R^s}{\tau}\frac{1}{\mu(B_{3R}(x_0))}\int_{\bigcup_{i=1}^\infty B^i}g\,\d\mu \\
        &\leq C\frac{R^s}{\tau}\fint_{B_{3R}(x_0)}g\,\d\mu,
    \end{align*}
    where we also used that the $B^i$ are pairwise disjoint and contained in $B_{3R}(x_0)$ in the last two steps. Recalling the definition of $E_\tau$, the proof is complete.
\end{proof}
If one assumes an absolute lower mass bound on $\mu$, that is,
\begin{align}\label{eqSobPoi:AbsoluteLowerMassBound}
    \mu(B_r(x))\geq cr^Q \enspace\text{for all}\ r>0,\ x\in X,
\end{align}
then one obtains the full weak-type bound on $X$ for fractional maximal functions, see e.g.\ \cite[Theorem 5.2]{HKKT15}. Without \eqref{eqSobPoi:AbsoluteLowerMassBound}, the weak-type bound as proven above still leads to the following dilated Sobolev--Poincar\'{e} inequality. The proof goes along the lines of \cite[Theorem 8.1.7]{HKST15} and was adjusted to Orlicz growth and nonlocal settings.

\begin{thm}[Sobolev--Poincar\'{e} inequality] \label{thmSobPoi:SobolevPoincare}
    Let $\theta\in[1,\frac{Q}{Q-s})$. There exists $\Lambda=\Lambda(\lambda)\geq 1$ such that, for any $x_0\in X$, $0<R\leq\frac{R_0}{2\lambda}$, and every $u\in W^{s,F,\nu}(B_{\Lambda R}(x_0))$,
    \begin{align*}
        \fint_BF^\theta\left(\frac{|u-(u)_{B_R(x_0)}|}{R^s}\right)\d\mu\leq C\left(\fint_{\Lambda B}\int_{\Lambda B}F\left(\frac{|u(x)-u(y)|}{d(x,y)^s}\right)\nu(x,\dy)\mu(\dx)\right)^\theta,
    \end{align*}    
    for some $C=C(\C,s,t,p,q,\lambda,C_1,C_2,\theta)>0$, where $B=B_R(x_0)$.
\end{thm}

\begin{proof}
    By Lemma \ref{lemPre:LebesgueDifferentiation}, it holds for $\mu$-almost all $x\in B_R(x_0)$ that
    \begin{align}\label{eqSobPoi:SobolevPoincareLebesgueDifferentiation}
        u(x)=\lim_{r\to 0}\fint_{B_r(x)}u\,\d\mu.
    \end{align}
    Fix some $x$ satisfying \eqref{eqSobPoi:SobolevPoincareLebesgueDifferentiation} and define a sequence of balls $B_i=B_{r_i}(x)$ with $r_i=2^{1-i}R$, $i\geq 0$. Using \eqref{eqSobPoi:SobolevPoincareLebesgueDifferentiation} and a telescopic sum, we may write
    \begin{align}\label{eqSobPoi:SobolevPoincareTelescopicSum}
        |u(x)-(u)_{B_R(x_0)}|\leq\sum_{i=0}^\infty|(u)_{B_{i+1}}-(u)_{B_i}|+|(u)_{B_0}-(u)_{B_R(x_0)}|.
    \end{align}
    For any $i$, we have $B_{i+1}\subseteq B_i$ and $\mu(B_i)\leq C\mu(B_{i+1})$ due to the doubling property. Therefore,
    \begin{align*}
        |(u)_{B_{i+1}}-(u)_{B_i}|&=\frac{1}{\mu(B_{i+1})}\left|\int_{B_{i+1}}(u-(u)_{B_i})\d\mu\right|\leq C\frac{1}{\mu(B_i)}\left|\int_{B_i}(u-(u)_{B_i})\d\mu\right| \\
        &\leq C\fint_{B_i}|u-(u)_{B_i}|\d\mu\leq C\left(\fint_{B_i}|u-(u)_{B_i}|^p\right)^\frac{1}{p}
    \end{align*}
    by Hölder's inequality. Similarly, $B_R(x_0)\subseteq B_{2R}(x)=B_0$ such that
    \begin{align*}
        |(u)_{B_0}-(u)_{B_R(x_0)}|\leq C\left(\fint_{B_0}|u-(u)_{B_0}|^p\d\mu\right)^\frac{1}{p}.
    \end{align*}
    Let us claim that, for any $\varepsilon>0$ and any sequence of nonnegative reals $a_i\geq 0$, it is true that
    \begin{align*}
        F\left(\sum_{i=1}^\infty a_i\right)\leq C\sum_{i=1}^\infty 2^{\varepsilon(q-1)i}F(a_i)
    \end{align*}
    for some $C=C(\varepsilon)>0$, provided the right hand side converges. Indeed, set
    \begin{align*}
        S=\sum_{i=1}^\infty 2^{-\varepsilon i}=S(\varepsilon)<\infty.
    \end{align*}
    By Jensen's inequality, we get that
    \begin{align*}
        F\left(\sum_{i=1}^\infty a_i\right)=F\left(\sum_{i=1}^\infty\frac{2^{-\varepsilon i}}{S} 2^{\varepsilon i}a_i\right)\leq\sum_{i=1}^\infty \frac{2^{-\varepsilon i}}{S}F\left(2^{\varepsilon i}a_i\right)\leq\frac{1}{S}\sum_{i=1}^\infty 2^{\varepsilon(q-1)i}F(a_i)
    \end{align*}
    as desired. Hence, we can substitute the previous estimates back into \eqref{eqSobPoi:SobolevPoincareTelescopicSum} and get
    \begin{align*}
        F\left(\frac{|u(x)-(u)_{B_R(x_0)}|}{R^s}\right)&\leq C\sum_{i=0}^\infty 2^{\varepsilon(q-1)i}F\left(\frac{|(u)_{B_{i+1}}-(u)_{B_i}|}{R^s}\right)+CF\left(\frac{|(u)_{B_0}-(u)_{B_R(x_0)}|}{R^s}\right) \\
        &\leq C\sum_{i=0}^\infty 2^{\varepsilon(q-1)i}F\left(R^{-s}\left(\fint_{B_i}|u-(u)_{B_i}|^p\d\mu\right)^\frac{1}{p}\right) \\
        &\leq C\sum_{i=0}^\infty 2^{\varepsilon(q-1)i}\cdot 2^{-spi}F\left(r_i^{-s}\left(\fint_{B_i}|u-(u)_{B_i}|^p\d\mu\right)^\frac{1}{p}\right).
    \end{align*}
    Now, we can apply \eqref{eqPre:PoincareAssumption} and the doubling property in order to obtain
    \begin{align*}
        F\left(\frac{|u(x)-(u)_{B_R(x_0)}|}{R^s}\right)\leq C\sum_{i=0}^\infty 2^{(\varepsilon(q-1)-sp)i}F\left(r_i^{t-s}\left(\fint_{B^i}\int_{B^i}\frac{|u(z)-u(y)|^p}{d(z,y)^{tp}}\nu(z,\dy)\mu(\dz)\right)^\frac{1}{p}\right)
    \end{align*}
    by defining $B^i=\lambda B_i$. By a similar argument as in the proof of Lemma \ref{lemSobPoi:JensenNu}, we can apply Jensen's inequality to the function $G$ of Corollary \ref{corPre:ConvexityComposition} and the outer $\mu$-integral. Then, applying Lemma \ref{lemSobPoi:JensenNu} to the inner $\nu$-integral eventually yields
    \begin{align}\label{eqSobPoi:SobolevPoincareApplicationPoincare}
        F\left(\frac{|u(x)-(u)_{B_R(x_0)}|}{R^s}\right)\leq C\sum_{i=0}^\infty 2^{(\varepsilon(q-1)-sp)i}\fint_{B^i}\int_{B^i}F\left(\frac{|u(z)-u(y)|}{d(z,y)^s}\right)\nu(z,\dy)\mu(\dz).
    \end{align}
    Set $\Lambda=1+2\lambda$ and define the function
    \begin{align*}
        g(z)=\int_{B_{\Lambda R}(x_0)}F\left(\frac{|u(z)-u(y)|}{d(z,y)^s}\right)\nu(z,\dy), \quad z\in B_{\Lambda R}(x_0),
    \end{align*}
    and its fractional maximal function in $B_{\Lambda R}(x_0)$ as in \eqref{eqSobPoi:FractionalMaximalFunction}. Since $B^i\subseteq B_{\Lambda R}(x_0)$, it follows from \eqref{eqSobPoi:SobolevPoincareApplicationPoincare} that
    \begin{align*}
        F\left(\frac{|u(x)-(u)_{B_R(x_0)}|}{R^s}\right)&\leq C\sum_{i=0}^\infty 2^{(\varepsilon(q-1)-sp)i}\fint_{B^i}g(z)\mu(\dz) \\
        &\leq\frac{C}{R^s}\sum_{i=0}^\infty 2^{(\varepsilon(q-1)-s(p-1))i}r_i^s\fint_{B^i}g(z)\mu(\dz)\leq\frac{C}{R^s}\M_{B_{\Lambda R}(x_0)}^s g(x).
    \end{align*}
    In the last step, we chose
    \begin{align*}
        \varepsilon=\frac{s(p-1)}{2(q-1)}=\varepsilon(s,p,q)>0
    \end{align*}
    in order to obtain a convergent geometric sum. Next, we may apply Lemma \ref{lemSobPoi:WeakTypeBoundFractionalMaximalFunction} to obtain that
    \begin{align}\label{eqSobPoi:SobolevPoincareWeakTypeBound}
        \left(\frac{\mu(\{x\in B_R(x_0)\mid F(R^{-s}|u(x)-(u)_{B_R(x_0)}|)>\tau\})}{\mu(B_R(x_0))}\right)^\frac{Q-s}{Q}\leq\frac{\overline{C}}{\tau}\fint_{B_{3\Lambda R}(x_0)}g\,\d\mu
    \end{align}
    for some specific constant $\overline{C}=\overline{C}(\C,s,t,p,q,\lambda,C_1,C_2)>0$. Define
    \begin{align*}
        A=\overline{C}\fint_{B_{3\Lambda R}(x_0)}\int_{B_{3\Lambda R}(x_0)}F\left(\frac{|u(x)-u(y)|}{d(x,y)^s}\right)\nu(x,\dy)\mu(\dx).
    \end{align*}
    Then \eqref{eqSobPoi:SobolevPoincareWeakTypeBound} gives that
    \begin{align*}
        \left(\frac{\mu(\{x\in B_R(x_0)\mid F(R^{-s}|u(x)-(u)_{B_R(x_0)}|)>\tau\})}{\mu(B_R(x_0))}\right)^\frac{Q-s}{Q}\leq\min\left\{1,\frac{A}{\tau}\right\}.
    \end{align*}
    Let us finally apply Cavalieri's principle for the function $\tau\mapsto\tau^\theta$, which yields
    \begin{align*}
        &\fint_{B_R(x_0)}F^\theta\left(\frac{|u-(u)_{B_R(x_0)}|}{R^s}\right)\d\mu \\
        &\quad =\theta\int_0^\infty\tau^{\theta-1}\frac{\mu(\{x\in B_R(x_0)\mid F(R^{-s}|u(x)-(u)_{B_R(x_0)}|)>\tau\})}{\mu(B_R(x_0))}\d\tau \\
        &\quad\leq\theta\int_0^A\tau^{\theta-1}\d\tau+C\int_A^\infty\tau^{\theta-1}\left(\frac{A}{\tau}\right)^\frac{Q}{Q-s}\d\tau\leq A^\theta+C A^\frac{Q}{Q-s}\int_A^\infty\tau^{\theta-\frac{Q}{Q-s}-1}\d\tau\leq CA^\theta,
    \end{align*}
    since $\theta<\frac{Q}{Q-s}$. Replacing $\Lambda$ by $3\Lambda$ now gives the desired result.
\end{proof}

\begin{rem}
    Let $\theta=1$. Then, in particular, Theorem \ref{thmSobPoi:SobolevPoincare} states that any fractional Poincar\'{e} inequality of the form \eqref{eqPre:PoincareAssumption} with some $t\in[0,s)$ implies a fractional Poincar\'{e} inequality with respect to the growth function $F$.
\end{rem}

\begin{rem}\label{remSobPoi:Remark1Poincare}
    One can also apply the usual weak-type bound for the maximal function of order zero to obtain Poincar\'{e} inequality with $F$-growth, i.e.\ $\theta=1$. It is straightforward to verify that the proof is stable as $p\to 1^+$, which means that the result remains true for $p=1$. Since we are allowed to choose $F(t)=t^p$, we can justify that the Poincar\'{e} inequality \eqref{eqIntro:PoincareAssumption} implies \eqref{eqPre:PoincareAssumption} for any $t\in[0,s)$. More generally, we notice that any Poincar\'{e} inequality
    \begin{align*}
        \int_{B_R(x_0)}|u-(u)_{B_R(x_0)}|^{p_0}\d\mu\lesssim R^{\tau p_0}\int_{B_{\lambda R}(x_0)}\int_{B_{\lambda R}(x_0)}\frac{|u(x)-u(y)|^{p_0}}{d(x,y)^{\tau p_0}}\nu(x,\dy)\mu(\dx)
    \end{align*}
    for $\lambda\geq 1$, $\tau\in[0,s)$ and $p_0\in[1,p]$ implies \eqref{eqPre:PoincareAssumption} for any $t\in(\tau,s)$ and a different $\lambda$. Hence, one may replace $p$ in \eqref{eqPre:PoincareAssumption} by any $p_0\in[1,p]$ and the results in this article remain true.
\end{rem}

The proof requires a dilated ball $B_{\Lambda R}(x_0)$ because no extension operator is assumed for the space $W^{s,F,\nu}$.

\section{Energy estimates}\label{sec:Energy}

By choosing certain test functions, we can deduce two important energy estimates for weak supersolutions to $\L u=0$, a Caccioppoli inequality and a logarithmic estimate in the spirit of \cite{CKW22,BKO23,BDNS25}.

\begin{lemma}[Caccioppoli inequality] \label{lemEnergy:Caccioppoli}
    Fix $x_0\in X$ and some radii $0<r<R\leq R_0$. Let $u\in W^{s,F,\nu}(B_R(x_0))\cap L_{s,\nu}^f(X)$ be a weak supersolution to $\L u=0$ in $B_R(x_0)$. For $\alpha\geq 0$, define $w=(u-\alpha)_-$. Then, for any $\eta\in C^{0,1}(B_R(x_0))$ with $0\leq\eta\leq 1$ and $\eta=0$ outside $B_r(x_0)$, it holds
    \begin{multline} \label{eqEnergy:Caccioppoli}
        \int_{B_R(x_0)}\int_{B_R(x_0)} F\left(\frac{|w(x)\eta^q(x)-w(y)\eta^q(y)|}{d(x,y)^s}\right)\nu(x,\dy)\mu(\dx) \\
        \leq C\int_{B_R(x_0)}\int_{B_R(x_0)} F\left(\frac{|\eta(x)-\eta(y)|}{d(x,y)^s}\max\{w(x),w(y)\}\right)\nu(x,\dy)\mu(\dx) \\
        +C\int_{B_R(x_0)} w(x)\eta^q(x)\mu(\dx)\sup_{x\in B_r(x_0)}\int_{X\setminus B_R(x_0)}f\left(\frac{w(y)}{d(x,y)^s}\right)\frac{\nu(x,\dy)}{d(x,y)^s}
    \end{multline}
    for some $C=C(p,q,\kappa)>0$.
\end{lemma}

\begin{proof}
    We use $\varphi=w\eta^q$ as a test function in \eqref{eqPre:Test}, which yields
    \begin{align}
        0&\leq\int_X\int_Xh\left(x,y,\frac{u(x)-u(y)}{d(x,y)^s}\right)\frac{w(x)\eta^q(x)-w(y)\eta^q(y)}{d(x,y)^s}\nu(x,\dy)\mu(\dx) \nonumber \\
        &=\int_{B_R(x_0)}\int_{B_R(x_0)} h\left(x,y,\frac{u(x)-u(y)}{d(x,y)^s}\right)\frac{w(x)\eta^q(x)-w(y)\eta^q(y)}{d(x,y)^s}\nu(x,\dy)\mu(\dx) \nonumber \\
        &\quad +\int_{B_r(x_0)}\int_{X\setminus B_R(x_0)}h\left(x,y,\frac{u(x)-u(y)}{d(x,y)^s}\right)w(x)\eta^q(x)\frac{\nu(x,\dy)}{d(x,y)^s}\mu(\dx) \nonumber \\
        &\quad -\int_{X\setminus B_R(x_0)}\int_{B_r(x_0)} h\left(x,y,\frac{u(x)-u(y)}{d(x,y)^s}\right)w(y)\eta^q(y)\frac{\nu(x,\dy)}{d(x,y)^s}\mu(\dx)\eqqcolon I+II+III\label{eqEnergy:CaccioppoliTest}
    \end{align}
    as $\eta$ vanishes outside $B_r(x_0)$. The admissibility of this test function can be verified as usual. We first estimate $I$ and fix $x,y\in B_R(x_0)$ with $x\neq y$ and $u(x)\leq u(y)$. Set
    \begin{align*}
        J(x,y)=h\left(x,y,\frac{u(x)-u(y)}{d(x,y)^s}\right)\frac{w(x)\eta^q(x)-w(y)\eta^q(y)}{d(x,y)^s}
    \end{align*}
    and let us distinguish several cases. If $u(x),u(y)\geq\alpha$, then $J(x,y)=0$. We consider the case where $u(x),u(y)<\alpha$. Then $u(x)-u(y)=w(y)-w(x)$ and thus,
    \begin{align*}
        J(x,y)=h\left(x,y,\frac{w(y)-w(x)}{d(x,y)^s}\right)\frac{w(x)\eta^q(x)-w(y)\eta^q(y)}{d(x,y)^s}.
    \end{align*}
    Let us use the discrete product rule
    \begin{align*}
        a_1b_1-a_2b_2=(a_1-a_2)\frac{b_1+b_2}{2}+\frac{a_1+a_2}{2}(b_1-b_2)
    \end{align*}
    to write
    \begin{align*}
        J(x,y)&=h\left(x,y,\frac{w(y)-w(x)}{d(x,y)^s}\right)\frac{w(x)-w(y)}{d(x,y)^s}\frac{\eta^q(x)+\eta^q(y)}{2} \\
        &\quad +h\left(x,y,\frac{w(y)-w(x)}{d(x,y)^s}\right)\frac{w(x)+w(y)}{2}\frac{\eta^q(x)-\eta^q(y)}{d(x,y)^s}  \\
        &\leq -h\left(x,y,\frac{w(y)-w(x)}{d(x,y)^s}\right)\frac{w(y)-w(x)}{|w(y)-w(x)|}\frac{w(x)-w(y)}{d(x,y)^s}\frac{\eta^q(x)+\eta^q(y)}{2} \\
        &\quad +h\left(x,y,\frac{w(y)-w(x)}{d(x,y)^s}\right)w(x)\frac{|\eta^q(x)-\eta^q(y)|}{d(x,y)^s}\eqqcolon J_1(x,y)+J_2(x,y),
    \end{align*}
    since $w(x)\geq w(y)$. For $J_1$, we apply \eqref{eqPre:Comparability} and \eqref{eqPre:GrowthAssumption} to obtain
    \begin{align}
        J_1(x,y)&\leq -\kappa^{-1}f\left(\frac{w(x)-w(y)}{d(x,y)^s}\right)\frac{w(x)-w(y)}{d(x,y)^s}\frac{\eta^q(x)+\eta^q(y)}{2} \nonumber \\
        &\leq -\frac{p}{\kappa}F\left(\frac{w(x)-w(y)}{d(x,y)^s}\right)\frac{\eta^q(x)+\eta^q(y)}{2}. \label{eqEnergy:CaccioppoliEstimateJ1}
    \end{align}
    In order to estimate $J_2$, we use \eqref{eqPre:UpperBoundNonStandardGrowthFunction}, the mean value theorem in the form
    \begin{align*}
        |\eta^q(x)-\eta^q(y)|&\leq q\max\{\eta^{q-1}(x),\eta^{q-1}(y)\}|\eta(x)-\eta(y)|\leq 2q\left(\frac{\eta^q(x)+\eta^q(y)}{2}\right)^\frac{q-1}{q}|\eta(x)-\eta(y)|,
    \end{align*}
    and Lemma \ref{lemPre:Young} to obtain
    \begin{align*}
        J_2(x,y)&\leq\kappa f\left(\frac{w(x)-w(y)}{d(x,y)^s}\right)w(x)\frac{|\eta^q(x)-\eta^q(y)|}{d(x,y)^s} \\
        &\leq 2q\kappa f\left(\frac{w(x)-w(y)}{d(x,y)^s}\right)w(x)\left(\frac{\eta^q(x)+\eta^q(y)}{2}\right)^\frac{q-1}{q}\frac{|\eta(x)-\eta(y)|}{d(x,y)^s} \\
        &\leq 2\varepsilon q(q-1)\kappa F\left(\frac{w(x)-w(y)}{d(x,y)^s}\right)\frac{\eta^q(x)+\eta^q(y)}{2}+2q\varepsilon^{-q}\kappa F\left(\frac{|\eta(x)-\eta(y)|}{d(x,y)^s}w(x)\right)
    \end{align*}
    for $\varepsilon\in(0,1]$. By choosing
    \begin{align*}
        \varepsilon=\min\left\{\frac{p}{4q(q-1)\kappa^2},1\right\}
    \end{align*}
    and recalling \eqref{eqEnergy:CaccioppoliEstimateJ1}, we get that
    \begin{align*}
        J_2(x,y)\leq -\frac{J_1(x,y)}{2}+CF\left(\frac{|\eta(x)-\eta(y)|}{d(x,y)^s}w(x)\right)
    \end{align*}
    and therefore,
    \begin{align}
        J(x,y)&\leq\frac{J_1(x,y)}{2}+CF\left(\frac{|\eta(x)-\eta(y)|}{d(x,y)^s}w(x)\right) \nonumber \\
        &\leq -cF\left(\frac{w(x)-w(y)}{d(x,y)^s}\right)\eta^q(x)+CF\left(\frac{|\eta(x)-\eta(y)|}{d(x,y)^s}w(x)\right). \label{eqEnergy:CaccioppoliEstimateJFirstCase}
    \end{align}
    We also need to treat the case where $u(x)<\alpha\leq u(y)$, that is, $w(x)=\alpha-u(x)$ and $w(y)=0$. Here we have
    \begin{align*}
        J(x,y)&=h\left(x,y,\frac{u(x)-u(y)}{d(x,y)^s}\right)\frac{w(x)\eta^q(x)}{d(x,y)^s}\leq -\kappa^{-1} f\left(\frac{u(y)-u(x)}{d(x,y)^s}\right)\frac{w(x)\eta^q(x)}{d(x,y)^s} \\
        &\leq -\kappa^{-1}f\left(\frac{\alpha-u(x)}{d(x,y)^s}\right)\frac{w(x)\eta^q(x)}{d(x,y)^s}=-\kappa^{-1}f\left(\frac{w(x)}{d(x,y)^s}\right)\frac{w(x)\eta^q(x)}{d(x,y)^s}
    \end{align*}
    due to \eqref{eqPre:Comparability} and the monotonicity of $f$. Let us apply \eqref{eqPre:GrowthAssumption} and $w(y)=0$ to see that
    \begin{align*}
        J(x,y)\leq -\frac{p}{\kappa}F\left(\frac{w(x)}{d(x,y)^s}\right)\eta^q(x)=-\frac{p}{\kappa}F\left(\frac{w(x)-w(y)}{d(x,y)^s}\right)\eta^q(x)
    \end{align*}
    Note that we just showed that \eqref{eqEnergy:CaccioppoliEstimateJFirstCase} also holds true in this case for some constant $c>0$. By symmetry, we can conclude that, for any $x,y\in B_R(x_0)$ with $x\neq y$, it is true that
    \begin{align*}
        J(x,y)\leq -cF\left(\frac{|w(x)-w(y)|}{d(x,y)^s}\right)\min\{\eta^q(x),\eta^q(y)\}+CF\left(\frac{|\eta(x)-\eta(y)|}{d(x,y)^s}\max\{w(x),w(y)\}\right),
    \end{align*}
    whenever $u(x)\leq u(y)$. This allows to estimate $I$, namely
    \begin{align}
        I&\leq\int_{B_R(x_0)}\int_{B_R(x_0)} J(x,y)\nu(x,\dy)\mu(\dx) \nonumber \\
        &\leq -c\int_{B_R(x_0)}\int_{B_R(x_0)} F\left(\frac{|w(x)-w(y)|}{d(x,y)^s}\right)\min\{\eta^q(x),\eta^q(y)\}\nu(x,\dy)\mu(\dx) \nonumber \\
        &\quad +C\int_{B_R(x_0)}\int_{B_R(x_0)} F\left(\frac{|\eta(x)-\eta(y)|}{d(x,y)^s}\max\{w(x),w(y)\}\right)\nu(x,\dy)\mu(\dx). \label{eqEnergy:CaccioppoliEstimateI}
    \end{align}

    Next, we estimate $II$. For $x\in B_r(x_0)$ and $y\notin B_R(x_0)$, note that
    \begin{align*}
        h\left(x,y,\frac{u(x)-u(y)}{d(x,y)^s}\right)w(x)\leq \kappa f\left(\frac{w(y)}{d(x,y)^s}\right)w(x).
    \end{align*}
    Indeed, if $u(x)\geq\alpha$, both sides of the inequality vanish. If $u(x)\leq u(y)$, the left hand side is nonpositive such that the estimate becomes also trivial. So we only need to consider the case $u(y)\leq u(x)<\alpha$, where we have
    \begin{align*}
        h\left(x,y,\frac{u(x)-u(y)}{d(x,y)^s}\right)w(x)&\leq\kappa f\left(\frac{u(x)-u(y)}{d(x,y)^s}\right)w(x) \\
        &\leq\kappa f\left(\frac{\alpha-u(y)}{d(x,y)^s}\right)w(x)=\kappa f\left(\frac{w(y)}{d(x,y)^s}\right)w(x)
    \end{align*}
    due to \eqref{eqPre:Comparability}. Therefore,
    \begin{align}
        II&\leq\kappa\int_{B_r(x_0)}\int_{X\setminus B_R(x_0)}f\left(\frac{w(y)}{d(x,y)^s}\right)w(x)\eta^q(x)\frac{\nu(x,\dy)}{d(x,y)^s}\mu(\dx) \nonumber \\
        &\leq\kappa\int_{B_R(x_0)}w(x)\eta^q(x)\mu(\dx)\sup_{x\in B_r(x_0)}\int_{X\setminus B_R(x_0)}f\left(\frac{w(y)}{d(x,y)^s}\right)\frac{\nu(x,\dy)}{d(x,y)^s}. \label{eqEnergy:CaccioppoliEstimateII}
    \end{align}

    Finally, we want to apply Lemma \ref{lemPre:SymmetryFunctions} to see that $III=II$. For that, it suffices to prove that
    \begin{align*}
        \int_{B_r(x_0)}\int_{X\setminus B_R(x_0)}\left|h\left(x,y,\frac{u(x)-u(y)}{d(x,y)^s}\right)\right|w(x)\eta^q(x)\frac{\nu(x,\dy)}{d(x,y)^s}\mu(\dx)<\infty.
    \end{align*}
    In fact, we can calculate
    \begin{align*}
        &\int_{B_r(x_0)}\int_{X\setminus B_R(x_0)}\left|h\left(x,y,\frac{u(x)-u(y)}{d(x,y)^s}\right)\right|w(x)\eta^q(x)\frac{\nu(x,\dy)}{d(x,y)^s}\mu(\dx) \\
        &\quad\leq\kappa\int_{B_r(x_0)}\int_{X\setminus B_R(x_0)}f\left(\frac{|u(x)-u(y)|}{d(x,y)^s}\right)w(x)\eta^q(x)\frac{\nu(x,\dy)}{d(x,y)^s}\mu(\dx) \\
        &\quad\leq C\int_{B_r(x_0)}\int_{X\setminus B_R(x_0)}f\left(\frac{|u(x)|}{d(x,y)^s}\right)w(x)\frac{\nu(x,\dy)}{d(x,y)^s}\mu(\dx) \\
        &\quad\quad +C\int_{B_r(x_0)}\int_{X\setminus B_R(x_0)}f\left(\frac{|u(y)|}{d(x,y)^s}\right)w(x)\frac{\nu(x,\dy)}{d(x,y)^s}\mu(\dx).
    \end{align*}
    Set $\rho=R-r$. The first term is finite due to \eqref{eqPre:TailIntegrabilityAssumption}, \eqref{eqPre:GrowthAssumption} and Lemma \ref{lemPre:Young}, namely
    \begin{align*}
        &\int_{B_r(x_0)}\int_{X\setminus B_R(x_0)}f\left(\frac{|u(x)|}{d(x,y)^s}\right)w(x)\frac{\nu(x,\dy)}{d(x,y)^s}\mu(\dx) \\
        &\quad\leq\int_{B_r(x_0)}\int_{X\setminus B_\rho(x)}f\left(\frac{|u(x)|}{\rho^s}\left(\frac{\rho}{d(x,y)}\right)^s\right)(|u(x)|+\alpha)\frac{\nu(x,\dy)}{d(x,y)^s}\mu(\dx) \\
        &\quad\leq C\rho^{sp}\int_{B_r(x_0)}\frac{|u(x)|+\alpha}{\rho^s}f\left(\frac{|u(x)|}{\rho^s}\right)\int_{X\setminus B_\rho(x)}\frac{\nu(x,\dy)}{d(x,y)^{sp}}\mu(\dx) \\
        &\quad\leq C\int_{B_r(x_0)}F\left(\frac{|u|}{\rho^s}\right)\d\mu+CF\left(\frac{\alpha}{r^s}\right)\mu(B_r(x_0))<\infty.
    \end{align*}
    For the second term, we have
    \begin{align*}
        &\int_{B_r(x_0)}\int_{X\setminus B_R(x_0)}f\left(\frac{|u(y)|}{d(x,y)^s}\right)w(x)\frac{\nu(x,\dy)}{d(x,y)^s}\mu(\dx) \\
        &\quad\leq\int_{B_R(x_0)} w(x)\mu(\dx)\sup_{x\in B_r(x_0)}\int_{X\setminus B_R(x_0)}f\left(\frac{|u(y)|}{d(x,y)^s}\right)\frac{\nu(x,\dy)}{d(x,y)^s}<\infty
    \end{align*}
    by Lemma \ref{lemPre:TailFiniteness}, which also proves that the right hand side of \eqref{eqEnergy:Caccioppoli} is finite. Hence, Lemma \ref{lemPre:SymmetryFunctions} and the symmetry of $h$ imply that $III=II$.
    
    Therefore, we can combine \eqref{eqEnergy:CaccioppoliTest} with the estimates \eqref{eqEnergy:CaccioppoliEstimateI} and \eqref{eqEnergy:CaccioppoliEstimateII} to obtain 
    \begin{multline*}
        \int_{B_R(x_0)}\int_{B_R(x_0)} F\left(\frac{|w(x)-w(y)|}{d(x,y)^s}\right)\min\{\eta^q(x),\eta^q(y)\}\nu(x,\dy)\mu(\dx) \\
        \leq C\int_{B_R(x_0)}\int_{B_R(x_0)} F\left(\frac{|\eta(x)-\eta(y)|}{d(x,y)^s}\max\{w(x),w(y)\}\right)\nu(x,\dy)\mu(\dx) \\
        +C\int_{B_R(x_0)} w(x)\eta^q(x)\mu(\dx)\sup_{x\in B_r(x_0)}\int_{X\setminus B_R(x_0)}f\left(\frac{w(y)}{d(x,y)^s}\right)\frac{\nu(x,\dy)}{d(x,y)^s}.
    \end{multline*}
    Finally, it is left to notice that
    \begin{align*}
        &F\left(\frac{|w(x)\eta^q(x)-w(y)\eta^q(y)|}{d(x,y)^s}\right) \\
        &\quad\leq C\left(F\left(\frac{|w(x)-w(y)|}{d(x,y)^s}\min\{\eta^q(x),\eta^q(y)\}\right)+F\left(\frac{|\eta^q(x)-\eta^q(y)|}{d(x,y)^s}\max\{w(x),w(y)\}\right)\right) \\
        &\quad\leq CF\left(\frac{|w(x)-w(y)|}{d(x,y)^s}\right)\min\{\eta^{pq}(x),\eta^{pq}(y)\} \\
        &\quad\quad +F\left(q\max\{\eta^{q-1}(x),\eta^{q-1}(y)\}\frac{|\eta(x)-\eta(y)|}{d(x,y)^s}\max\{w(x),w(y)\}\right) \\
        &\quad\leq CF\left(\frac{|w(x)-w(y)|}{d(x,y)^s}\right)\min\{\eta^q(x),\eta^q(y)\}+CF\left(\frac{|\eta(x)-\eta(y)|}{d(x,y)^s}\max\{w(x),w(y)\}\right)
    \end{align*}
    as $0\leq\eta\leq 1$. Thus, the desired result follows from the previous estimate.
\end{proof}

\begin{rem}
    The analogous result for subsolutions $u$ also holds true for $w=(u-\alpha)_+$.
\end{rem}

The following logarithmic estimate is similar to \cite[Proposition 3.4]{BKO23}. However, we improved their estimate in two directions, even in the case $\nu(x,\dy)=|x-y|^{-d}\dy$ on $\Rd$. Our version allows for $p$-growth on the left hand side of the estimate and moreover, it has a fractional form of order almost $s$, while the logarithmic estimate of \cite{BKO23} has linear growth and order zero. These improvements allow us to treat a weaker Poincar\'{e} inequality \eqref{eqPre:PoincareAssumption}, despite being of independent interest.

\begin{lemma}[Logarithmic estimate] \label{lemEnergy:LogEstimate}
    Fix $x_0\in X$, $0<R\leq R_0$ and $a>0$. Let $u\in W^{s,F,\nu}(B_R(x_0))\cap L_{s,\nu}^f(X)$ be a weak supersolution to $\L u=0$ in $B_R(x_0)$ such that $u\geq 0$ in $B_R(x_0)$. Then, for any $0<r\leq R/4$,
    \begin{align*}
        \fint_{B_r(x_0)}\int_{B_r(x_0)}\left|\log\frac{u(x)+a}{u(y)+a}\right|^p\frac{\nu(x,\dy)}{d(x,y)^{tp}}\mu(\dx)\leq\frac{C}{r^{tp}}\left(1+\frac{r^s}{f\left(\frac{a}{r^s}\right)}\tail{u_-}{x_0}{R}\right)
    \end{align*}
    for some $C=C(\C,s,t,p,q,\kappa,C_2)>0$.
\end{lemma}

\begin{proof}
    Define $v=u+a$ and let $\eta\colon X\to[0,1]$ be a Lipschitz function with
    \begin{align*}
        \indicator_{B_r(x_0)}\leq\eta\leq\indicator_{B_{2r}(x_0)} \and \sup_{\substack{x,y\in X\\x\neq y}}\frac{|\eta(x)-\eta(y)|}{d(x,y)}\leq\frac{1}{2r-r}=\frac{1}{r}.
    \end{align*}
    The existence of such $\eta$ can be justified on any metric space. Indeed, one may take
    \begin{align*}
        \eta(x)=\max\left\{0,\min\left\{1,\frac{2r-d(x,x_0)}{r}\right\}\right\}.
    \end{align*}
    We will use
    \begin{align*}
        \varphi=\frac{v\eta^q}{F\left(\frac{v}{r^s}\right)}
    \end{align*}
    as a test function, which gives
    \begin{align}
        0&\leq\int_X\int_Xh\left(x,y,\frac{u(x)-u(y)}{d(x,y)^s}\right)\frac{\varphi(x)-\varphi(y)}{d(x,y)^s}\nu(x,\dy)\mu(\dx) \nonumber \\
        &=\int_{B_{4r}(x_0)}\int_{B_{4r}(x_0)}h\left(x,y,\frac{u(x)-u(y)}{d(x,y)^s}\right)\frac{\varphi(x)-\varphi(y)}{d(x,y)^s}\nu(x,\dy)\mu(\dx) \nonumber \\
        &\quad +\int_{B_{4r}(x_0)}\int_{X\setminus B_{4r}(x_0)}h\left(x,y,\frac{u(x)-u(y)}{d(x,y)^s}\right)\varphi(x)\frac{\nu(x,\dy)}{d(x,y)^s}\mu(\dx) \nonumber \\
        &\quad -\int_{X\setminus B_{4r}(x_0)}\int_{B_{4r}(x_0)}h\left(x,y,\frac{u(x)-u(y)}{d(x,y)^s}\right)\varphi(y)\frac{\nu(x,\dy)}{d(x,y)^s}\mu(\dx)\eqqcolon I+II+III. \label{eqEnergy:LogLemmaTest}
    \end{align}

    We start by estimating $I$. Fix $x,y\in B_{4r}(x_0)$ and define
    \begin{align*}
        J(x,y)=h\left(x,y,\frac{u(x)-u(y)}{d(x,y)^s}\right)\frac{\varphi(x)-\varphi(y)}{d(x,y)^s}.
    \end{align*}
    Assume first that $v(y)\leq v(x)\leq 2v(y)$. Since $u(x)-u(y)=v(x)-v(y)$, we have
    \begin{align*}
        J(x,y)&=h\left(x,y,\frac{v(x)-v(y)}{d(x,y)^s}\right)\left(\frac{v(x)}{F\left(\frac{v(x)}{r^s}\right)}-\frac{v(y)}{F\left(\frac{v(y)}{r^s}\right)}\right)\frac{\eta^q(x)}{d(x,y)^s} \\
        &\quad +h\left(x,y,\frac{v(x)-v(y)}{d(x,y)^s}\right)\frac{v(y)}{F\left(\frac{v(y)}{r^s}\right)}\frac{\eta^q(x)-\eta^q(y)}{d(x,y)^s}\eqqcolon J_1(x,y)+J_2(x,y).
    \end{align*}
    $J_1$ can be estimated as follows. By the mean value theorem, there exists $\xi\in[v(y),v(x)]$ with
    \begin{align*}
        \frac{v(x)}{F\left(\frac{v(x)}{r^s}\right)}-\frac{v(y)}{F\left(\frac{v(y)}{r^s}\right)}&=\frac{1}{F\left(\frac{\xi}{r^s}\right)}\left(1-\frac{\frac{\xi}{r^s}f\left(\frac{\xi}{r^s}\right)}{F\left(\frac{\xi}{r^s}\right)}\right)(v(x)-v(y)) \\
        &\leq -(p-1)\frac{v(x)-v(y)}{F\left(\frac{v(x)}{r^s}\right)}\leq -c\frac{v(x)-v(y)}{F\left(\frac{v(y)}{r^s}\right)}
    \end{align*}
    where we also used \eqref{eqPre:GrowthAssumption} and $v(x)\leq 2v(y)$. Hence, using \eqref{eqPre:Comparability} and \eqref{eqPre:GrowthAssumption} once more, we see that
    \begin{align*}
        J_1(x,y)\leq -\kappa^{-1}cf\left(\frac{v(x)-v(y)}{d(x,y)^s}\right)\frac{\eta^q(x)}{F\left(\frac{v(y)}{r^s}\right)}\frac{v(x)-v(y)}{{d(x,y)^s}}\leq -\overline{c}F\left(\frac{v(x)-v(y)}{d(x,y)^s}\right)\frac{\eta^q(x)}{F\left(\frac{v(y)}{r^s}\right)}
    \end{align*}
    for some $\overline{c}=\overline{c}(p,q,\kappa)>0$. For $J_2$, let us also distinguish two cases. If $\eta(x)\leq\eta(y)$, then $J_2(x,y)\leq 0$. If $\eta(x)>\eta(y)$, then we use \eqref{eqPre:UpperBoundNonStandardGrowthFunction}, the mean value theorem for the function $\tau\mapsto\tau^q$, and Lemma \ref{lemPre:Young}, which yield
    \begin{align*}
        J_2(x,y)&\leq q\kappa f\left(\frac{v(x)-v(y)}{d(x,y)^s}\right)\frac{v(y)}{F\left(\frac{v(y)}{r^s}\right)}\eta^{q-1}(x)\frac{\eta(x)-\eta(y)}{d(x,y)^s} \\
        &\leq q\kappa\left(\varepsilon(q-1)F\left(\frac{v(x)-v(y)}{d(x,y)^s}\right)\eta^q(x)+\varepsilon^{-q}F\left(\frac{\eta(x)-\eta(y)}{d(x,y)^s}v(y)\right)\right)\frac{1}{F\left(\frac{v(y)}{r^s}\right)}
    \end{align*}
    for any $\varepsilon\in(0,1]$. Let us choose
    \begin{align*}
        \varepsilon=\min\left\{\frac{\overline{c}}{2q(q-1)\kappa},1\right\}.
    \end{align*}
    Then
    \begin{align*}
        J_2(x,y)\leq -\frac{J_1(x,y)}{2}+CF\left(\frac{\eta(x)-\eta(y)}{d(x,y)^s}v(y)\right)\frac{1}{F\left(\frac{v(y)}{r^s}\right)}
    \end{align*}
    and by the choice of $\eta$, we have
    \begin{align*}
        F\left(\frac{\eta(x)-\eta(y)}{d(x,y)^s}v(y)\right)&\leq CF\left(\left(\frac{d(x,y)}{r}\right)^{1-s}\frac{v(y)}{r^s}\right)\leq C\left(\frac{d(x,y)}{r}\right)^{(1-s)p}F\left(\frac{v(y)}{r^s}\right).
    \end{align*}
    Therefore, we see that
    \begin{align}\label{eqEnergy:LogLemmaFirstEstimateJ}
        J(x,y)\leq -cF\left(\frac{v(x)-v(y)}{d(x,y)^s}\right)\frac{\eta^q(x)}{F\left(\frac{v(y)}{r^s}\right)}+C\left(\frac{d(x,y)}{r}\right)^{(1-s)p}.
    \end{align}
    Finally, let us use the mean value theorem for the logarithm to get
    \begin{align*}
        \left(\log\frac{v(x)}{v(y)}\right)^p\leq\left(\frac{v(x)-v(y)}{v(y)}\right)^p=\frac{\left(\frac{v(x)-v(y)}{d(x,y)^s}\right)^p}{\left(\frac{v(y)}{r^s}\right)^p}\left(\frac{d(x,y)}{r}\right)^{sp}.
    \end{align*}
    Set
    \begin{align*}
        A=\frac{v(x)-v(y)}{d(x,y)^s} \and B=\frac{v(y)}{r^s}.
    \end{align*}
    If $A\geq B$, then the monotonicity of the function $\tau\mapsto\frac{F(\tau)}{\tau^p}$ gives
    \begin{align*}
        \left(\log\frac{v(x)}{v(y)}\right)^p\leq\frac{A^p}{B^p}\left(\frac{d(x,y)}{r}\right)^{sp}\leq C\frac{F(A)}{F(B)}\frac{A^p}{F(A)}\frac{F(B)}{B^p}\left(\frac{d(x,y)}{r}\right)^{sp}\leq C\frac{F(A)}{F(B)}\left(\frac{d(x,y)}{r}\right)^{sp}.
    \end{align*}
    On the other hand, if $A<B$, then we write
    \begin{align*}
        \left(\log\frac{v(x)}{v(y)}\right)^p\leq\frac{A^p}{B^p}\left(\frac{d(x,y)}{r}\right)^{sp}\leq\left(\frac{d(x,y)}{r}\right)^{sp}.
    \end{align*}
    Therefore, in any case where $v(y)\leq v(x)\leq 2v(y)$ holds true, we find that
    \begin{align*}
        \left(\log\frac{v(x)}{v(y)}\right)^p&\leq C\left(\frac{F\left(\frac{v(x)-v(y)}{d(x,y)^s}\right)}{F\left(\frac{v(y)}{r^s}\right)}+1\right)\left(\frac{d(x,y)}{r}\right)^{sp} \\
        &\leq C\left(\frac{F\left(\frac{v(x)-v(y)}{d(x,y)^s}\right)}{F\left(\frac{v(y)}{r^s}\right)}+\left(\frac{d(x,y)}{r}\right)^{(s-t)p}\right)\left(\frac{d(x,y)}{r}\right)^{tp},
    \end{align*}
    where we also used that $d(x,y)\leq 8r$ and $t\leq s$. Thus, by \eqref{eqEnergy:LogLemmaFirstEstimateJ}, we get
    \begin{align*}
        J(x,y)\leq -c\eta^q(x)\left(\frac{r}{d(x,y)}\right)^{tp}\left(\log\frac{v(x)}{v(y)}\right)^p+C\left(\frac{d(x,y)}{r}\right)^{(s-t)p}+C\left(\frac{d(x,y)}{r}\right)^{(1-s)p}
    \end{align*}
    as $\eta^q(y)\leq 1$.

    Now, we will consider the case where $v(x)>2v(y)$. Here, we expand $J$ differently, namely
    \begin{align*}
        J(x,y)&=h\left(x,y,\frac{v(x)-v(y)}{d(x,y)^s}\right)\left(\frac{v(x)}{F\left(\frac{v(x)}{r^s}\right)}-\frac{v(y)}{F\left(\frac{v(y)}{r^s}\right)}\right)\frac{\eta^q(y)}{d(x,y)^s} \\
        &\quad +h\left(x,y,\frac{v(x)-v(y)}{d(x,y)^s}\right)\frac{v(x)}{F\left(\frac{v(x)}{r^s}\right)}\frac{\eta^q(x)-\eta^q(y)}{d(x,y)^s}\eqqcolon \tilde{J}_1(x,y)+\tilde{J}_2(x,y).
    \end{align*}
    Using the assumption $v(x)>2v(y)$ and the fact that $\tau\mapsto\frac{F(\tau)}{\tau}$ is non-decreasing, we see that
    \begin{align}
        \tilde{J}_1(x,y)&\leq h\left(x,y,\frac{v(x)-v(y)}{d(x,y)^s}\right)\left(\frac{2v(y)}{F\left(\frac{2v(y)}{r^s}\right)}-\frac{v(y)}{F\left(\frac{v(y)}{r^s}\right)}\right)\frac{\eta^q(y)}{d(x,y)^s} \nonumber \\
        &=h\left(x,y,\frac{v(x)-v(y)}{d(x,y)^s}\right)\frac{v(y)}{F\left(\frac{v(y)}{r^s}\right)}\left(2\frac{F\left(\frac{v(y)}{r^s}\right)}{F\left(\frac{2v(y)}{r^s}\right)}-1\right)\frac{\eta^q(y)}{d(x,y)^s} \nonumber \\
        &\leq -\overline{c}f\left(\frac{v(x)-v(y)}{d(x,y)^s}\right)\frac{v(y)}{F\left(\frac{v(y)}{r^s}\right)}\frac{\eta^q(y)}{d(x,y)^s} \label{eqEnergy:LogLemmaCase2Integrand1}
    \end{align}
    for $\overline{c}=\kappa^{-1}(1-2^{1-p})=\overline{c}(p,\kappa)>0$. Now, we use the inequality
    \begin{align*}
        a^q-b^q\leq\overline{C}\varepsilon b^q+C_\varepsilon|a-b|^q,
    \end{align*}
    that holds true for any $a,b\geq 0$, every $\varepsilon\in(0,1]$ and some $\overline{C}=\overline{C}(q)>0$ and $C_\varepsilon(q,\varepsilon)>0$ (\cite[Lemma 3.1]{DKP16}). Setting $a=\eta(x)$ and $b=\eta(y)$ gives the estimate
    \begin{align}
        \tilde{J}_2(x,y)&\leq\overline{C}\varepsilon\kappa f\left(\frac{v(x)-v(y)}{d(x,y)^s}\right)\frac{v(x)}{F\left(\frac{v(x)}{r^s}\right)}\frac{\eta^q(y)}{d(x,y)^s} \nonumber \\
        &\quad +C_\varepsilon \kappa f\left(\frac{v(x)-v(y)}{d(x,y)^s}\right)\frac{v(x)}{F\left(\frac{v(x)}{r^s}\right)}\frac{|\eta(x)-\eta(y)|}{d(x,y)^s} \nonumber \\
        &\leq -\frac{\tilde{J}_1(x,y)}{2}+Cf\left(\frac{v(x)-v(y)}{d(x,y)^s}\right)\frac{v(x)}{F\left(\frac{v(x)}{r^s}\right)}\frac{|\eta(x)-\eta(y)|^q}{d(x,y)^s}, \label{eqEnergy:LogLemmaCase2Integrand2}
    \end{align}
    where we chose
    \begin{align*}
        \varepsilon=\min\left\{\frac{\overline{c}}{2\overline{C}\kappa},1\right\}
    \end{align*}
    in the last step. As before, we use the properties of $\eta$ and \eqref{eqPre:GrowthAssumption} to get
    \begin{multline*}
        f\left(\frac{v(x)-v(y)}{d(x,y)^s}\right)\frac{v(x)}{F\left(\frac{v(x)}{r^s}\right)}\frac{|\eta(x)-\eta(y)|^q}{d(x,y)^s}\leq Cf\left(\frac{v(x)}{d(x,y)^s}\right)\frac{v(x)}{d(x,y)^s}\frac{1}{F\left(\frac{v(x)}{r^s}\right)} \left(\frac{d(x,y)}{r}\right)^q \\
        \leq C\frac{F\left(\frac{v(x)}{d(x,y)^s}\right)}{F\left(\frac{v(x)}{r^s}\right)}\left(\frac{d(x,y)}{r}\right)^q\leq C\left(\frac{d(x,y)}{r}\right)^{(1-s)q}\leq C\left(\frac{d(x,y)}{r}\right)^{(1-s)p},
    \end{multline*}
    where we also used that
    \begin{align*}
        F\left(\frac{v(x)}{d(x,y)^s}\right)=F\left(\frac{v(x)}{r^s}\left(\frac{r}{d(x,y)}\right)^s\right)\leq CF\left(\frac{v(x)}{r^s}\right)\left(\frac{r}{d(x,y)}\right)^{sq}.
    \end{align*}
    Note also that due to \eqref{eqPre:GrowthAssumption}, we have
    \begin{align*}
        \frac{1}{d(x,y)^s}\frac{v(y)}{F\left(\frac{v(y)}{r^s}\right)}=\frac{\frac{v(y)}{r^s}}{F\left(\frac{v(y)}{r^s}\right)}\left(\frac{r}{d(x,y)}\right)^s\geq\frac{c}{f\left(\frac{v(y)}{r^s}\right)}\left(\frac{r}{d(x,y)}\right)^s
    \end{align*}
    such that it will follow from \eqref{eqEnergy:LogLemmaCase2Integrand1} and \eqref{eqEnergy:LogLemmaCase2Integrand2} that
    \begin{align}\label{eqEnergy:LogLemmaEstimateJSecondCase}
        J(x,y)\leq -cf\left(\frac{v(x)-v(y)}{d(x,y)^s}\right)\frac{\eta^q(y)}{f\left(\frac{v(y)}{r^s}\right)}\left(\frac{r}{d(x,y)}\right)^s+C\left(\frac{d(x,y)}{r}\right)^{(1-s)p}.
    \end{align}
    Let us apply the inequality $\log\tau\leq\tau^\theta/\theta$, which holds for any $\theta>0$ and $\tau\geq 1$. In fact, we can apply it for $\theta=(p-1)/p$, which yields
    \begin{align*}
        \left(\log\frac{v(x)}{v(y)}\right)^p\leq C\left(\frac{v(x)}{v(y)}\right)^{p-1}\leq C\left(\frac{2(v(x)-v(y))}{v(y)}\right)^{p-1}\leq C\left(\frac{v(x)-v(y)}{v(y)}\right)^{p-1},
    \end{align*}
    where we also used that $v(x)>2v(y)$. Since $d(x,y)\leq 8r$ and $v(x)-v(y)>v(y)$, we have
    \begin{align*}
        f\left(\frac{v(y)}{r^s}\right)&=f\left(\frac{v(x)-v(y)}{d(x,y)^s}\left(\frac{d(x,y)}{r}\right)^s\frac{v(y)}{v(x)-v(y)}\right) \\ 
        &\leq Cf\left(\frac{v(x)-v(y)}{d(x,y)^s}\right)\left(\frac{d(x,y)}{r}\right)^{s(p-1)}\left(\frac{v(y)}{v(x)-v(y)}\right)^{p-1}
    \end{align*}
    and therefore,
    \begin{align*}
        \left(\log\frac{v(x)}{v(y)}\right)^p\leq C\frac{f\left(\frac{v(x)-v(y)}{d(x,y)^s}\right)}{f\left(\frac{v(y)}{r^s}\right)}\left(\frac{d(x,y)}{r}\right)^{s(p-1)}\leq C\frac{f\left(\frac{v(x)-v(y)}{d(x,y)^s}\right)}{f\left(\frac{v(y)}{r^s}\right)}\left(\frac{r}{d(x,y)}\right)^s\left(\frac{d(x,y)}{r}\right)^{tp}
    \end{align*}
    as $t\leq s$. By substituting this observation back into \eqref{eqEnergy:LogLemmaEstimateJSecondCase}, we arrive at
    \begin{align*}
        J(x,y)\leq -c\left(\frac{r}{d(x,y)}\right)^{tp}\eta^q(y)\left(\log\frac{v(x)}{v(y)}\right)^p+C\left(\frac{d(x,y)}{r}\right)^{(1-s)p}.
    \end{align*}
    We have now found estimates for $J$ when $v(x)\geq v(y)$. By symmetry, we can conclude that, for any $x,y\in B_{4r}(x_0)$, it is true that
    \begin{align*}
        J(x,y)\leq -c\left(\frac{r}{d(x,y)}\right)^{tp}\min\left\{\eta^q(x),\eta^q(y)\right\}\left|\log\frac{v(x)}{v(y)}\right|^p+C\left(\frac{d(x,y)}{r}\right)^\delta,
    \end{align*}
    where we defined $\delta=\min\{(1-s)p,(s-t)p\}>0$. Therefore,
    \begin{align*}
        I&=\int_{B_{4r}(x_0)}\int_{B_{4r}(x_0)}J(x,y)\nu(x,\dy)\mu(\dx) \\
        &\leq -c\int_{B_{4r}(x_0)}\int_{B_{4r}(x_0)}\left(\frac{r}{d(x,y)}\right)^{tp}\min\left\{\eta^q(x),\eta^q(y)\right\}\log\left|\frac{v(x)}{v(y)}\right|\nu(x,\dy)\mu(\dx) \\
        &\quad +C\int_{B_{4r}(x_0)}\int_{B_{4r}(x_0)}\left(\frac{d(x,y)}{r}\right)^\delta\nu(x,\dy)\mu(\dx).
    \end{align*}
    The second term on the right hand side can directly be estimated by $C\mu(B_r(x_0))$ using Lemma \ref{lemPre:distance_integrability_nu} and the doubling property. Thus, we may write
    \begin{align}
        &\int_{B_r(x_0)}\int_{B_r(x_0)}\left|\log\frac{v(x)}{v(y)}\right|^p\frac{\nu(x,\dy)}{d(x,y)^{tp}}\mu(\dx) \nonumber \\
        &\quad\leq r^{-tp}\int_{B_{4r}(x_0)}\int_{B_{4r}(x_0)}\left(\frac{r}{d(x,y)}\right)^{tp}\min\left\{\eta^q(x),\eta^q(y)\right\}\left|\log\frac{v(x)}{v(y)}\right|^p\nu(x,\dy)\mu(\dx) \nonumber \\
        &\quad\leq\frac{C}{r^{tp}}(\mu(B_r(x_0))-I). \label{eqEnergy:LogLemmaEstimate1}
    \end{align}
    
    Next, let us estimate $II$. For any $x\in B_{2r}(x_0)$ and $y\in X\setminus B_{4r}(x_0)$, we have $u(x)\geq 0$ such that
    \begin{align*}
        h\left(x,y,\frac{u(x)-u(y)}{d(x,y)^s}\right)\leq\kappa f\left(\frac{(u(x)-u(y))_+}{d(x,y)^s}\right)\leq Cf\left(\frac{u(x)}{d(x,y)^s}\right)+Cf\left(\frac{u_-(y)}{d(x,y)^s}\right)
    \end{align*}
    by \eqref{eqPre:Comparability}. Moreover, \eqref{eqPre:GrowthAssumption} implies that
    \begin{align*}
        \frac{v(x)}{F\left(\frac{v(x)}{r^s}\right)}=\frac{\frac{v(x)}{r^s}}{F\left(\frac{v(x)}{r^s}\right)}r^s\leq q\frac{r^s}{f\left(\frac{v(x)}{r^s}\right)}.
    \end{align*}
    Since $\eta$ vanishes outside $B_{2r}(x_0)$, we have seen that
    \begin{align*}
        II&\leq C\int_{B_{2r(x_0)}}\int_{X\setminus B_{4r}(x_0)}f\left(\frac{u(x)}{d(x,y)^s}\right)\frac{1}{f\left(\frac{v(x)}{r^s}\right)}\left(\frac{r}{d(x,y)}\right)^s\nu(x,\dy)\mu(\dx) \\
        &\quad +C\int_{B_{4r}(x_0)}\int_{X\setminus B_{4r}(x_0)}f\left(\frac{u_-(y)}{d(x,y)^s}\right)\frac{1}{f\left(\frac{v(x)}{r^s}\right)}\left(\frac{r}{d(x,y)}\right)^s\nu(x,\dy)\mu(\dx)\eqqcolon II_1+II_2.
    \end{align*}
    Since $u(x)\leq v(x)$, we can estimate
    \begin{align*}
        II_1&\leq C\int_{B_{2r}(x_0)}\int_{X\setminus B_{4r}(x_0)}\frac{f\left(\frac{v(x)}{d(x,y)^s}\right)}{f\left(\frac{v(x)}{r^s}\right)}\left(\frac{r}{d(x,y)}\right)^s\nu(x,\dy)\mu(\dx) \\
        &=C\int_{B_{2r}(x_0)}\int_{X\setminus B_{4r}(x_0)}\frac{f\left(\frac{v(x)}{r^s}\left(\frac{r}{d(x,y)}\right)^s\right)}{f\left(\frac{v(x)}{r^s}\right)}\left(\frac{r}{d(x,y)}\right)^s\nu(x,\dy)\mu(\dx) \\
        &\leq Cr^{sp}\int_{B_{2r}(x_0)}\int_{X\setminus B_{2r}(x)}\frac{\nu(x,\dy)}{d(x,y)^{sp}}\mu(\dx)\leq C\mu(B_r(x_0))
    \end{align*}
    by \eqref{eqPre:TailIntegrabilityAssumption} and the doubling property. For $II_2$, we just have to notice that $u(y)\geq 0$ if $y\in B_R(x_0)$ such that
    \begin{align*}
        II_2&=C\int_{B_{2r}(x_0)}\int_{X\setminus B_R(x_0)}f\left(\frac{u_-(y)}{d(x,y)^s}\right)\frac{1}{f\left(\frac{v(x)}{r^s}\right)}\left(\frac{r}{d(x,y)}\right)^s\nu(x,\dy)\mu(\dx) \\
        &\leq\frac{Cr^s}{f\left(\frac{a}{r^s}\right)}\int_{B_{2r}(x_0)}\int_{X\setminus B_R(x_0)}f\left(\frac{u_-(y)}{d(x,y)^s}\right)\frac{\nu(x,\dy)}{d(x,y)^s}\mu(\dx)\leq\frac{Cr^s}{f\left(\frac{a}{r^s}\right)}\mu(B_r(x_0))\tail{u_-}{x_0}{R},
    \end{align*}
    since $B_{2r}(x_0)\subseteq B_{R/2}(x_0)$ due to the assumption $r\leq R/4$. Hence, we found that
    \begin{align}\label{eqEnergy:LogLemmaEstimate2}
        II\leq C\mu(B_r(x_0))\left(1+\frac{r^s}{f\left(\frac{a}{r^s}\right)}\tail{u_-}{x_0}{R}\right).
    \end{align}

    Let us finally estimate $III$ by applying Lemma \ref{lemPre:SymmetryFunctions}. By similar arguments as already performed above, we see that
    \begin{align*}
        &\int_{B_{2r}(x_0)}\int_{X\setminus B_{4r}(x_0)}\left|h\left(x,y,\frac{u(x)-u(y)}{d(x,y)^s}\right)\varphi(x)\right|\frac{\nu(x,\dy)}{d(x,y)^s}\mu(\dx) \\
        &\quad\leq C\int_{B_{2r}(x_0)}\int_{X\setminus B_{4r}(x_0)}f\left(\frac{u(x)}{d(x,y)^s}\right)\frac{v(x)}{F\left(\frac{v(x)}{r^s}\right)}\frac{\nu(x,\dy)}{d(x,y)^s}\mu(\dx) \\
        &\quad\quad +C\int_{B_{2r}(x_0)}\int_{X\setminus B_{4r}(x_0)}f\left(\frac{|u(y)|}{d(x,y)^s}\right)\frac{v(x)}{F\left(\frac{v(x)}{r^s}\right)}\frac{\nu(x,\dy)}{d(x,y)^s}\mu(\dx) \\
        &\quad\leq C\mu(B_r(x_0))\left(1+\frac{r^s}{f\left(\frac{a}{r^s}\right)}\tail{u}{x_0}{4r}\right)<\infty.
    \end{align*}
    Hence, we can indeed apply Lemma \ref{lemPre:SymmetryFunctions} to obtain that $III=II$ using the symmetry of $h$. 

    We can conclude the proof by combining \eqref{eqEnergy:LogLemmaEstimate1}, \eqref{eqEnergy:LogLemmaTest} and \eqref{eqEnergy:LogLemmaEstimate2} to obtain
    \begin{align*}
        \fint_{B_r(x_0)}\int_{B_r(x_0)}\left|\log\frac{v(x)}{v(y)}\right|^p\frac{\nu(x,\dy)}{d(x,y)^{tp}}\mu(\dx)&\leq\frac{C}{r^{tp}}\frac{\mu(B_r(x_0))-I}{\mu(B_r(x_0))}\leq\frac{C}{r^{tp}}\left(1+\frac{II+III}{\mu(B_r(x_0))}\right) \\
        &\leq\frac{C}{r^{tp}}\left(1+\frac{r^s}{f\left(\frac{a}{r^s}\right)}\tail{u_-}{x_0}{R}\right).
    \end{align*}
\end{proof}

Since we do not invoke \eqref{eqPre:PoincareAssumption} in the proof of Lemma \ref{lemEnergy:LogEstimate}, the statement holds true for any $t\in[0,s)$.

Inspired by the method of \cite{Coz17}, the authors of \cite{CKW22} completely bypassed the logarithmic estimate. To the best of our knowledge, this method relies heavily on a pointwise lower bound on the kernel. Since we are not assuming any pointwise bounds on $\nu(x,\cdot)$, not even the existence of a density with respect to $\mu$, we used the logarithmic estimate.

\section{Expansion of Positivity}\label{sec:EoP}

The key ingredient for the proof of Hölder regularity via De Giorgi's technique is an expansion of positivity lemma, which forces a decay of oscillation on a smaller ball and in turn, leads to Hölder continuity. Together with the logarithmic estimate, we have to perform a De Giorgi iteration in order to prove Lemma \ref{lemEoP:CriticalMass} below, which one might call a critical mass lemma. We use the well-known fast convergence lemma of De Giorgi.

\begin{lemma}[{\cite[Lemma 7.1]{Giu03}}]\label{lemEoP:DeGiorgiFastConvergence}
    Let $\{A_i\}_{i=0}^\infty$ be a sequence of nonnegative reals, $\overline{C},t>0$ and $\beta>1$. If
    \begin{align*}
        A_{i+1}\leq\overline{C}\beta^iA_i^{1+t} \enspace\text{for all}\  i\geq 0 \and A_0\leq\overline{C}^{-1/t}\beta^{-1/t^2},
    \end{align*}
    then $A_i\to 0$ as $i\to\infty$.
\end{lemma}

\begin{lemma} \label{lemEoP:CriticalMass}
    Fix $x_0\in X$, $0<R\leq\frac{R_0}{2\lambda}$ and $a>0$. Let $u\in W^{s,F,\nu}(B_R(x_0))\cap L_{s,\nu}^f(X)$ be a weak supersolution to $\L u=0$ in $B_R(x_0)$ such that $u\geq 0$ in $B_R(x_0)$. Moreover, let $0<r\leq R/2$. There exists some $\gamma=\gamma(\C,s,t,p,q,\kappa,\lambda,C_1,C_2)>0$ such that if
    \begin{align}\label{eqEoP:CriticalMassAssumption}
        \mu(\{u\leq 2a\}\cap B_{2r}(x_0))\leq\gamma\mu(B_{2r}(x_0)),
    \end{align}
    then
    \begin{align}\label{eqEoP:CriticalMassTailResult}
        \tail{u_-}{x_0}{R}>\frac{1}{r^s}f\left(\frac{a}{r^s}\right)
    \end{align}
    or
    \begin{align}\label{eqEoP:CriticalMassToProve}
        \inf_{B_r(x_0)}u\geq a.
    \end{align}
\end{lemma}

\begin{proof}
    If \eqref{eqEoP:CriticalMassTailResult} holds true, then there is nothing to prove. Let us suppose that
    \begin{align}\label{eqEoP:CriticalMassTailAssumption}
        \tail{u_-}{x_0}{R}\leq\frac{1}{r^s} f\left(\frac{a}{r^s}\right)
    \end{align}
    and show that \eqref{eqEoP:CriticalMassToProve} holds. Define sequences
    \begin{align*}
        r_i=r(1+2^{-i}), \quad B_i=B_{r_i}(x_0) \and \alpha_i=a(1+2^{-i})
    \end{align*}
    as well as
    \begin{align*}
        \rho_i=\frac{r_i+r_{i+1}}{2}, \quad B^i=B_{\rho_i}(x_0) \and \beta_i=\frac{\alpha_i+\alpha_{i+1}}{2}
    \end{align*}
    for $i\geq 0$. In particular, we obtain the relations
    \begin{align*}
        r<r_{i+1}<\rho_i<r_i\leq 2r \and a<\alpha_{i+1}<\beta_i<\alpha_i\leq 2a,
    \end{align*}
    which we will use frequently. Moreover, let $\eta_i\colon X\to[0,1]$ be a Lipschitz function such that
    \begin{align*}
        \indicator_{B_{i+1}}\leq\eta_i\leq\indicator_{B^i} \and \sup_{\substack{x,y\in X\\x\neq y}}\frac{|\eta_i(x)-\eta_i(y)|}{d(x,y)}\leq\frac{1}{\rho_i-r_{i+1}}\leq\frac{1}{\frac{r_i-r_{i+1}}{2}}\leq C\cdot\frac{2^i}{r}.
    \end{align*}
    We introduce the truncations
    \begin{align*}
        w_i=(u-\alpha_i)_- \and v_i=(u-\beta_i)_-.
    \end{align*}
    Let us first apply Lemma \ref{lemEnergy:Caccioppoli} to $v_i$ and $\eta_i$. If we divide both sides of \eqref{eqEnergy:Caccioppoli} by $\mu(B_i)$, we obtain
    \begin{align}
        I&\coloneqq\fint_{B_i}\int_{B_i} F\left(\frac{|v_i(x)\eta_i^q(x)-v_i(y)\eta_i^q(y)|}{d(x,y)^s}\right)\nu(x,\dy)\mu(\dx) \nonumber \\
        &\leq C\fint_{B_i}\int_{B_i} F\left(\frac{|\eta_i(x)-\eta_i(y)|}{d(x,y)^s}\max\{v_i(x),v_i(y)\}\right)\nu(x,\dy)\mu(\dx) \nonumber \\
        &\quad +C\fint_{B_i} v_i(x)\eta_i^q(x)\mu(\dx)\sup_{x\in B^i}\int_{X\setminus B_i} f\left(\frac{v_i(y)}{d(x,y)^s}\right)\frac{\nu(x,\dy)}{d(x,y)^s}\eqqcolon II+III.\label{eqEoP:CriticalMassCaccioppoli}
    \end{align}

    Let us estimate the appearing quantities further. First, we can apply Theorem \ref{thmSobPoi:SobolevPoincare} to the function $v_i\eta_i^q$ in $B_i$, i.e.
    \begin{align}
        &\fint_{B_{i+1}} F^\theta\left(\frac{v_i}{r^s}\right)\d\mu\leq C\fint_{B_i} F^\theta\left(\frac{|v_i\eta_i^q-(v_i\eta_i^q)_{B_i}|}{r_i^s}\right)\d\mu+CF^\theta\left(\frac{(v_i\eta_i^q)_{B_i}}{r^s}\right) \nonumber \\
        &\quad\leq C\left(\fint_{\Lambda B_i}\int_{\Lambda B_i} F\left(\frac{|v_i(x)\eta_i^q(x)-v_i(y)\eta_i^q(y)|}{d(x,y)^s}\right)\nu(x,\dy)\mu(\dx)\right)^\theta+C\left(\fint_{B_i}F\left(\frac{v_i\eta_i^q}{r^s}\right)\d\mu\right)^\theta \nonumber \\
        &\quad\leq C\left(I+\fint_{B_i}\int_{X\setminus B_i}F\left(\frac{v_i(x)\eta_i^q(x)}{d(x,y)^s}\right)\nu(x,\dy)\mu(\dx)\right)^\theta+C\left(\fint_{B_i}F\left(\frac{v_i}{r^s}\right)\d\mu\right)^\theta \nonumber \\
        &\quad\leq C\left(I+2^{iq}\fint_{B_i}F\left(\frac{v_i}{r^s}\right)\d\mu\right)^\theta \label{eqEoP:CriticalMassCaccioppoliTerm1}
    \end{align}
    for $\theta>1$, e.g.\ choose $\theta=\frac{Q}{Q-s/2}$. In addition to Jensen's inequality, we have used that $\eta_i\leq 1$ is supported in $B_{\rho_i}(x_0)$ and thus,
    \begin{align*}
        &\fint_{B_i}\int_{X\setminus B_i}F\left(\frac{|v_i(x)\eta_i^q(x)|}{d(x,y)^s}\right)\nu(x,\dy)\mu(\dx) \\
        &\quad\leq C\fint_{B_i}\int_{X\setminus B_{r_i-\rho_i}(x)}F\left(\frac{|v_i(x)\eta_i^q(x)|}{(r_i-\rho_i)^s}\left(\frac{r_i-\rho_i}{d(x,y)}\right)^s\right)\nu(x,\dy)\mu(\dx) \\
        &\quad\leq C(r_i-\rho_i)^{sp}\fint_{B_i}F\left(2^{is}\frac{v_i(x)}{r^s}\right)\int_{X\setminus B_{r_i-\rho_i}(x)}\frac{\nu(x,\dy)}{d(x,y)^{sp}}\mu(\dx)\leq C\cdot 2^{iq}\fint_{B_i}F\left(\frac{v_i}{r^s}\right)\d\mu
    \end{align*}
    by \eqref{eqPre:TailIntegrabilityAssumption}. For the right hand side of \eqref{eqEoP:CriticalMassCaccioppoli}, we can estimate
    \begin{align*}
        II&=\fint_{B_i}\int_{B_i} F\left(\frac{C\frac{2^i}{r}d(x,y)}{d(x,y)^s}\max\{v_i(x),v_i(y)\}\right)\nu(x,\dy)\mu(\dx) \\
        &\leq C\fint_{B_i}\int_{B_i} F\left(\frac{2^i}{r}d(x,y)^{1-s}v_i(x)\right)\nu(x,\dy)\mu(\dx) \\
        &\leq C\cdot 2^{iq}\fint_{B_i}\int_{B_i}F\left(\frac{v_i(x)}{r^s}\left(\frac{d(x,y)}{r}\right)^{1-s}\right)\nu(x,\dy)\mu(\dx).
    \end{align*}
    Note that
    \begin{align*}
        \frac{d(x,y)}{r}\leq\frac{2r_i}{r}\leq\frac{2\cdot 2r}{r}=4 \enspace\text{for}\ x,y\in B_i
    \end{align*}
    such that, by Lemma \ref{lemPre:distance_integrability_nu},
    \begin{align}
        II&\leq C\cdot 2^{iq}\fint_{B_i}F\left(\frac{v_i(x)}{r^s}\right)\int_{B_i}\left(\frac{d(x,y)}{r}\right)^{(1-s)p}\nu(x,\dy)\mu(\dx) \nonumber \\
        &\leq C\cdot 2^{iq}\frac{r_i^{(1-s)p}}{r^{(1-s)p}}\fint_{B_i}F\left(\frac{v_i}{r^s}\right)\d\mu\leq C\cdot 2^{iq}\fint_{B_i}F\left(\frac{w_i}{r^s}\right)\d\mu.\label{eqEoP:CriticalMassCaccioppoliTerm2}
    \end{align}
    In the last step, we also used that $v_i\leq w_i$. Let us now estimate $III$. For any $x\in B^i$, we have
    \begin{align*}
        &\int_{X\setminus B_i} f\left(\frac{v_i(y)}{d(x,y)^s}\right)\frac{\nu(x,\dy)}{d(x,y)^s}\leq\int_{X\setminus B_i} f\left(\frac{|u_-(y)|+\beta_i}{d(x,y)^s}\right)\frac{\nu(x,\dy)}{d(x,y)^s} \\
        &\quad\leq C\left(\int_{X\setminus B_i} f\left(\frac{|u_-(y)|}{d(x,y)^s}\right)\frac{\nu(x,\dy)}{d(x,y)^s}+\int_{X\setminus B_i}f\left(\frac{\beta_i}{d(x,y)^s}\right)\frac{\nu(x,\dy)}{d(x,y)^s}\right) \\
        &\quad\leq C\left(\int_{X\setminus B_R(x_0)}f\left(\frac{|u_-(y)|}{d(x,y)^s}\right)\frac{\nu(x,\dy)}{d(x,y)^s}+\int_{X\setminus B_{r_i-\rho_i}(x)}f\left(\frac{2a}{(r_i-\rho_i)^s}\left(\frac{r_i-\rho_i}{d(x,y)}\right)^s\right)\frac{\nu(x,\dy)}{d(x,y)^s}\right) \\
        &\quad\leq C\tail{u_-}{x_0}{R}+Cf\left(\frac{a}{(r_i-\rho_i)^s}\right)(r_i-\rho_i)^{s(p-1)}\int_{X\setminus B_{r_i-\rho_i}(x)}\frac{\nu(x,\dy)}{d(x,y)^{sp}} \\
        &\quad\leq\frac{C}{r^s}f\left(\frac{a}{r^s}\right)+\frac{C}{(r_i-\rho_i)^s}f\left(\frac{a}{(r_i-\rho_i)^s}\right)
    \end{align*}
    by \eqref{eqEoP:CriticalMassTailAssumption}. Since $r_i-\rho_i=(r_i-r_{i+1})/2=2^{-i-2}r$, we get
    \begin{align*}
        \int_{X\setminus B_i}f\left(\frac{v_i(x)}{d(x,y)^s}\right)\frac{\nu(x,\dy)}{d(x,y)^s}\leq\frac{C}{r^s}f\left(\frac{a}{r^s}\right)+\frac{C}{r^s}\cdot 2^{isq}f\left(\frac{a}{r^s}\right)\leq\frac{C}{r^s}\cdot 2^{iq}f\left(\frac{a}{r^s}\right).
    \end{align*}  
    Note that, on the set $\{v_i>0\}=\{u<\beta_i\}$, it is true that
    \begin{align*}
        w_i=(u-\alpha_i)_-=\alpha_i-u>\alpha_i-\beta_i=\frac{\alpha_i-\alpha_{i+1}}{2}=\frac{a}{2^{i+2}}    
    \end{align*}
    Hence, as $v_i\leq w_i$, we obtain
    \begin{align}
        III&\leq\frac{C}{r^s}\cdot 2^{iq}f\left(\frac{a}{r^s}\right)\fint_{B_i}v_i\d\mu\leq C\cdot 2^{iq}\fint_{B_i}f\left(\frac{2^{i+2}w_i}{r^s}\right)\frac{w_i}{r^s}\d\mu \nonumber \\
        &\leq C\cdot 2^{iq}\cdot 2^{i(q-1)}\fint_{B_i}f\left(\frac{w_i}{r^s}\right)\frac{w_i}{r^s}\d\mu\leq C\cdot 4^{iq}\fint_{B_i}F\left(\frac{w_i}{r^s}\right)\d\mu,\label{eqEoP:CriticalMassCaccioppoliTerm3}
    \end{align}
    where we also used \eqref{eqPre:GrowthAssumption}.

    We are now in a good position to combine the estimates \eqref{eqEoP:CriticalMassCaccioppoli}, \eqref{eqEoP:CriticalMassCaccioppoliTerm1}, \eqref{eqEoP:CriticalMassCaccioppoliTerm2} and \eqref{eqEoP:CriticalMassCaccioppoliTerm3}, which will yield
    \begin{align}
        \fint_{B_{i+1}} F^\theta\left(\frac{v_i}{r^s}\right)\d\mu&\leq C\left(I+2^{iq}\fint_{B_i}F\left(\frac{v_i}{r^s}\right)\d\mu\right)^\theta \nonumber \\
        &\leq C\left(II+III+2^{iq}\fint_{B_i}F\left(\frac{w_i}{r^s}\right)\d\mu\right)^\theta\leq C\cdot 4^{iq\theta}\left(\fint_{B_i}F\left(\frac{w_i}{r^s}\right)\d\mu\right)^\theta. \label{eqEoP:CriticalMassCombinationOfEstimates}
    \end{align}
    The integral on the right hand side can be estimated by
    \begin{align*}
        \fint_{B_i}F\left(\frac{w_i}{r^s}\right)\d\mu=\frac{1}{\mu(B_i)}\int_{B_i\cap\{u<\alpha_i\}}F\left(\frac{w_i}{r^s}\right)\d\mu\leq CF\left(\frac{a}{r^s}\right)\frac{\mu(B_i\cap\{u\leq\alpha_i\})}{\mu(B_i)},
    \end{align*}
    since $u\geq 0$ in $B_i$ and therefore, $w_i\leq \alpha_i\leq 2a$. For the left hand side, we estimate
    \begin{align*}
        \fint_{B_{i+1}}F^\theta\left(\frac{v_i}{r^s}\right)\d\mu&\geq\frac{1}{\mu(B_{i+1})}\int_{B_{i+1}\cap\{u\leq\alpha_{i+1}\}}F^\theta\left(\frac{\beta_i-u}{r^s}\right)\d\mu \\
        &\geq\frac{1}{\mu(B_{i+1})}\int_{B_{i+1}\cap\{u\leq\alpha_{i+1}\}}F^\theta\left(\frac{\beta_i-\alpha_{i+1}}{r^s}\right)\d\mu \\
        &\geq\frac{c}{2^{iq\theta}}F^\theta\left(\frac{a}{r^s}\right)\frac{\mu(B_{i+1}\cap\{u\leq\alpha_{i+1}\})}{\mu(B_{i+1})}
    \end{align*}
    by using that
    \begin{align*}
        F^\theta\left(\frac{\beta_i-\alpha_{i+1}}{r^s}\right)=F^\theta\left(\frac{\alpha_i-\alpha_{i+1}}{2r^s}\right)=F^\theta\left(\frac{a}{2^{i+2}r^s}\right)\geq\frac{c}{2^{iq\theta}}F^\theta\left(\frac{a}{r^s}\right)
    \end{align*}
    Let us substitute those estimates back into \eqref{eqEoP:CriticalMassCombinationOfEstimates}, which gives
    \begin{align*}
        \frac{c}{2^{iq\theta}}F^\theta\left(\frac{a}{r^s}\right)\frac{\mu(B_{i+1}\cap\{u\leq\alpha_{i+1}\})}{\mu(B_{i+1})}\leq C\cdot 4^{iq\theta}\left(F\left(\frac{a}{r^s}\right)\frac{\mu(B_i\cap\{u\leq\alpha_i\})}{\mu(B_i)}\right)^\theta
    \end{align*}
    or
    \begin{align*}
        A_{i+1}\leq\overline{C}\cdot 8^{iq\theta}A_i^\theta
    \end{align*}
    for some $\overline{C}=\overline{C}(\C,s,t,p,q,\lambda,\kappa,C_1,C_2)>0$ after defining
    \begin{align*}
        A_i=\frac{\mu(B_i\cap\{u\leq\alpha_i\})}{\mu(B_i)}.
    \end{align*}
    Since $\theta>1$, we deduce from Lemma \ref{lemEoP:DeGiorgiFastConvergence} that $A_i\to 0$, provided
    \begin{align*}
        A_0\leq \overline{C}^\frac{-1}{\theta-1}\cdot 8^{-\frac{q\theta}{(\theta-1)^2}}.
    \end{align*}
    Since
    \begin{align*}
        A_0=\frac{\mu(B_0\cap\{u\leq\alpha_0\})}{\mu(B_0)}=\frac{\mu(B_{2r}(x_0)\cap\{u\leq 2a\})}{\mu(B_{2r}(x_0))}\leq\gamma
    \end{align*}
    due to \eqref{eqEoP:CriticalMassAssumption}, it suffices to choose
    \begin{align*}
        \gamma=\overline{C}^\frac{-1}{\theta-1}\cdot 8^{-\frac{q\theta}{(\theta-1)^2}}=\gamma(\C,s,t,p,q,\kappa,\lambda,C_1,C_2)>0.
    \end{align*}
    Finally, by continuity of measure, we get that
    \begin{align*}
        \frac{\mu(B_r(x_0)\cap\{u\leq a\})}{\mu(B_r(x_0))}=0.
    \end{align*}
    In other words, $u\geq a$ holds true $\mu$-almost everywhere on $B_r(x_0)$.    
\end{proof}

\begin{lemma}[Expansion of positivity] \label{lemEoP:ExpansionOfPositivity}
    Fix $x_0\in X$, $0<R\leq\frac{R_0}{2\lambda}$ and $a>0$. Let $u\in W^{s,F,\nu}(B_R(x_0))\cap L_{s,\nu}^f(X)$ be a weak supersolution to $\L u=0$ in $B_R(x_0)$ such that $u\geq 0$ in $B_R(x_0)$. Let $0<r\leq\frac{R}{8\lambda}$ and assume that
    \begin{align}\label{eqEoP:ExpansionOfPositivityAssumption}
        \mu(\{u\geq a\}\cap B_{4r}(x_0))\geq\varepsilon\mu(B_{4r}(x_0))
    \end{align}
    for some $\varepsilon\in(0,1)$. Then, there exists some $\delta=\delta(\C,s,t,p,q,\kappa,\lambda,C_1,C_2,\varepsilon)>0$ such that
    \begin{align*}
        \tail{u_-}{x_0}{R}>\frac{1}{r^s}f\left(\frac{\delta a}{r^s}\right)
    \end{align*}
    or
    \begin{align} \label{eqEoP:ExpansionOfPositivityToProve}
        \inf_{B_r(x_0)}u\geq\delta a.
    \end{align}
\end{lemma}

\begin{proof}
    Let us assume
    \begin{align} \label{eqEoP:ExpansionOfPositivityTailAssumption}
        \tail{u_-}{x_0}{R}\leq\frac{1}{r^s} f\left(\frac{\delta a}{r^s}\right),
    \end{align}
    so it suffices to prove \eqref{eqEoP:ExpansionOfPositivityToProve}. To this end, let $\delta>0$ be arbitrary and set
    \begin{align*}
        v(x)=\left(\log\frac{(1+\delta)a}{u(x)+\delta a}\right)_+, \quad x\in B_R(x_0).
    \end{align*}
    It is easy to check that $v\in W^{s,F,\nu}(B_R(x_0))\cap L_{s,\nu}^f(X)$ and moreover,
    \begin{align*}
        |v(x)-v(y)|\leq\left|\log\frac{u(x)+\delta a}{u(y)+\delta a}\right|
    \end{align*}
    for all $x,y\in X$. Hence, we may apply Lemma \ref{lemEnergy:LogEstimate} to obtain
    \begin{align*}
        \fint_{B_{2\lambda r}(x_0)}\int_{B_{2\lambda r}(x_0)}\frac{|v(x)-v(y)|^p}{d(x,y)^{tp}}\nu(x,\dy)\mu(\dx)\leq\frac{C}{r^{tp}}\left(1+\frac{r^s}{f\left(\frac{\delta a}{r^s}\right)}\tail{u_-}{x_0}{R}\right).
    \end{align*}
    We will combine this estimate with \eqref{eqPre:PoincareAssumption} and \eqref{eqEoP:ExpansionOfPositivityTailAssumption}, which implies
    \begin{align*}
        \fint_{B_{2r}(x_0)}|v-(v)_{B_{2r}(x_0)}|\d\mu\leq C.
    \end{align*}
    By recalling the definition of $v$, we may write
    \begin{align*}
        \frac{\mu(\{u\geq a\}\cap B_{2r}(x_0))}{\mu(B_{2r}(x_0))}(v)_{B_{2r}(x_0)}&=\frac{\mu(\{v=0\}\cap B_{2r}(x_0))}{\mu(B_{2r}(x_0))}|(v)_{B_{2r}(x_0)}| \\
        &=\frac{1}{\mu(B_{2r}(x_0))}\int_{\{v=0\}\cap B_{2r}(x_0)}|v(x)-(v)_{B_{2r}(x_0)}|\mu(\dx) \\
        &\leq\fint_{B_{2r}(x_0)}|v(x)-(v)_{B_{2r}(x_0)}|\mu(\dx)\leq C.
    \end{align*}
    By \eqref{eqEoP:ExpansionOfPositivityAssumption}, we get
    \begin{align*}
        \fint_{B_{2r}(x_0)}v(x)\mu(\dx)\leq\frac{\overline{C}}{\varepsilon}
    \end{align*}
    for some $\overline{C}=\overline{C}(\C,p,q,\lambda,C_1,C_2)>0$. On the other hand, we have
    \begin{align*}
        \fint_{B_{2r}(x_0)}v(x)\mu(\dx)&\geq\frac{1}{\mu(B_{2r}(x_0))}\int_{\{u\leq 2\delta a\}\cap B_{2r}(x_0)}\log\frac{(1+\delta)a}{u(x)+\delta a}\mu(\dx) \\
        &\geq\frac{\mu(\{u\leq 2\delta a\}\cap B_{2r}(x_0))}{\mu(B_{2r}(x_0))}\log\frac{1+\delta}{3\delta}.
    \end{align*}
    Let $\gamma=\gamma(\C,s,t,p,q,\kappa,\lambda,C_1,C_2)>0$ be the parameter from Lemma \ref{lemEoP:CriticalMass} and choose
    \begin{align*}
        \delta=\frac{1}{3e^\frac{\overline{C}\gamma}{\varepsilon}-1}=\delta(\C,s,t,p,q,\kappa,\lambda,C_1,C_2,\varepsilon)>0.
    \end{align*}
    Then
    \begin{align*}
        \mu(\{u\leq 2\delta a\}\cap B_{2r}(x_0))\leq\frac{\overline{C}}{\varepsilon}\mu(B_{2r}(x_0))\log\frac{1+\delta}{3\delta}\leq\gamma\mu(B_{2r}(x_0))
    \end{align*}
    and Lemma \ref{lemEoP:CriticalMass} together with \eqref{eqEoP:ExpansionOfPositivityTailAssumption} implies that
    \begin{align*}
        \inf_{B_r(x_0)}u\geq\delta a
    \end{align*}
    as desired.
\end{proof}

\section{Hölder continuity}\label{sec:Hoelder}

\begin{lemma} \label{lemHoelder:OscillationDiminish}
    Fix $x_0\in X$ and $0<R\leq\frac{R_0}{2\lambda}$. Let $u\in W^{s,F,\nu}(B_R(x_0))\cap L_{s,\nu}^f(X)\cap L^\infty(B_R(x_0))$ be a bounded weak solution to $\L u=0$ in $B_R(x_0)$. Then, for any $r\in(0,R]$, it is true that
    \begin{align}\label{eqHoelder:OscillationDiminish}
        \osc_{B_r(x_0)}u\leq C\left(\frac{r}{R}\right)^\alpha\left(\|u\|_{L^\infty(B_R(x_0))}+\Tail{u}{x_0}{R}\right)
    \end{align}
    for some $\alpha=\alpha(\C,s,t,p,q,\kappa,\lambda,C_1,C_2)\in(0,1)$ and $C=C(\C,s,t,p,q,\kappa,\lambda,C_1,C_2)>0$.
\end{lemma}

\begin{proof}
    Let us fix some parameters
    \begin{align*}
        0<\alpha\leq\min\left\{\frac{sp}{2(p-1)},s\right\} \and 0<\sigma\leq\frac{1}{8\lambda}
    \end{align*}
    that will be chosen during the proof. For $j\geq 0$, define
    \begin{align*}
        r_j=\sigma^jR, \quad  B_j=B_{r_j}(x_0) \and \omega_j=\sigma^{\alpha j}\left(\|u\|_{L^\infty(B_R(x_0))}+\Tail{u}{x_0}{R}\right).
    \end{align*}
    We will prove by induction that
    \begin{align}\label{eqHoelder:ProofMainInductionGoal}
        \osc_{B_j}u\leq 2\omega_j
    \end{align}
    for any $j\geq 0$. Once we have proved \eqref{eqHoelder:ProofMainInductionGoal}, we can conclude the proof. Indeed, given $r\in(0,R]$, choose the integer $N\geq 0$ with
    \begin{align*}
        r_{N+1}<r\leq r_N
    \end{align*}
    and see that, by \eqref{eqHoelder:ProofMainInductionGoal},
    \begin{align*}
        \osc_{B_r(x_0)}u&\leq\osc_{B_N}u\leq 2\omega_N=2\sigma^{\alpha N}\left(\|u\|_{L^\infty(B_R(x_0))}+\Tail{u}{x_0}{R}\right) \\
        &=2\sigma^{\alpha(N+1)}\sigma^{-\alpha}\left(\|u\|_{L^\infty(B_R(x_0))}+\Tail{u}{x_0}{R}\right) \\
        &\leq 2\left(\frac{r}{R}\right)^\alpha\sigma^{-\alpha}\left(\|u\|_{L^\infty(B_R(x_0))}+\Tail{u}{x_0}{R}\right),
    \end{align*}
    which proves \eqref{eqHoelder:OscillationDiminish} for $C=2\sigma^{-\alpha}$.
    
    Now, let us prove \eqref{eqHoelder:ProofMainInductionGoal}. The case $j=0$ is trivial as
    \begin{align*}
        \osc_{B_0}u\leq 2\|u\|_{L^\infty(B_R(x_0))}\leq 2\omega_0.
    \end{align*}
    Assume that, for some $j\geq 0$,
    \begin{align}\label{eqHoelder:ProofMainInductionHypothesis}
        \osc_{B_i}u\leq 2\omega_i \quad\forall i\in\{0,\ldots, j\}.
    \end{align}
    We need to prove that $\osc_{B_{j+1}}u\leq 2\omega_{j+1}$. Note that at least one of the inequalities
    \begin{align}\label{eqHoelder:ProofMainPoss1}
        \mu(\{u-\inf_{B_j}u\geq\omega_j\}\cap 2B_{j+1})&\geq\frac{1}{2}\mu(2B_{j+1})
    \end{align}
    and
    \begin{align}\label{eqHoelder:ProofMainPoss2}
        \mu(\{u-\inf_{B_j}u\leq\omega_j\}\cap 2B_{j+1})&\geq\frac{1}{2}\mu(2B_{j+1})
    \end{align}
    must hold. We define
    \begin{align*}
        v=\begin{cases}
            u-\inf_{B_j}u, &\ \text{if \eqref{eqHoelder:ProofMainPoss1} holds true,} \\
            2\omega_j-\left(u-\inf_{B_j}u\right), &\ \text{if \eqref{eqHoelder:ProofMainPoss2} holds true.}
        \end{cases}
    \end{align*} 
    Then, in any case, we have that $v$ is a weak supersolution of $\L v=0$ in $B_j$, which satisfies $v\geq 0$ in $B_j$, $\sup_{B_i}v\leq 4\omega_i$ for any $i\in\{0,\ldots,j\}$, and
    \begin{align*}
        \mu(\{v\geq\omega_j\}\cap 2B_{j+1})&\geq\frac{1}{2}\mu(2B_{j+1}).
    \end{align*}
    In fact, since $h(x,y,\cdot)$ is not odd, $v$ is not necessarily a weak supersolution of $\L v=0$ if \eqref{eqHoelder:ProofMainPoss2} is in force. However, if we replace $h(x,y,t)$ by $\tilde{h}(x,y,t)=-h(x,y,-t)$, then $v$ does satisfy $\L v=0$ weakly. Since $\tilde{h}$ satisfies the exact same structural assumptions as $h$, all of the following estimates remain true. Therefore, it follows from Lemma \ref{lemEoP:ExpansionOfPositivity} that
    \begin{align}\label{eqHoelder:ProofMainFirstAlternative}
        r_{j+1}^s\tail{v_-}{x_0}{r_j}>f\left(\frac{\delta\omega_j}{r_{j+1}^s}\right)
    \end{align}
    or
    \begin{align}\label{eqHoelder:ProofMainSecondAlternative}
        \inf_{B_{j+1}}v\geq\delta\omega_j
    \end{align}
    must be true. Now, we want to choose $\sigma$ appropriately small such that \eqref{eqHoelder:ProofMainFirstAlternative} is impossible. Let us compute the tail of $v_-$ using the induction hypothesis \eqref{eqHoelder:ProofMainInductionHypothesis}. We can split the integration domain into annuli, i.e.
    \begin{align*}
        \tail{v_-}{x_0}{r_j}\leq\sup_{x\in B_j/2}\left(\sum_{i=0}^{j-1}\int_{B_i\setminus B_{i+1}}f\left(\frac{|v(y)|}{d(x,y)^s}\right)\frac{\nu(x,\dy)}{d(x,y)^s}+\int_{X\setminus B_0}f\left(\frac{|v(y)|}{d(x,y)^s}\right)\frac{\nu(x,\dy)}{d(x,y)^s}\right)
    \end{align*}
    and estimate each term separately. Fix $x\in B_j/2$. For $i\in\{0,\ldots,j-1\}$, we have
    \begin{align*}
        \int_{B_i\setminus B_{i+1}}f\left(\frac{|v(y)|}{d(x,y)^s}\right)\frac{\nu(x,\dy)}{d(x,y)^s}&\leq\int_{X\setminus B_{r_{i+1}/2}(x)}f\left(\frac{4\omega_i}{r_{i+1}^s}\left(\frac{r_{i+1}}{d(x,y)}\right)^s\right)\frac{\nu(x,\dy)}{d(x,y)^s} \\
        &\leq Cr_{i+1}^{s(p-1)}f\left(\frac{\omega_i}{r_{i+1}^s}\right)\int_{X\setminus B_{r_{i+1}/2}(x)}\frac{\nu(x,\dy)}{d(x,y)^{sp}}\leq\frac{C}{r_{i+1}^s}f\left(\frac{\omega_i}{r_{i+1}^s}\right).    
    \end{align*}
    Here we have used that for $y\in X\setminus B_{i+1}$, we have
    \begin{align*}
        d(x,y)\geq d(x_0,y)-d(x_0,x)\geq r_{i+1}-\frac{r_j}{2}\geq r_{i+1}-\frac{r_{i+1}}{2}=\frac{r_{i+1}}{2}.
    \end{align*}
    Similarly, since
    \begin{align*}
        |v|\leq|u|+\sup_{B_j}|u|+2\omega_j\leq |u|+3\omega_0 \and \Tail{u}{x_0}{R}\leq\omega_0,
    \end{align*}
    Lemma \ref{lemPre:SupEstimate} implies
    \begin{align*}
        \int_{X\setminus B_0}f\left(\frac{|v(y)|}{d(x,y)^s}\right)\frac{\nu(x,\dy)}{d(x,y)^s}&\leq C\int_{X\setminus B_R(x_0)}f\left(\frac{|u(y)|}{d(x,y)^s}\right)\frac{\nu(x,\dy)}{d(x,y)^s} \\
        &\quad +C\int_{X\setminus B_{R/2}(x)}f\left(\frac{\omega_0}{R^s}\left(\frac{R}{d(x,y)}\right)^s\right)\frac{\nu(x,\dy)}{d(x,y)^s} \\
        &\leq C\tail{u}{x_0}{R}+CR^{s(p-1)}f\left(\frac{\omega_0}{R^s}\right)\int_{X\setminus B_{R/2}(x)}\frac{\nu(x,\dy)}{d(x,y)^{sp}} \\
        &\leq CR^{-s}f\left(R^{-s}\Tail{u}{x_0}{R}\right)+\frac{C}{R^s}f\left(\frac{\omega_0}{R^s}\right)\leq\frac{C}{R^s}f\left(\frac{\omega_0}{R^s}\right).
    \end{align*}
    Since $R\geq r_1$ and $\alpha\leq s$, it follows that
    \begin{align*}
        r_{j+1}^s\tail{v_-}{x_0}{r_j}&\leq C\sum_{i=0}^{j-1}\frac{r_{j+1}^s}{r_{i+1}^s}f\left(\frac{\omega_i}{r_{i+1}^s}\right)=C\sum_{i=0}^{j-1}\sigma^{(j-i)s}f\left(\sigma^{(j-i)(s-\alpha)}\frac{\omega_j}{r_{j+1}^s}\right) \\
        &\leq C\sum_{i=0}^{j-1}\sigma^{(j-i)(sp-\alpha(p-1))}f\left(\frac{\omega_j}{r_{j+1}^s}\right)\leq\overline{C}\frac{\sigma^{sp-\alpha(p-1)}}{1-\sigma^{sp-\alpha(p-1)}}\delta^{1-q}f\left(\frac{\delta\omega_j}{r_{j+1}^s}\right),
    \end{align*} 
    where $\delta=\delta(\C,s,t,p,q,\kappa,\lambda,C_1,C_2)\in(0,1)$ is the constant from Lemma \ref{lemEoP:ExpansionOfPositivity} for $\varepsilon=1/2$ and $\overline{C}=\overline{C}(s,p,q,C_2)>0$. Since $sp-\alpha(p-1)\geq sp/2$ and $\sigma\leq 1/8$, we have
    \begin{align*}
        \frac{\sigma^{sp-\alpha(p-1)}}{1-\sigma^{sp-\alpha(p-1)}}\leq\frac{\sigma^{sp-\alpha(p-1)}}{1-8^{\alpha(p-1)-sp}}\leq\frac{\sigma^{sp/2}}{1-8^{-sp}}.
    \end{align*}
    In order to ensure that \eqref{eqHoelder:ProofMainFirstAlternative} is false, it suffices to have
    \begin{align*}
        \overline{C}\frac{\sigma^{sp/2}}{1-8^{-sp}}\leq\delta^{q-1},
    \end{align*}
    i.e.\ we choose
    \begin{align*}
        \sigma=\min\left\{\left(\frac{\delta^{q-1}}{\overline{C}}(1-8^{-sp})\right)^\frac{2}{sp},\frac{1}{8\lambda}\right\}=\sigma(\C,s,t,p,q,\kappa,\lambda,C_1,C_2)\in\left(0,\frac{1}{8\lambda}\right].
    \end{align*}
    Lemma \ref{lemEoP:ExpansionOfPositivity} now yields that \eqref{eqHoelder:ProofMainSecondAlternative} is indeed true. Recalling the definition of $v$, this implies
    \begin{align*}
        \osc_{B_{j+1}}u\leq 2\left(1-\frac{\delta}{2}\right)\sigma^{-\alpha}\omega_{j+1}
    \end{align*}
    in any case. Finally, it remains to choose
    \begin{align*}
        \alpha=\min\left\{\frac{\log\left(1-\frac{\delta}{2}\right)}{\log\sigma},\frac{sp}{2(p-1)},s\right\}=\alpha(\C,s,t,p,q,\kappa,\lambda,C_1,C_2)\in(0,1)
    \end{align*}
    because then
    \begin{align*}
        \osc_{B_{j+1}}u\leq 2\left(1-\frac{\delta}{2}\right)\sigma^{-\alpha}\omega_{j+1}\leq 2\omega_{j+1}
    \end{align*}
    and the induction is complete.
\end{proof}

For convenience, we derive Theorem \ref{thmPre:Hoelder} explicitly from Lemma \ref{lemHoelder:OscillationDiminish} because on an abstract metric measure space it might be unclear how one extracts the continuous version of the solution and handles the tail term appropriately.

\begin{proof}[Proof of Theorem \ref{thmPre:Hoelder}]
    For any $x\in B_{R/4}(x_0)$ and $r>0$, we can define
    \begin{align*}
        M(x,r)=\sup_{B_r(x)}u \and m(x,r)=\inf_{B_r(x)}u
    \end{align*}
    and also
    \begin{align*}
         M(x)=\lim_{r\to 0}M(x,r) \and m(x)=\lim_{r\to 0}m(x,r)
    \end{align*}
    as their monotone limits. For $0<r\leq R/2$, we may apply Lemma \ref{lemHoelder:OscillationDiminish} in $B_{R/2}(x)$ to see that
    \begin{align}
        M(x,r)-m(x,r)=\osc_{B_r(x)}u&\leq C\left(\frac{r}{R}\right)^\alpha\left(\|u\|_{L^\infty(B_{R/2}(x))}+\Tail{u}{x}{R/2}\right) \nonumber \\
        &\leq C\left(\frac{r}{R}\right)^\alpha\left(\|u\|_{L^\infty(B_R(x_0))}+\Tail{u}{x_0}{R}\right) \label{eqHoelder:OscillationDecay}
    \end{align}
    by Lemma \ref{lemPre:TailMonotonicity}. Letting $r\to 0$, we obtain $M(x)=m(x)$. We define the function $\tilde{u}$ by
    \begin{align*}
        \tilde{u}(x)=\begin{cases}
            M(x),\ &x\in B_{R/4}(x_0), \\
            u(x),\ &x\in X\setminus B_{R/4}(x_0)
        \end{cases}
    \end{align*}
    and claim that $\tilde{u}$ is a version of $u$ that is locally Hölder continuous in $B_{R/4}(x_0)$. The first claim follows by Lemma \ref{lemPre:LebesgueDifferentiation}, i.e.
    \begin{align*}
        u(x)=\lim_{r\to 0}\fint_{B_r(x)}u(y)\mu(\dy)\leq\lim_{r\to 0}\sup_{B_r(x)}u=M(x)=\tilde{u}(x)
    \end{align*}
    and similarly, $u(x)\geq m(x)=\tilde{u}(x)$ for $\mu$-almost every $x\in B_{R/4}(x_0)$. Let us rename $\tilde{u}$ back to $u$ and prove the Hölder continuity assertion. Fix $x,y\in B_{R/4}(x_0)$ and assume $r\coloneqq d(x,y)<\frac{R}{4\lambda}$. Then, by \eqref{eqHoelder:OscillationDecay},
    \begin{align*}
        \frac{|u(x)-u(y)|}{d(x,y)^\alpha}\leq\frac{\osc_{B_{2r}(x)}}{r^\alpha}\leq\frac{C}{R^\alpha}\left(\|u\|_{L^\infty(B_R(x_0))}+\Tail{u}{x_0}{R}\right)
    \end{align*} 
    as desired. On the other hand, if $r\geq\frac{R}{4\lambda}$, then
    \begin{align*}
        \frac{|u(x)-u(y)|}{d(x,y)^\alpha}\leq\frac{2\|u\|_{L^\infty(B_{R/4}(x_0))}}{\left(\frac{R}{4\lambda}\right)^\alpha}\leq\frac{C}{R^\alpha}\|u\|_{L^\infty(B_R(x_0))}.
    \end{align*}
    Hence, we found that
    \begin{align*}
        [u]_{C^{0,\alpha}(B_{R/4}(x_0))}\leq\frac{C}{R^\alpha}\left(\|u\|_{L^\infty(B_R(x_0))}+\Tail{u}{x_0}{R}\right).
    \end{align*}
    Since $\mu$ has full support and all of those local Hölder continuous representatives coincide $\mu$-almost everywhere, we get that $u\in C_\loc^{0,\alpha}(\Omega)$.
\end{proof}

\begin{proof}[Proof of Corollary \ref{corPre:Hoelder}]
    Let us first prove that any $u\in L^\infty(X)$ belongs to $L_{s,\nu}^f(X)$, that is, all tails are finite. Let $0<R\leq R_0$ and $x_0\in X$ be arbitrary. Then
    \begin{align*}
        \tail{u}{x_0}{R}&\leq\sup_{x\in B_{R/2}(x_0)}\int_{X\setminus B_R(x_0)}f\left(\frac{\|u\|_{L^\infty(X)}}{R^s}\left(\frac{R}{d(x,y)}\right)^s\right)\frac{\nu(x,\dy)}{d(x,y)^s} \nonumber \\
        &\leq Cf\left(\frac{\|u\|_{L^\infty(X)}}{R^s}\right)R^{s(p-1)}\sup_{x\in B_{R/2}(x_0)}\int_{X\setminus B_{R/2}(x)}\frac{\nu(x,\dy)}{d(x,y)^{sp}} \nonumber \\
        &\leq\frac{C}{R^s}f\left(\frac{\|u\|_{L^\infty(X)}}{R^s}\right)<\infty
    \end{align*}
    by \eqref{eqPre:TailIntegrabilityAssumption} and Lemma \ref{lemPre:SupEstimate}. The result will follow by similar and standard arguments as in the previous proof.
\end{proof}

\printbibliography[heading=bibintoc]

@article{BDNS25,
 author = {Behn, Linus and Diening, Lars and Nowak, Simon and Scharle, Toni},
 title = {The {De} {Giorgi} method for local and nonlocal systems},
 fjournal = {Journal of the London Mathematical Society. Second Series},
 journal = {J. Lond. Math. Soc., II. Ser.},
 issn = {0024-6107},
 volume = {112},
 number = {1},
 eid = {e70237},
 note = {Id/No e70237},
 year = {2025},
 language = {English},
 doi = {10.1112/jlms.70237},
 zbMATH = {8072365},
 Zbl = {1572.35282}
}

@book{HKST15,
 author = {Heinonen, Juha and Koskela, Pekka and Shanmugalingam, Nageswari and Tyson, Jeremy T.},
 title = {Sobolev spaces on metric measure spaces. {An} approach based on upper gradients},
 fseries = {New Mathematical Monographs},
 series = {New Math. Monogr.},
 volume = {27},
 isbn = {978-1-107-09234-1; 978-1-316-13591-4},
 year = {2015},
 publisher = {Cambridge: Cambridge University Press},
 language = {English},
 doi = {10.1017/CBO9781316135914},
 zbMATH = {6397370},
 Zbl = {1332.46001}
}

@article{DKP16,
 author = {Di Castro, Agnese and Kuusi, Tuomo and Palatucci, Giampiero},
 title = {Local behavior of fractional {$p$}-minimizers},
 fjournal = {Annales de l'Institut Henri Poincar{\'e}. Analyse Non Lin{\'e}aire},
 journal = {Ann. Inst. Henri Poincar{\'e}, Anal. Non Lin{\'e}aire},
 issn = {0294-1449},
 volume = {33},
 number = {5},
 pages = {1279--1299},
 year = {2016},
 language = {English},
 doi = {10.1016/j.anihpc.2015.04.003},
 zbMATH = {6640998},
 Zbl = {1355.35192}
}

@article{BKO23,
 author = {Byun, Sun-Sig and Kim, Hyojin and Ok, Jihoon},
 title = {Local {Hölder} continuity for fractional nonlocal equations with general growth},
 fjournal = {Mathematische Annalen},
 journal = {Math. Ann.},
 issn = {0025-5831},
 volume = {387},
 number = {1-2},
 pages = {807--846},
 year = {2023},
 language = {English},
 doi = {10.1007/s00208-022-02472-y},
 zbMATH = {7735164},
 Zbl = {1522.35543}
}

@article{CKW22,
 author = {Chaker, Jamil and Kim, Minhyun and Weidner, Marvin},
 title = {Regularity for nonlocal problems with non-standard growth},
 fjournal = {Calculus of Variations and Partial Differential Equations},
 journal = {Calc. Var. Partial Differ. Equ.},
 issn = {0944-2669},
 volume = {61},
 number = {6},
 eid = {227},
 note = {Id/No 227},
 year = {2022},
 language = {English},
 doi = {10.1007/s00526-022-02364-8},
 zbMATH = {7608958},
 Zbl = {1501.35106}
}

@article{CKW23,
 author = {Chaker, Jamil and Kim, Minhyun and Weidner, Marvin},
 title = {Harnack inequality for nonlocal problems with non-standard growth},
 fjournal = {Mathematische Annalen},
 journal = {Math. Ann.},
 issn = {0025-5831},
 volume = {386},
 number = {1-2},
 pages = {533--550},
 year = {2023},
 language = {English},
 doi = {10.1007/s00208-022-02405-9},
 zbMATH = {7686089},
 Zbl = {1515.35082}
}

@article{Che25,
 author = {Chen, Jingya},
 title = {Hölder and {Harnack} estimates for integro-differential operators with kernels of measure},
 fjournal = {Annali di Matematica Pura ed Applicata. Serie Quarta},
 journal = {Ann. Mat. Pura Appl. (4)},
 issn = {0373-3114},
 volume = {204},
 number = {4},
 pages = {1729--1766},
 year = {2025},
 language = {English},
 doi = {10.1007/s10231-025-01546-3},
 zbMATH = {8071166},
 Zbl = {1570.35085}
}

@article{DK20,
 author = {Dyda, Bart{\l}omiej and Kassmann, Moritz},
 title = {Regularity estimates for elliptic nonlocal operators},
 fjournal = {Analysis \& PDE},
 journal = {Anal. PDE},
 issn = {2157-5045},
 volume = {13},
 number = {2},
 pages = {317--370},
 year = {2020},
 language = {English},
 doi = {10.2140/apde.2020.13.317},
 zbMATH = {7181503},
 Zbl = {1437.35175}
}

@misc{Gio57,
 author = {De Giorgi, Ennio},
 title = {Sulla differenziabilit{\`a} e l'analiticit{\`a} delle estremali degli integrali multipli regolari},
 year = {1957},
 language = {Italian},
 howpublished = {Mem. {Accad}. {Sci}. {Torino}, {P}. {I}., {III}. {Ser}. 3, 25-43.},
 zbMATH = {3138423},
 Zbl = {0084.31901}
}

@article{CS20,
 author = {Chaker, Jamil and Silvestre, Luis},
 title = {Coercivity estimates for integro-differential operators},
 fjournal = {Calculus of Variations and Partial Differential Equations},
 journal = {Calc. Var. Partial Differ. Equ.},
 issn = {0944-2669},
 volume = {59},
 number = {4},
 eid = {106},
 note = {Id/No 106},
 year = {2020},
 language = {English},
 doi = {10.1007/s00526-020-01764-y},
 zbMATH = {7212284},
 Zbl = {1444.47058}
}

@article{HKKT15,
 author = {Heikkinen, Toni and Kinnunen, Juha and Korvenpää, Janne and Tuominen, Heli},
 title = {Regularity of the local fractional maximal function},
 fjournal = {Arkiv för Matematik},
 journal = {Ark. Mat.},
 issn = {0004-2080},
 volume = {53},
 number = {1},
 pages = {127--154},
 year = {2015},
 language = {English},
 doi = {10.1007/s11512-014-0199-2},
 zbMATH = {6424056},
 Zbl = {1316.42019}
}

@book{Giu03,
 author = {Giusti, Enrico},
 title = {Direct methods in the calculus of variations},
 isbn = {981-238-043-4},
 year = {2003},
 publisher = {Singapore: World Scientific},
 language = {English},
 zbMATH = {1815381},
 Zbl = {1028.49001}
}

@book{Hei01,
 author = {Heinonen, Juha},
 title = {Lectures on analysis on metric spaces},
 fseries = {Universitext},
 series = {Universitext},
 issn = {0172-5939},
 isbn = {0-387-95104-0},
 year = {2001},
 publisher = {New York, NY: Springer},
 language = {English},
 doi = {10.1007/978-1-4613-0131-8},
 zbMATH = {1535638},
 Zbl = {0985.46008}
}

@article{Kas09,
 author = {Kassmann, Moritz},
 title = {A priori estimates for integro-differential operators with measurable kernels},
 fjournal = {Calculus of Variations and Partial Differential Equations},
 journal = {Calc. Var. Partial Differ. Equ.},
 issn = {0944-2669},
 volume = {34},
 number = {1},
 pages = {1--21},
 year = {2009},
 language = {English},
 doi = {10.1007/s00526-008-0173-6},
 zbMATH = {5373494},
 Zbl = {1158.35019}
}

@article{DKP14,
 author = {Di Castro, Agnese and Kuusi, Tuomo and Palatucci, Giampiero},
 title = {Nonlocal {Harnack} inequalities},
 fjournal = {Journal of Functional Analysis},
 journal = {J. Funct. Anal.},
 issn = {0022-1236},
 volume = {267},
 number = {6},
 pages = {1807--1836},
 year = {2014},
 language = {English},
 doi = {10.1016/j.jfa.2014.05.023},
 zbMATH = {6330975},
 Zbl = {1302.35082}
}

@article{Now21,
 author = {Nowak, Simon},
 title = {Higher {Hölder} regularity for nonlocal equations with irregular kernel},
 fjournal = {Calculus of Variations and Partial Differential Equations},
 journal = {Calc. Var. Partial Differ. Equ.},
 issn = {0944-2669},
 volume = {60},
 number = {1},
 eid = {24},
 note = {Id/No 24},
 year = {2021},
 language = {English},
 doi = {10.1007/s00526-020-01915-1},
 zbMATH = {7309168},
 Zbl = {1509.35087}
}

@article{BLS18,
 author = {Brasco, Lorenzo and Lindgren, Erik and Schikorra, Armin},
 title = {Higher {Hölder} regularity for the fractional {$p$}-{Laplacian} in the superquadratic case},
 fjournal = {Advances in Mathematics},
 journal = {Adv. Math.},
 issn = {0001-8708},
 volume = {338},
 pages = {782--846},
 year = {2018},
 language = {English},
 doi = {10.1016/j.aim.2018.09.009},
 zbMATH = {6950284},
 Zbl = {1400.35049}
}

@article{GL24,
 author = {Garain, Prashanta and Lindgren, Erik},
 title = {Higher {Hölder} regularity for the fractional {$p$}-{Laplace} equation in the subquadratic case},
 fjournal = {Mathematische Annalen},
 journal = {Math. Ann.},
 issn = {0025-5831},
 volume = {390},
 number = {4},
 pages = {5753--5792},
 year = {2024},
 language = {English},
 doi = {10.1007/s00208-024-02891-z},
 zbMATH = {7940270},
 Zbl = {1555.35079}
}

@article{DP19,
 author = {De Filippis, Cristiana and Palatucci, Giampiero},
 title = {Hölder regularity for nonlocal double phase equations},
 fjournal = {Journal of Differential Equations},
 journal = {J. Differ. Equations},
 issn = {0022-0396},
 volume = {267},
 number = {1},
 pages = {547--586},
 year = {2019},
 language = {English},
 doi = {10.1016/j.jde.2019.01.017},
 zbMATH = {7046606},
 Zbl = {1412.35041}
}

@article{BOS22,
 author = {Byun, Sun-Sig and Ok, Jihoon and Song, Kyeong},
 title = {Hölder regularity for weak solutions to nonlocal double phase problems},
 fjournal = {Journal de Math{\'e}matiques Pures et Appliqu{\'e}es. Neuvi{\`e}me S{\'e}rie},
 journal = {J. Math. Pures Appl. (9)},
 issn = {0021-7824},
 volume = {168},
 pages = {110--142},
 year = {2022},
 language = {English},
 doi = {10.1016/j.matpur.2022.11.001},
 zbMATH = {7625580},
 Zbl = {1504.35104}
}

@article{CK23,
 author = {Chaker, Jamil and Kim, Minhyun},
 title = {Local regularity for {Nonlocal} equations with variable exponents},
 fjournal = {Mathematische Nachrichten},
 journal = {Math. Nachr.},
 issn = {0025-584X},
 volume = {296},
 number = {9},
 pages = {4463--4489},
 year = {2023},
 language = {English},
 doi = {10.1002/mana.202100521},
 zbMATH = {7750723},
 Zbl = {1525.35056}
}

@article{Ok23,
 author = {Ok, Jihoon},
 title = {Local {Hölder} regularity for nonlocal equations with variable powers},
 fjournal = {Calculus of Variations and Partial Differential Equations},
 journal = {Calc. Var. Partial Differ. Equ.},
 issn = {0944-2669},
 volume = {62},
 number = {1},
 eid = {32},
 note = {Id/No 32},
 year = {2023},
 language = {English},
 doi = {10.1007/s00526-022-02353-x},
 zbMATH = {7620718},
 Zbl = {1507.35330}
}

@article{FSV22,
 author = {Fern{\'a}ndez Bonder, Juli{\'a}n and Salort, Ariel and Vivas, Hern{\'a}n},
 title = {Interior and up to the boundary regularity for the fractional {$g$}-{Laplacian}: {The} convex case},
 fjournal = {Nonlinear Analysis. Theory, Methods \& Applications. Series A: Theory and Methods},
 journal = {Nonlinear Anal., Theory Methods Appl., Ser. A, Theory Methods},
 issn = {0362-546X},
 volume = {223},
 eid = {11360},
 note = {Id/No 113060},
 year = {2022},
 language = {English},
 doi = {10.1016/j.na.2022.113060},
 zbMATH = {7564617},
 Zbl = {1497.35274}
}

@article{FS19,
 author = {Fern{\'a}ndez Bonder, Juli{\'a}n and Salort, Ariel M.},
 title = {Fractional order {Orlicz}-{Sobolev} spaces},
 fjournal = {Journal of Functional Analysis},
 journal = {J. Funct. Anal.},
 issn = {0022-1236},
 volume = {277},
 number = {2},
 pages = {333--367},
 year = {2019},
 language = {English},
 doi = {10.1016/j.jfa.2019.04.003},
 zbMATH = {7063142},
 Zbl = {1426.46018}
}

@misc{GJS25,
 author = {Davide Giovagnoli and David Jesus and Luis Silvestre},
 title = {{$C^{1+\alpha}$} regularity for fractional {$p$}-harmonic functions},
 year = {2025},
 eprint = {2509.26565},
 eprinttype = {arXiv},
 eprintclass = {math.AP},
 url = {https://arxiv.org/abs/2509.26565},
 arXiv = {arXiv:2509.26565}
}

@article{BT25,
 author = {Biswas, Anup and Topp, Erwin},
 title = {Lipschitz regularity of fractional {$p$}-{Laplacian}},
 fjournal = {Annals of PDE},
 journal = {Ann. PDE},
 issn = {2524-5317},
 volume = {11},
 number = {2},
 eid = {27},
 note = {Id/No 27},
 year = {2025},
 language = {English},
 doi = {10.1007/s40818-025-00220-4},
 zbMATH = {8105425},
 Zbl = {1575.35034}
}

@article{BL02,
 author = {Bass, Richard F. and Levin, David A.},
 title = {Transition probabilities for symmetric jump processes},
 fjournal = {Transactions of the American Mathematical Society},
 journal = {Trans. Am. Math. Soc.},
 issn = {0002-9947},
 volume = {354},
 number = {7},
 pages = {2933--2953},
 year = {2002},
 language = {English},
 doi = {10.1090/S0002-9947-02-02998-7},
 zbMATH = {1725362},
 Zbl = {0993.60070}
}

@article{CK03,
 author = {Chen, Zhen-Qing and Kumagai, Takashi},
 title = {Heat kernel estimates for stable-like processes on {$d$}-sets.},
 fjournal = {Stochastic Processes and their Applications},
 journal = {Stochastic Processes Appl.},
 issn = {0304-4149},
 volume = {108},
 number = {1},
 pages = {27--62},
 year = {2003},
 language = {English},
 doi = {10.1016/S0304-4149(03)00105-4},
 zbMATH = {2233794},
 Zbl = {1075.60556}
}

@article{BGK09,
 author = {Barlow, Martin T. and Grigor'yan, Alexander and Kumagai, Takashi},
 title = {Heat kernel upper bounds for jump processes and the first exit time},
 fjournal = {Journal f{\"u}r die Reine und Angewandte Mathematik},
 journal = {J. Reine Angew. Math.},
 issn = {0075-4102},
 volume = {626},
 pages = {135--157},
 year = {2009},
 language = {English},
 doi = {10.1515/CRELLE.2009.005},
 zbMATH = {5505183},
 Zbl = {1158.60039}
}

@article{GH08,
 author = {Grigor'yan, Alexander and Hu, Jiaxin},
 title = {Off-diagonal upper estimates for the heat kernel of the {Dirichlet} forms on metric spaces},
 fjournal = {Inventiones Mathematicae},
 journal = {Invent. Math.},
 issn = {0020-9910},
 volume = {174},
 number = {1},
 pages = {81--126},
 year = {2008},
 language = {English},
 doi = {10.1007/s00222-008-0135-9},
 zbMATH = {5343983},
 Zbl = {1154.47034}
}

@article{CKKW18,
 author = {Chen, Zhen-Qing and Kim, Panki and Kumagai, Takashi and Wang, Jian},
 title = {Heat kernel estimates for time fractional equations},
 fjournal = {Forum Mathematicum},
 journal = {Forum Math.},
 issn = {0933-7741},
 volume = {30},
 number = {5},
 pages = {1163--1192},
 year = {2018},
 language = {English},
 doi = {10.1515/forum-2017-0192},
 zbMATH = {6947017},
 Zbl = {1401.60088}
}

@article{GHL14,
 author = {Grigor'yan, Alexander and Hu, Jiaxin and Lau, Ka-Sing},
 title = {Estimates of heat kernels for non-local regular {Dirichlet} forms},
 fjournal = {Transactions of the American Mathematical Society},
 journal = {Trans. Am. Math. Soc.},
 issn = {0002-9947},
 volume = {366},
 number = {12},
 pages = {6397--6441},
 year = {2014},
 language = {English},
 doi = {10.1090/S0002-9947-2014-06034-0},
 zbMATH = {6371211},
 Zbl = {1304.47058}
}

@article{Coz17,
 author = {Cozzi, Matteo},
 title = {Regularity results and {Harnack} inequalities for minimizers and solutions of nonlocal problems: a unified approach via fractional {De} {Giorgi} classes},
 fjournal = {Journal of Functional Analysis},
 journal = {J. Funct. Anal.},
 issn = {0022-1236},
 volume = {272},
 number = {11},
 pages = {4762--4837},
 year = {2017},
 language = {English},
 doi = {10.1016/j.jfa.2017.02.016},
 zbMATH = {6710259},
 Zbl = {1366.49040}
}

@article{ACPS21,
 author = {Alberico, Angela and Cianchi, Andrea and Pick, Lubo{\v{s}} and Slav{\'{\i}}kov{\'a}, Lenka},
 title = {On fractional {Orlicz}-{Sobolev} spaces},
 fjournal = {Analysis and Mathematical Physics},
 journal = {Anal. Math. Phys.},
 issn = {1664-2368},
 volume = {11},
 number = {2},
 eid = {84},
 note = {Id/No 84},
 year = {2021},
 language = {English},
 doi = {10.1007/s13324-021-00511-6},
 zbMATH = {7355860},
 Zbl = {1477.46039}
}

@book{HH19,
 author = {Harjulehto, Petteri and Hästö, Peter},
 title = {Orlicz spaces and generalized {Orlicz} spaces},
 fseries = {Lecture Notes in Mathematics},
 series = {Lect. Notes Math.},
 issn = {0075-8434},
 volume = {2236},
 isbn = {978-3-030-15099-0; 978-3-030-15100-3},
 year = {2019},
 publisher = {Cham: Springer},
 language = {English},
 doi = {10.1007/978-3-030-15100-3},
 zbMATH = {7045570},
 Zbl = {1436.46002}
}

@book{LU68,
 author = {Ladyzhenskaya, O. A. and Ural'tseva, N. N.},
 title = {Linear and quasilinear elliptic equations},
 fseries = {Mathematics in Science and Engineering},
 series = {Math. Sci. Eng.},
 volume = {46},
 year = {1968},
 publisher = {Elsevier, Amsterdam},
 language = {English},
 zbMATH = {3262653},
 Zbl = {0164.13002}
}

@article{GG82,
 author = {Giaquinta, Mariano and Giusti, Enrico},
 title = {On the regularity of the minima of variational integrals},
 fjournal = {Acta Mathematica},
 journal = {Acta Math.},
 issn = {0001-5962},
 volume = {148},
 pages = {31--46},
 year = {1982},
 language = {English},
 doi = {10.1007/BF02392725},
 zbMATH = {3778229},
 Zbl = {0494.49031}
}

@article{Lie91,
 author = {Lieberman, Gary M.},
 title = {The natural generalization of the natural conditions of {Ladyzhenskaya} and {Ural}'tseva for elliptic equations},
 fjournal = {Communications in Partial Differential Equations},
 journal = {Commun. Partial Differ. Equations},
 issn = {0360-5302},
 volume = {16},
 number = {2-3},
 pages = {311--361},
 year = {1991},
 language = {English},
 doi = {10.1080/03605309108820761},
 zbMATH = {10065},
 Zbl = {0742.35028}
}

@article{Mar89,
 author = {Marcellini, Paolo},
 title = {Regularity of minimizers of integrals of the calculus of variations with non-standard growth conditions},
 fjournal = {Archive for Rational Mechanics and Analysis},
 journal = {Arch. Ration. Mech. Anal.},
 issn = {0003-9527},
 volume = {105},
 number = {3},
 pages = {267--284},
 year = {1989},
 language = {English},
 doi = {10.1007/BF00251503},
 zbMATH = {4092224},
 Zbl = {0667.49032}
}

@article{KS01,
 author = {Kinnunen, Juha and Shanmugalingam, Nageswari},
 title = {Regularity of quasi-minimizers on metric spaces},
 fjournal = {Manuscripta Mathematica},
 journal = {Manuscr. Math.},
 issn = {0025-2611},
 volume = {105},
 number = {3},
 pages = {401--423},
 year = {2001},
 language = {English},
 doi = {10.1007/s002290100193},
 zbMATH = {1690802},
 Zbl = {1006.49027}
}

@article{CM15,
 author = {Colombo, Maria and Mingione, Giuseppe},
 title = {Regularity for double phase variational problems},
 fjournal = {Archive for Rational Mechanics and Analysis},
 journal = {Arch. Ration. Mech. Anal.},
 issn = {0003-9527},
 volume = {215},
 number = {2},
 pages = {443--496},
 year = {2015},
 language = {English},
 doi = {10.1007/s00205-014-0785-2},
 url = {infoscience.epfl.ch/handle/20.500.14299/165511},
 zbMATH = {6399641},
 Zbl = {1322.49065}
}

@article{Ok20,
 author = {Ok, Jihoon},
 title = {Regularity for double phase problems under additional integrability assumptions},
 fjournal = {Nonlinear Analysis. Theory, Methods \& Applications. Series A: Theory and Methods},
 journal = {Nonlinear Anal., Theory Methods Appl., Ser. A, Theory Methods},
 issn = {0362-546X},
 volume = {194},
 eid = {111408},
 note = {Id/No 111408},
 year = {2020},
 language = {English},
 doi = {10.1016/j.na.2018.12.019},
 zbMATH = {7184500},
 Zbl = {1437.49052}
}

@article{MN91,
 author = {Moscariello, Gioconda and Nania, Luciana},
 title = {Hölder continuity of minimizers of functionals with non standard growth conditions},
 fjournal = {Ricerche di Matematica},
 journal = {Ric. Mat.},
 issn = {0035-5038},
 volume = {40},
 number = {2},
 pages = {259--273},
 year = {1991},
 language = {English},
 zbMATH = {96522},
 Zbl = {0773.49019}
}

@article{AM01,
 author = {Acerbi, Emilio and Mingione, Giuseppe},
 title = {Regularity results for a class of functionals with non-standard growth},
 fjournal = {Archive for Rational Mechanics and Analysis},
 journal = {Arch. Ration. Mech. Anal.},
 issn = {0003-9527},
 volume = {156},
 number = {2},
 pages = {121--140},
 year = {2001},
 language = {English},
 doi = {10.1007/s002050100117},
 zbMATH = {1586360},
 Zbl = {0984.49020}
}

@article{RS16,
 author = {Ros-Oton, Xavier and Serra, Joaquim},
 title = {Regularity theory for general stable operators},
 fjournal = {Journal of Differential Equations},
 journal = {J. Differ. Equations},
 issn = {0022-0396},
 volume = {260},
 number = {12},
 pages = {8675--8715},
 year = {2016},
 language = {English},
 doi = {10.1016/j.jde.2016.02.033},
 zbMATH = {6568053},
 Zbl = {1346.35220}
}

@article{BKS19,
 author = {Bux, Kai-Uwe and Kassmann, Moritz and Schulze, Tim},
 title = {Quadratic forms and {Sobolev} spaces of fractional order},
 fjournal = {Proceedings of the London Mathematical Society. Third Series},
 journal = {Proc. Lond. Math. Soc. (3)},
 issn = {0024-6115},
 volume = {119},
 number = {3},
 pages = {841--866},
 year = {2019},
 language = {English},
 doi = {10.1112/plms.12246},
 zbMATH = {7119250}
}

@article{IS20,
 author = {Imbert, Cyril and Silvestre, Luis},
 title = {The weak {Harnack} inequality for the {Boltzmann} equation without cut-off},
 fjournal = {Journal of the European Mathematical Society (JEMS)},
 journal = {J. Eur. Math. Soc. (JEMS)},
 issn = {1435-9855},
 volume = {22},
 number = {2},
 pages = {507--592},
 year = {2020},
 language = {English},
 doi = {10.4171/JEMS/928},
 zbMATH = {7174139},
 Zbl = {1473.35077}
}

@phdthesis{Hep26,
  author = {Hepp, Solveig},
  school = {Bielefeld University},
  title = {Quadratic forms with space-dependent singular measures},
  year = {2026}
}

@article{BS05,
 author = {Bogdan, Krzysztof and Sztonyk, Pawe{\l}},
 title = {Harnack's inequality for stable {L{\'e}vy} processes},
 fjournal = {Potential Analysis},
 journal = {Potential Anal.},
 issn = {0926-2601},
 volume = {22},
 number = {2},
 pages = {133--150},
 year = {2005},
 language = {English},
 doi = {10.1007/s11118-004-0590-x},
 zbMATH = {2162535},
 Zbl = {1081.60055}
}

@article{CKKSS25,
 author = {Capogna, Luca and Kline, Josh and Korte, Riikka and Shanmugalingam, Nageswari and Snipes, Marie},
 title = {Neumann problems for {{\(p\)}}-harmonic functions, and induced nonlocal operators in metric measure spaces},
 fjournal = {American Journal of Mathematics},
 journal = {Am. J. Math.},
 issn = {0002-9327},
 volume = {147},
 number = {6},
 pages = {1653--1711},
 year = {2025},
 language = {English},
 url = {muse.jhu.edu/pub/1/article/975705/summary},
 zbMATH = {8131998}
}

@misc{CGKS26,
 author = {Luca Capogna and Ryan Gibara and Riikka Korte and Nageswari Shanmugalingam},
 title = {Fractional {$p$}-{Laplacians} via {Neumann} problems in unbounded metric measure spaces},
 year = {2026},
 url = {https://arxiv.org/abs/2410.18883},
 eprint = {2410.18883},
 eprinttype = {arXiv},
 eprintclass = {math.AP},
}

@article{CS07,
 author = {Caffarelli, Luis and Silvestre, Luis},
 title = {An extension problem related to the fractional {Laplacian}},
 fjournal = {Communications in Partial Differential Equations},
 journal = {Commun. Partial Differ. Equations},
 issn = {0360-5302},
 volume = {32},
 number = {8},
 pages = {1245--1260},
 year = {2007},
 language = {English},
 doi = {10.1080/03605300600987306},
 zbMATH = {5204436},
 Zbl = {1143.26002}
}

@misc{Cho24,
 author = {Soobin Cho},
 title = {Robust estimates for elliptic nonlocal operators on doubling spaces},
 year = {2024},
 url = {https://arxiv.org/abs/2403.16853},
 eprint = {2403.16853.4242},
 eprinttype = {arXiv},
 eprintclass = {math.AP},
}

@book{ADM11,
 author = {Ambrosio, Luigi and Da Prato, Giuseppe and Mennucci, Andrea},
 title = {Introduction to measure theory and integration},
 fseries = {Appunti. Scuola Normale Superiore di Pisa (Nuova Serie)},
 series = {Appunti, Sc. Norm. Super. Pisa (N.S.)},
 issn = {2532-991X},
 volume = {10},
 isbn = {978-88-7642-385-7},
 year = {2011},
 publisher = {Pisa: Edizioni della Normale},
 language = {English},
 zbMATH = {5943675},
 Zbl = {1264.28001}
}

\end{document}